\documentclass[12pt]{article}
\pdfoutput=1

\usepackage{array}
\usepackage[affil-it]{authblk}
\usepackage{amsmath, amsthm, amssymb}
\usepackage{fullpage}
\usepackage{enumerate}
\usepackage{ifpdf}
\usepackage{mathrsfs}
\usepackage[normalem]{ulem}
\ifpdf
\usepackage[pdftex]{graphicx}
\else
\usepackage[dvips]{graphicx}
\fi
\usepackage{tikz}
\usepackage[all]{xy}
\usetikzlibrary{arrows,chains,matrix,positioning,scopes,decorations.pathreplacing,patterns}

\usepackage{tkz-graph}
\GraphInit[vstyle = Classic]
\tikzset{
  VertexStyle/.append style = {minimum size=5pt, inner sep=0pt,
                                font = \Large\bfseries},
  EdgeStyle/.append style = {} }

\makeatletter
\tikzset{join/.code=\tikzset{after node path={%
\ifx\tikzchainprevious\pgfutil@empty\else(\tikzchainprevious)%
edge[every join]#1(\tikzchaincurrent)\fi}}}
\makeatother
\tikzset{>=stealth',every on chain/.append style={join},
         every join/.style={->}}
\tikzstyle{labeled}=[execute at begin node=$\scriptstyle,
   execute at end node=$]

\input xy
\xyoption{all}

\usepackage[pdftex,plainpages=false,hypertexnames=false,pdfpagelabels]{hyperref}
\newcommand{\arXiv}[1]{\href{http://arxiv.org/abs/#1}{\tt arXiv:\nolinkurl{#1}}}
\newcommand{\arxiv}[1]{\href{http://arxiv.org/abs/#1}{\tt arXiv:\nolinkurl{#1}}}

\newcommand{\doi}[1]{\href{http://dx.doi.org/#1}{{\tt DOI:#1}}}

\newcommand{\googlebooks}[1]{(preview at \href{http://books.google.com/books?id=#1}{google books})}

\usepackage{xcolor}
\definecolor{dark-red}{rgb}{0.7,0.25,0.25}
\definecolor{dark-blue}{rgb}{0.15,0.15,0.55}
\definecolor{medium-blue}{rgb}{0,0,.8}
\definecolor{DarkGreen}{RGB}{0,150,0}
\hypersetup{
   colorlinks, linkcolor={purple},
   citecolor={medium-blue}, urlcolor={medium-blue}
}
\definecolor{lightred}{RGB}{255,99,99}
\definecolor{lightblue}{RGB}{99,99,255}

\usepackage{relsize}
	\newcommand{\bigast}{\mathop{\mathlarger{\mathlarger{\Asterisk}}}}

\theoremstyle{plain}
\newtheorem{thm}{Theorem}[section]
\newtheorem*{thm*}{Theorem}

\newtheorem{cor}[thm]{Corollary}
\newtheorem*{cor*}{Corollary}
\newtheorem{lem}[thm]{Lemma}
\newtheorem{claim}{Claim}
\newtheorem{prop}[thm]{Proposition}
\newtheorem*{prop*}{Proposition}

\newtheorem*{quest*}{Question}

\theoremstyle{definition}
\newtheorem{defn}[thm]{Definition}

\newtheorem{nota}[thm]{Notation}

\newtheorem{rem}[thm]{Remark}

\DeclareMathOperator{\Gr}{Gr}

\DeclareMathOperator{\op}{op}

\DeclareMathOperator{\Tr}{Tr}

\newcommand{\comment}[1]{}

\newcommand{\be}{\begin{enumerate}[(1)]}
\newcommand{\ee}{\end{enumerate}}

\newcommand{\N}{\mathbb{N}}
\newcommand{\Z}{\mathbb{Z}}

\newcommand{\F}{\mathbb{F}}
\newcommand{\R}{\mathbb{R}}

\newcommand{\C}{\mathbb{C}}

\newcommand{\e}{\epsilon}

\newcommand{\noshow}[1]{}
\newcommand{\MR}[1]{}

\newcommand{\Asterisk}{\mathop{\scalebox{1.5}{\raisebox{-0.2ex}{$*$}}}}%

\newcommand{\vphi}{\varphi}

\newcommand{\TL}{\cT\hspace{-.08cm}\cL}

\newcommand{\mc}[1]{\mathcal{#1}}
\newcommand{\<}{\left\langle}
\renewcommand{\>}{\right\rangle}

\def\semicolon{;}
\def\applytolist#1{
    \expandafter\def\csname multi#1\endcsname##1{
        \def\multiack{##1}\ifx\multiack\semicolon
            \def\next{\relax}
        \else
            \csname #1\endcsname{##1}
            \def\next{\csname multi#1\endcsname}
        \fi
        \next}
    \csname multi#1\endcsname}

\def\calc#1{\expandafter\def\csname c#1\endcsname{{\mathcal #1}}}
\applytolist{calc}QWERTYUIOPLKJHGFDSAZXCVBNM;
\def\bbc#1{\expandafter\def\csname bb#1\endcsname{{\mathbb #1}}}
\applytolist{bbc}QWERTYUIOPLKJHGFDSAZXCVBNM;
\def\bfc#1{\expandafter\def\csname bf#1\endcsname{{\mathbf #1}}}
\applytolist{bfc}QWERTYUIOPLKJHGFDSAZXCVBNM;
\def\sfc#1{\expandafter\def\csname s#1\endcsname{{\sf #1}}}
\applytolist{sfc}QWERTYUIOPLKJHGFDSAZXCVBNM;
\def\ffc#1{\expandafter\def\csname f#1\endcsname{{\mathfrak #1}}}
\applytolist{ffc}QWERTYUIOPLKJHGFDSAZXCVBNM;

\usepackage{yfonts}

\tikzstyle{shaded}=[fill=red!10!blue!20!gray!30!white]
\tikzstyle{unshaded}=[fill=white]
\tikzstyle{empty box}=[circle, draw, thick, fill=white, opaque, inner sep=2mm]
\tikzstyle{annular}=[scale=.7, inner sep=1mm, baseline]
\tikzstyle{rectangular}=[scale=.75, inner sep=1mm, baseline=-.1cm]
\usetikzlibrary{calc}

\newcommand{\nbox}[6]{
	\draw[thick, #1] ($#2+(-#3,-#3)+(-#4,0)$) rectangle ($#2+(#3,#3)+(#5,0)$);
	\coordinate (ZZa) at ($#2+(-#4,0)$);
	\coordinate (ZZb) at ($#2+(#5,0)$);
	\node at ($1/2*(ZZa)+1/2*(ZZb)$) {#6};
}

\begin{document}

\title{Non-tracial free graph von Neumann algebras}
\author{Michael Hartglass\thanks{Department of Mathematics and Computer Science, Santa Clara University}  and Brent Nelson\thanks{Department of Mathematics, Vanderbilt University}}
\date{}
\maketitle

\begin{abstract}
\noindent
Given a finite, directed, connected graph $\Gamma$ equipped with a weighting $\mu$ on its edges, we provide a construction of a von Neumann algebra equipped with a faithful, normal, positive linear functional $(\cM(\Gamma,\mu),\vphi)$. When the weighting $\mu$ is instead on the vertices of $\Gamma$, the first author showed the isomorphism class of $(\cM(\Gamma,\mu),\vphi)$ depends only on the data $(\Gamma,\mu)$ and is an interpolated free group factor equipped with a scaling of its unique trace (possibly direct sum copies of $\bbC$). Moreover, the free dimension of the interpolated free group factor is easily computed from $\mu$. In this paper, we show for a weighting $\mu$ on the edges of $\Gamma$ that the isomorphism class of $(\cM(\Gamma,\mu),\vphi)$ depends only on the data $(\Gamma,\mu)$, and is either as in the vertex weighting case or is a free Araki--Woods factor equipped with a scaling of its free quasi-free state (possibly direct sum copies of $\bbC$). The latter occurs when the subgroup of $\R^+$ generated by $\mu(e_1)\cdots \mu(e_n)$ for loops $e_1\cdots e_n$ in $\Gamma$ is non-trivial, and in this case the point spectrum of the free quasi-free state will be precisely this subgroup.  As an application, we give the isomorphism type of some infinite index subfactors considered previously by Jones and Penneys.
\end{abstract}


\section*{Introduction}

Given a finite, directed, connected graph $\Gamma = (V,E)$, there has been for some time an interest in von Neumann algebras associated to this initial data. More precisely, an interest in a von Neumann algebra generated by projections $p_v$ with $v\in E$ and operators $Y_e$ with $e\in E$ satisfying algebraic relations that encode the structure of the graph. In \cite{MR2732052,MR2807103}, Guionnet, Jones, and Shlyakhtenko used free probabilistic methods to construct $\mathrm{II}_1$-factors associated to the principal graph of a subfactor planar algebra, and used this construction to give a new proof of \cite[Theorem 3.1]{popa1}. Since these $\mathrm{II}_1$ factors contain the projections $p_v$, $v\in V$, the unique tracial state $\tau$ induces a weighting  on the vertices $\mu\colon V\to \bbR^+$ given by $\mu(v):=\tau(p_v)$. In \cite{MR3110503}, the first author showed that these $\mathrm{II}_1$ factors are always interpolated free group factors. Moreover, a more general construction was provided for von Neumann algebras with tracial states $(\cM(\Gamma, \mu),\tau)$ associated to graphs $\Gamma$ and vertex weightings $\mu$. The isomorphism class of $(\cM(\Gamma,\mu),\tau)$ is completely determined by $\mu$, and is always that of an interpolated free group factor (with their unique tracial state) possibly direct sum copies of $\bbC$. In particular, these graph von Neumann algebras provide a convenient presentation of the interpolated free group factors (see also \cite{MR3679687}), which the authors used in \cite{MR3718052} to adapt the free transport results of \cite{MR3251831} to the context of interpolated free group factors. 

The goal of the present paper, is to modify the construction in \cite{MR3110503} in order to produce von Neumann algebras naturally equipped with non-tracial states, and then determine their isomorphism class. It turns out that in order to move beyond tracial states, one must replace the vertex weighting with an \emph{edge weighting} $\mu\colon E\to \bbR^+$ satisfying $\mu(e^{\op})=\mu(e)^{-1}$ for all $e\in E$. In fact, this insight was actually previously observed by Jones and Penneys in \cite{JP17}, where they studied a loop von Neumann algebra $\cM$ with an infinite index subfactor $\cN \subset \cM$ where $L(\F_{\infty}) \cong \cN \subset \cM^{\vphi}$. They showed that the associated factor $\cM$ is type $\mathrm{III}$ precisely when the group
	\[
		\< \mu(e_1)\cdots \mu(e_n)\colon e_1\cdots e_n\text{ is a loop in }\Gamma\> < \bbR^+
	\]
is non-trivial, and that the inclusion $\cN \subset \cM$ is irreducible and discrete. 

One complication that arises when using an edge weighting to construct a von Neumann algebra associated to a graph $\Gamma$ is that there is no longer a canonical state associated with $\cM(\Gamma,\mu)$, though this is to be expected precisely in the non-tracial case. Fortunately, due to the condition $\mu(e^{\op})=\mu(e)^{-1}$, $\Gamma$ admits a large subgraph which is essentially ``tracial'' with respect to the weighting $\mu$, and which in turn allows us to define a state on the algebra generated by the projections $p_v$, $v\in V$. In this way we are able to construct a von Neumann algebra with a faithful state $(\cM(\Gamma,\mu),\vphi)$ associated to the inital data $(\Gamma,\mu)$. Moreover, we show that these pairs are always isomorphic to an almost periodic \emph{free Araki--Woods factor} equipped with its \emph{free quasi-free state} possibly direct sum copies of $\bbC$.

Defined in \cite{Shl97}, Shlyakhtenko's free Araki--Woods factors can be regarded as the non-tracial analogue of the free group factors. Indeed, both arise from a standard free probabilistic construction using creation and annihilation operators on a Fock space, and the free quasi-free state is given by the vacuum vector state, which in the tracial case yields the unique tracial state. 
Importantly, free Araki--Woods factors admit \emph{matricial models} (see \cite[Section 5]{Shl97}), which are---roughly speaking---amplified representations of the factor. The flexibility of such representations is quite powerful and allows one to show that the free Araki--Woods factors have the so-called \emph{free absorption} property: they are stable under free products with the free group factors (see \cite[Corollary 5.5]{Shl97}). We utilize this property frequently in the present paper, and moreover use a matricial model to show that the isomorphism class of free Araki--Woods factors with their free quasi-free states is stable under compressions and amplifications by projections in the centralizer with full central support.

Our general strategy for studying $(\cM(\Gamma,\mu),\vphi)$ is to consider a compression by a vertex projection $p_v$, $v\in V$. Dykema's free product techniques from \cite{MR1201693} allow us to analyze such compressions. The aforementioned stability properties allow us to assert that not only are the compressions free Araki--Woods factors (possibly direct sum copies of $\bbC$), but so is original von Neumann algebra. Of course, great care must be taken to ensure every isomorphism is state preserving.  Additionally, this graphical picture of the free Araki--Woods factors can be used to study free products of arbitrary finite dimensional von Neumann algebras, which the authors pursue in a second paper \cite{Paper2}. 

The structure of the paper is as follows. In Section~\ref{sec:prelim}, we establish notation for states and positive linear functions we will frequently use; we recall the definition of a free Araki--Woods factors and discuss their structure; we recall some existing results due to Dykema (\cite{MR1201693, Dyk97}), Houdayer (\cite{Hou07}), and Shlyakhtenko (\cite{Shl97}) that will be frequently cited; and we recall the Fock-space  and Toeplitz algebra associated to a directed, connected graph $\Gamma$. This Toeplitz algebra is the foundation upon which $\cM(\Gamma,\mu)$ is built, in both the present case and in the vertex weighting case considered in \cite{MR3110503}.

In Section~\ref{sec:adding_weighting_to_edges} we present the construction of the von Neumann algebra $\cM(\Gamma,\mu)$ and a faithful, normal, positive linear functional $\vphi$. The latter is (non-canonically) defined by considering a (non-unique) maximal subraph of $\Gamma$ subject to the condition that $\mu(e_1)\cdots \mu(e_n)=1$ for all loops $e_1\cdots e_n$ in the subgraph. This maximal subgraph then corresponds to a tracial subalgebra of $\cM(\Gamma,\mu)$, which can be classified using \cite{MR3110503}. This section then concludes with an analysis of the central supports of certain projections in the centralizer $\cM(\Gamma,\mu)^\vphi$, which is a crucial part of our aforementioned strategy for studying $(\cM(\Gamma,\mu),\vphi)$. 

In Section~\ref{sec:technical_tools}, we present some technical results, including our matricial model and some essential compression/amplification lemmas. 

In Section~\ref{sec:building_the_graph}, we undertake the analysis of $(\cM(\Gamma,\mu),\vphi)$. We first consider a cyclic subgraph $\Gamma_0$ of $\Gamma$ and the corresponding subalgebra $\cM(\Gamma_0,\mu)$, which we show has a free Araki--Woods factor as its diffuse component. We then ``build'' $\Gamma$ from $\Gamma_0$ by succesively adding edges (and sometimes vertices), and since we are able to control the isomorphism classes of the corresponding subalgebras along the way we are therefore able to deduce the isomorphism class of $(\cM(\Gamma,\mu),\vphi)$. As an application, we classify the aforementioned subfactors considered by Jones in Penneys in \cite{JP17}. This is presented in the Appendix, along with a summary of notation that we have compiled for the convenience of the reader.


\subsection*{Acknowledgements} 

We would like to thank Corey Jones and David Penneys for initially suggesting the idea of this paper. We would also like to thank Dimitri Shlyakhtenko for many helpful conversations about free Araki--Woods factors.  This work was initiated at the Mathematical Sciences Research Institute (MSRI) Summer School on \emph{Subfactors: planar algebras, quantum symmetries, and random matrices}, and continued while Brent Nelson was visiting the Institute for Pure and Applied Mathematics (IPAM) during the Long Program on \emph{Quantitative Linear Algebra}, both of which are supported by the National Science Foundation. Brent Nelson's work was also supported by NSF grant DMS-1502822.


\section{Preliminaries}\label{sec:prelim}


\subsection{Status quo}\label{subsec:status_quo}

Given the non-tracial nature of our analysis, it is important that we specify the positive linear functionals involved in any free product. Toward that end, we establish some common notation for positive linear functionals that will be frequently used:

	\begin{itemize}
	\item\label{notation:implicit_states} After \cite{MR1201693,Dyk97} we use the following notation:
		\begin{itemize}
		\item[$\circ$] For $t>0$ and a projection $p$ 
			\[
				\overset{p}{\underset{t}{\bbC}}:=( \bbC p, \phi),
			\]
		where $\phi$ is determined by $\phi(p)=t$. We may suppress either `$t$' or `$p$' if they are clear from context. In the context of a direct sum, if $t\leq 0$ then we mean that the summand should be omitted.
		
		\item[$\circ$] For $s,t>0$ and $v$ a partial isometry such that $v^*v=p$ and $vv^*=q$ are orthogonal,
			\[
				\overset{p,q}{\underset{s,t}{M_2(\bbC)}}:= ( \bbC\<v\>, \phi),
			\]
		where $\phi$ is determined by $\phi(v)=0$, $\phi(p)=s$, and $\phi(q)=t$. We may suppress any of `$t$',`$s$', `$p$', or `$q$' if they are clear from context. If we merely wish to establish notation for the identity element $r:=p+q$, then we may simply write $\overset{r}{\underset{s,t}{M_2(\bbC)}}$.
		
		\item[$\circ$] With $p$, $q$, $s$, $t$, and $v$ as in the previous bullet point, we denote
		\[
			\overset{p,q}{\underset{s,t}{M_2(L(\Z))}}:= ( \bbC\<v\>, \phi) \otimes (L(\Z), \tau)
		\]	
		with $\tau$ the canonical group-algebra tracial state on $L(\Z)$.
		
		\item[$\circ$] For $t>0$ and a von Neumann algebra $A$ with identity element $p$ and a state $\phi$
			\[
				\overset{p}{\underset{t}{(A,\phi)}} := (A,t \phi).
			\]
		We may suppress any of `$t$', `$p$', or `$\phi$' if they are clear from context (e.g. a $\mathrm{II}_1$ factor and its canonical trace).
		\end{itemize}
	The above notations allow us to concisely express direct sums with explicit (and sometimes implicit) weightings. E.g.:
		\[
			\overset{p_1}{\underset{t_1}{\bbC}}\oplus \overset{p_2,q_2}{\underset{s_2,t_2}{M_2(\bbC)}}\oplus \overset{p_3}{\underset{t_3}{(A,\vphi)}}.
		\]
	If $t_1+s_2+t_2+t_3=1$ then the associated positive linear functional on this direct sum is a state. However, it will often be notationally convenient to \textbf{not} demand such normalization. If such an unnormalized direct sum appears in a free product, we will ensure that each factor in the free product has the same total mass.
	
	\item\label{notation:psi} Let $\cH$ be a separable infinite-dimensional Hilbert space, and $\{e_{i,j}\}_{i,j\in \bbN_0}$ be a system of matrix units for $\cB(\cH)$. For $\lambda\in(0,1)$, after \cite{Shl97} we define a state $\psi_\lambda\colon \cB(\cH)\to\bbC$ by
		\[
			\psi_\lambda(e_{i,j}):=\begin{cases} \lambda^i(1-\lambda) & \text{if }i=j\\ 0 & \text{otherwise}\end{cases}.
		\]
	If $\cH$ is finite dimensional so that $\cB(\cH)\cong M_n(\bbC)$ is generated by matrix units $\{e_{i,j}\}_{i,j=0}^{n-1}$, for some $n\in \bbN$, we define a state $\psi_\lambda\colon M_n(\bbC)\to\bbC$ by
		\[
			\psi_\lambda(e_{i,j}):=\begin{cases} \lambda^i \frac{(1-\lambda)}{(1-\lambda^n)} & \text{if }i=j\\ 0 & \text{otherwise}\end{cases}.
		\]

	\item\label{notation:state_compression} For a von Neumann algebra $A$ with a positive linear functional $\phi$ and a non-zero projection $p\in A$, denote
		\[
			\phi^p(\ \cdot\ ) := \frac{1}{\phi(p)} \phi(p\ \cdot\ p).
		\]
	\end{itemize}


\subsection{Free Araki-Woods factors}

We recall the construction of Shlyakhtenko's free Araki--Woods factors \cite{Shl97} and their salient properties.

Fix a real Hilbert space $\cH_\bbR$ along with an orthogonal representation $\{U_t\}_{t\in \bbR}$ of $\bbR$ on $\cH_\bbR$. Extend this orthogonal representation to a unitary one on $\cH_\bbC:=\cH_\bbR\otimes \bbC$, the complexification of $\cH_\bbR$. Invoke Stone's theorem to produce an infinitesimal generator: there exists a (potentially unbounded) positive, non-singular, self-adjoint operator $A$ such that $A^{it}=U_t$ for all $t\in \bbR$. For $v,w\in \cH_\bbC$, define an new inner product by:
	\[
		\<v,w\>_U : = \<\frac{2}{1+A^{-1}} v, w\>_{\cH_\bbC},
	\]
which is $\bbC$-linear in the right entry. For $v,w\in\cH_\bbR$ it follows that $\text{Re}{\<v,w\>_U} = \<v,w\>_{\cH_\bbR}$. In particular, if $v$ and $w$ are orthogonal in $\cH_\bbR$, then $\<v,w\>_U\in i\bbR$. Denote the completion of $\cH_\bbC$ under this inner product by $\cH$.

Next, the Fock space generated by $\cH$, denoted $\cF(\cH)$, is the completion of
	\[
		\bbC\Omega\oplus \bigoplus_{d\geq 1} \cH^{\otimes d},
	\]
where $\Omega$, a unit vector, is called the \emph{vacuum vector}. Let $\omega(\,\cdot\,):=\<\Omega,\,\cdot\,\Omega\>$ denote the vacuum vector state. For any $v\in \cH$, one can define its \emph{left creation operator} $\ell(v)\in \cB(\cF(\cH))$ by
	\begin{align*}
		\ell(v)& \Omega = v\\
		\ell(v)& w_1\otimes\cdots \otimes w_d = v\otimes w_1 \otimes \cdots \otimes w_d.
	\end{align*}
Its adjoint $\ell(v)^*$, is called the \emph{left annihilation operator}, is determined by
	\begin{align*}
		\ell(v)^*&\Omega = 0\\
		\ell(v)^*& w_1\otimes \cdots \otimes w_d = \<v,w_1\> w_2\otimes \cdots \otimes w_d.
	\end{align*}
Denote $s(v):=\ell(v)+\ell(v)^*$.

The von Neumann algebra
	\[
		\Gamma(\cH_\bbR,U_t)'' := W^*(s(v)\colon v\in \cH_\bbR)
	\]
is called a \emph{free Araki--Woods factor}. The restriction of $\omega$ to this von Neumann algebra, which we denote by $\vphi$, is called the \emph{free quasi-free state}. This state is tracial if and only if $\{U_t\}_{t\in\bbR}$ is the trivial representation. Otherwise, the modular automorphism group $\sigma^\vphi=\{\sigma_t^\vphi\}_{t\in\bbR}$ acts by
	\[
		\sigma_t^\vphi(s(v)) = s(U_{-t} v).
	\]
In this paper, we will be concerned mainly with free Araki--Woods factors whose free quasi-free states are \emph{almost perdioc}, meaning that the modular operator $\Delta_\vphi$ is diagonalizable.  In particular, this arises when the factor is generated by generalized circular elements.
	
\subsubsection*{Generalized Circular Elements}
For any $v\in \cH$, the operators $s(v)$ (with respect to $\vphi$) have a semicircular distribution with mean zero and variance $\|v\|^2$. However, in contrast to the tracial case, if $v,w\in\cH_\bbR$ are orthogonal then $s(v)$ and $s(w)$ are not necessarily free with respect to $\vphi$. Somewhat easier to work with are the so-called \textit{generalized circular elements}: 
	\[
		\ell(g) +\sqrt{\lambda} \ell(h)^*
	\]
where $g,h\in \cH$ are orthogonal unit vectors and $0\leq \lambda \leq 1$. For orthogonal unit vectors $v,w\in \cH_\bbR$ consider:
	\[
		\lambda:= \frac{i+\<v,w\>_U}{i-\<v,w\>_U}
	\]
and
	\[
		g:=\frac{\sqrt{1+\lambda}}{2}(v+iw)\qquad h:= \frac{\sqrt{1+\lambda}}{2\sqrt{\lambda}} (v-iw).
	\]
Then one easily checks that $g$ and $h$ are orthogonal unit vectors in $\cH$ such that
	\[
		\frac{s(v)+is(w)}{2} = \frac{1}{\sqrt{1+\lambda}}(\ell(g) + \sqrt{\lambda} \ell(h)^*).
	\]
Consequently, $\Gamma(\cH_\bbR,U_t)''$ is generated by generalized circular elements. Moreover, if $\text{span}\{v,w\}$ is invariant under $\{U_t\}_{t\in \R}$ then one can show that
	\[
		\sigma_t^\vphi( \ell(g) + \sqrt{\lambda} \ell(h)^*)=\lambda^{it}(\ell(g) + \sqrt{\lambda} \ell(h)^*)
	\]
(see \cite[Proposition 2.1]{Nel17}). That is, $\ell(g)+\sqrt{\lambda}\ell(h)^*$ is an \textit{eigenoperator} with \textit{eigenvalue} $\lambda$ (see \cite[Definition 2.5]{Nel17}).

\subsubsection*{Two Variable Case}
Given $\lambda\in (0,1]$, define an orthogonal representation of $\bbR$ on $\cH_\bbR=\bbR^2$ by
	\[
		U_t = \left[\begin{array}{cc} \cos(t \ln(\lambda)) & -\sin(t\ln(\lambda)) \\ \sin(t\ln(\lambda)) & \cos(t \ln(\lambda)) \end{array}\right]\qquad t\in \bbR.
	\]
Then $\Gamma(\cH_\bbR,U_t)''$ is denoted $T_\lambda$\label{notation:fAWf} and its free quasi-free state is denoted $\vphi_\lambda$. If $\lambda=1$, then $(T_\lambda, \vphi_\lambda)\cong (L(\bbF_2),\tau)$.

For $e_1,e_2\in \R^2$ the usual orthonormal basis, we denote associated generalized circular element by
	\[\label{notation:gce}
		y_\lambda := \frac{\sqrt{1+\lambda}}{2}(s(e_1) +i s(e_2))
	\]
If we $y_\lambda=v_\lambda |y_\lambda|$ be the polar decomposition, then $v_y$ and $|y_\lambda|$ are free with with respect to $\vphi_\lambda$ by \cite[Theorem 4.8]{Shl97}. Moreover, $|y_\lambda|$ is diffuse and $v_\lambda$ is an isometry (non-unitary when $\lambda<1$): $\vphi_\lambda(v_\lambda^*v_\lambda)= 1$ and $\varphi_\lambda(v_\lambda v_\lambda^*)=\lambda$. From this we obtain the following picture:
	\[
		(T_\lambda, \vphi_\lambda) \cong (L(\bbZ),\tau) * (\cB(\ell^2(\N_{0})),\psi_\lambda).
	\]
It is also known that the law of $|y_\lambda^*|$ has an atom of size $1-\lambda$ at zero, but is otherwise diffuse so that:
	\[
		(W^*(|y_\lambda^*|),\vphi_\lambda)\cong \underset{\lambda}{(L(\bbZ),\tau)}\oplus \underset{1-\lambda}{\bbC}
	\]
(see \cite[Remark 4.4]{Shl97}).

For $\lambda>1$, we write $(T_\lambda,\vphi_\lambda):=(T_{\lambda^{-1}},\vphi_{\lambda^{-1}})$, and $y_\lambda:=y_{\lambda^{-1}}^*$. In this way, we have that for any $\lambda>0$, $y_\lambda$ is eigenoperator of $\vphi_\lambda$ with eigenvalue $\lambda$ and that $T_\lambda=W^*(y_\lambda)$. Additionally, if $H$ is a countable subgroup of $\bbR^+$ then (after \cite{Hou07}) we denote
		\[\label{notation:fAWfH}
			(T_H,\vphi_H):= \bigast_{\lambda\in H}(T_\lambda,\varphi_\lambda).
		\]
By \cite[Theorem 6.4]{Shl97}, $(T_H,\vphi_H)$ depends only on $H$ so that $(T_H,\vphi_H)*(T_\lambda,\vphi_\lambda) \cong (T_H,\vphi_H)$ for all $\lambda\in H$. In addition, the above free product can be taken over any generating set of $H$, i.e. one has 
		\[
			(T_H,\vphi_H)\cong \bigast_{\lambda\in S}(T_\lambda,\varphi_\lambda)
		\]
for any generating set $S$ of $H$.


\subsection{References to existing results}

For the convenience of the reader, we state here some existing results that will be cited frequently in the present paper. Where appropriate, we have adapted the notation. In particular, for $M$ a von Neumann algebra and $p\in M$ a projection, we denote the central support of $p$ in $M$ by $z(p\colon M)$\label{notation:central_support}.

The first lemma concerns free products with respect to general states and follows from the same proof as \cite[Theorem 1.2]{MR1201693} (see also \cite[Proposition 5.1]{Dyk97} and \cite[Proposition 3.10]{Hou07}). In particular, we will frequently use the cases when either $\cB(\cH)=\bbC$ or $B=0$.

\begin{lem}\label{lem:Dykema}
Let $(A,\phi)$, $(B,\psi)$, and $(C,\nu)$ be von Neumann algebras equipped with faithful normal states. Let $\cH$ be a separable Hilbert space, equip $\cB(\cH)$ with a faithful normal state $\omega$, and let $p\in \cB(\cH)^\omega$ be a minimal projection. If
	\begin{align*}
		(M,\vphi)&:= \left[ \left\{(A,\phi)\bar{\otimes} (\cB(\cH),\omega)\right\} \oplus (B,\psi)\right] * (C,\nu)\\
		(N,\vphi)&:= \left[ (\cB(\cH),\omega) \oplus (B,\psi)\right]* (C,\nu),
	\end{align*}
then
	\[
		(pMp,\vphi^p) \cong (pNp,\varphi^p) * (A,\phi).
	\]
Moreover, $z(p\colon M)=z(p\colon N)$.
\end{lem}

The next two lemmas more specifically concern free Araki--Woods factors. The following is shown explicitly in the proof of \cite[Theorem 3.1]{Hou07} when $\cH$ is infinite dimensional, but using \cite[Proposition 6.9]{Shl97} the proof for the finite dimensional case is identical.

\begin{lem}\label{lem:mini_absorption}
Let $\lambda\in(0,1)$ and let $t\in [0,1]$. Let $\cH$ be a separable Hilbert space, and equip $\cB(\cH)$ with the state $\psi_{\lambda}$ (relative to some choice of matrix units). Then
	\[
		(L(\bbZ),\tau)) * \left[ (\cB(\cH),\psi_\lambda) \otimes (\underset{t}{\bbC}\oplus \underset{1-t}{\bbC})\right] \cong (T_\lambda,\vphi_\lambda).
	\]
\end{lem}

The next lemma establishes a property of the free Araki--Woods factors known as \emph{free absorption}, and this is how we shall refer to it throughout the present paper. The proof uses a ``matricial model'' for $\Gamma(U_t,\cH_\bbR)''$, which we present a slightly modified version of in Subsection~\ref{subsec:matrix_model}.

\begin{lem}[{\cite[Corollary 5.5]{Shl97}}]
For any $\lambda\in (0,1)$ and any $t\geq 1$,
	\[
		(T_\lambda,\vphi_\lambda) * (L(\bbF_t),\tau) \cong (T_\lambda,\vphi_\lambda).
	\]
\end{lem}

Since $(T_\lambda,\vphi_\lambda) * (L(\bbF_\infty),\tau) \cong (T_\lambda,\vphi_\lambda)$, we have the following slightly stronger statement, which we will invoke later in the paper.

\begin{prop}

If $(A, \phi)$ is a countable direct sum of finite dimensional von Neumann algebras, hyperfinite von Neumann algebras, and interpolated free group factors with $\phi$ a trace, then
\[
	(T_\lambda,\vphi_\lambda) * (A, \phi) \cong (T_\lambda,\vphi_\lambda).	
\]

\end{prop}

Finally, we will make use of the following proposition that will allow us to convert some amalgamated free products to free products over the scalars.

\begin{prop}[{\cite[Proposition 4.1]{Hou07}}]\label{prop:amalgamate}

Let $(M, \phi)$ be a von Neumann algebra with a faithful normal state, and $B \subset M$ a von Neumann subalgebra with a $\phi$-preserving conditional expectation $E_{1}: M \rightarrow B$.  Let $(A, \psi)$ be another von Neumann algebra with faithful normal state, and $E_{2}: (A, \psi) * (B, \phi) \rightarrow (B, \phi)$ the canonical $\phi$-preserving conditional expectation.  Set 
$$
(\cM, E) = (M, E_{1}) \Asterisk_B ((A, \psi) * (B, \phi), E_{2}).
$$
Then if $\vphi = \phi \circ E$,
$$
(\cM, \vphi) \cong (M, \phi) * (A, \psi).
$$
\end{prop}

\subsection{Graph Fock-space and the Toeplitz algebra}\label{subsec:graph_Fock_space}

Suppose $\Gamma= (V, E)$\label{notation:graph} is a finite, directed, connected graph with vertex set $V$, and edge set $E$.  For $e \in E$, we let $s(e)$\label{notation:edge} and $t(e)$ denote the source and target of $e$, respectively.  We assume that $E$ satisfies the following property:
\begin{itemize}

\item For each $e \in E$, there exists a unique $e^{\op} \in E$ satisfying $s(e) = t(e^{\op})$ and $t(e) = s(e^{\op})$.  We require that $(e^{\op})^{\op} = e$ for all $e \in E$.

\end{itemize}
Note that if $e$ is a self loop based at some $v \in V$, we could assign $e = e^{\op}$, but that is not required. 

We denote by $A = \ell^{\infty}(V)$ the space of complex valued functions on $V$, and by $p_{v}$ the indicator function on $v \in V$.  We recall the \emph{Toeplitz algebra} of $\Gamma$ as follows: Let $\C[E]$ be the the complex vector space with basis $E$.  $\C[E]$ comes equipped with a $A-A$ bimodule structure determined by
$$
	p_{v}\cdot e\cdot p_w = \delta_{v, s(e)}\delta_{w,t(e)} e
$$
and $A$-valued inner product given by
$$
\langle e | e' \rangle_{A} = \delta_{e, e'} p_{t(e)}
$$
which is extended to be linear in the right variable. The \emph{Fock space} of $\Gamma$, denoted $\cF(\Gamma)$, is the right C*-Hilbert module
$$
\cF(\Gamma) = A \oplus \bigoplus_{n \geq 1} \C[E]^{\otimes^{n}_{A}}.
$$
$\cF(\Gamma)$ has a canonical left action by $\ell^{\infty}(V)$ given by bounded, adjointable operators:
	\[
		p_{v}\cdot e_{1}\otimes \cdots \otimes e_{n} = \delta_{v, s(\e_{1})} e_{1} \otimes \cdots \otimes e_{n}.
	\]  
For each $e \in E$, we define the creation operator $\ell(e)$ by
	\begin{align*}
		\ell(e)&  p_{v} = \delta_{v, t(e)} e \\
		\ell(e)&  e_{1} \otimes\cdots\otimes e_{n} = e \otimes e_{1} \otimes \cdots \otimes e_{n}.
	\end{align*}
Then $\ell(e)$ is bounded and adjointable with adjoint given by
	\begin{align*}
		\ell(e)^{*}&  p_{v} = 0\\
		\ell(e)^{*}&  e_{1} \otimes \cdots \otimes e_{n} = \langle e|e_{1}\rangle_{A} e_{2} \otimes \cdots \otimes e_{n}.
	\end{align*}
 The Toeplitz algebra associated to $\Gamma$, denote $\cT(\Gamma)$, is the C*-algebra generated by the collection $\{\ell(e)\, \colon e \in E\}$.


\section{Graph Algebras with Edge Weightings}\label{sec:adding_weighting_to_edges}

Let $\Gamma=(V, E)$ be as in Subsection~\ref{subsec:graph_Fock_space}. In contrast with the tracial setting in which one typically equips $\Gamma$ with a vertex weightings (see \cite{ MR3110503, MR3266249, MR3679687, MR3718052}), we will equip $\Gamma$ with an \textbf{edge weighting}: a function $\mu: E \rightarrow \R^{+}$\label{notation:mu} satisfying $\mu(e^{\op}) = \mu(e)^{-1}$. Note that this is more general since any vertex weighting $\mu_0\colon V\to \bbR^+$ defines an edge weighting $\mu$ by $\mu(e):=\frac{\mu_0(t(e))}{\mu_0(s(e))}$.

To each $e \in E$, define $Y_{e}\in\cT(\Gamma)$\label{notation:Y_e} to be the element $Y_{e} = \ell(e) + \sqrt{\mu(e)}\ell(e^{\op})^{*}$.  Note that
\begin{itemize}

\item $p_{s(e)}Y_{e}p_{t(e)} = Y_{e}$

\item $Y_{e}^{*} = \sqrt{\mu(e)}Y_{e^{\op}} $

\end{itemize}
We will also denote the polar decomposition by $Y_e=u_e |Y_e|$.

As is \cite{MR3266249, MR3679687}, we set $\cS(\Gamma, \mu)$\label{notation:S} to be the C*-algebra generated by $A$ together with $\{ Y_{e}: e \in \vec{E}\}$.  The arguments used in \cite[Theorem 5.19]{MR3266249} can be used to show the following:
\begin{prop}
Define $\bbE: \cS(\Gamma, \mu) \rightarrow A$ by $\bbE(x) = \langle 1_{A} | x1_{A}\rangle_{A}$.  Then: 
\begin{enumerate}[(a)]
\item $\bbE$ is a faithful conditional expectation of $ \cS(\Gamma, \mu)$ onto $A$.

\item The C*-algebras C*$(A, Y_{f},Y_{f^{op}})$ as $f$ ranges through all pairs $(e, e^{\op})$ are free with amalgamation over $A$ under $\bbE$.
\end{enumerate}
\end{prop}
Let $\phi$ be any faithful state on $\cS(\Gamma,\mu)$ that preserves $\bbE$, and let $\cM(\Gamma,\mu,\phi)$ denote the von Neumann algebra generated by $\cS(\Gamma,\mu)$ via the GNS representation associated to $\phi$. From the formula for $Y_{e}$, it follows that in $p_{s(e)}\cM(\Gamma, \mu, \phi)p_{s(e)}$, $Y_{e}Y_{e}^{*}$ has the same distribution as $y_{e}y_{e}^{*}$ where $y_{e}$ is a generalized circular element.  Explicitly, this means
	\[
		p_{s(e)}W^*(Y_{e}Y_{e}^{*})p_{s(e)} \cong \begin{cases}
													\underset{\color{white}\phi(p_{s(e)})\mu(e)}{L(\Z)} &\text{ if } \mu(e) \geq 1\\
													\underset{\phi(p_{s(e)})\mu(e)}{L(\Z)} \oplus \underset{\phi(p_{s(e)})(1 - \mu(e))}{\C} &\text{ if } \mu(e) < 1
													\end{cases}
	\]
In particular, it follows that $u_{e}u_{e}^{*} = p_{s(e)}$ if and only if $\mu(e) \geq 1$, and $u_{e}^{*}u_{e} = p_{t(e)}$ if and only if $\mu(e) \leq 1$.

There is another approach to understanding the above proposition in terms of Shlyakhtenko's operator valued semicircular systems \cite{Shl99}. Place an equivalence relation on $E$ by identifying $e$ with $e^{\op}$ for each $e \in E$ except for self-loops satisfying $e\neq e^{op}$, in which case we define $e$ and $e^{\op}$ to be inequivalent. Denote $\overline{E}$ to be the space of equivalence classes, and let $[e]$ denote the equivalence class of $e \in E$. 

For each pair $[e], [f] \in \overline{E}$, we define maps $\eta_{[e], [f]}: A \rightarrow A$ as follows. If $e\in E$ satisfies $[e]=[e^{op}]$, then define $\eta_{[e],[f]}$ to be the identically zero if $[f]\neq [e]$, and otherwise let it be the linear extension of
	\[
		\eta_{[e], [e]}(p_{v}) =  \begin{cases} 0 &\text{ if } v\neq s(e) \text{ and } v \neq t(e)\\
								\sqrt{\mu(e)}p_{s(e)} &\text{ if } v = t(e)\\
								\sqrt{\mu(e^{\op})}p_{t(e)} &\text{ if } v = s(e)
						\end{cases}.
	\]
If $e\in E$ satisfies $[e]\neq [e^{\op}]$, then define $\eta_{[e],[f]}$ to be identically zero for all $f\not\in \{e,e^{\op}\}$, and define $\eta_{[e],[e]}$ and $\eta_{[e],[e^{\op}]}$ to be the respective linear extensions of
	\begin{align*}
		\eta_{[e],[e]}(p_v) &= \delta_{s(e),v} p_{s(e)}\\
		\eta_{[e],[e^{\op}]}(p_v) &= -i\frac{\mu(e) -1}{\mu(e)+1} \delta_{s(e),v} \cdot  p_{s(e)}.
	\end{align*}
It is easy to see that the map $\eta: M_{\overline{E} \times \overline{E}}(A) \rightarrow M_{\overline{E} \times \overline{E}}(A)$ given by $(\eta)_{[e], [f]} = \eta_{[e], [f]}$ is completely positive. Therefore, one can form the C*-algebra $\Phi(A, \eta)$ as in \cite{Shl99}, which is generated by $A$ and self adjoint elements $X_{[e]}$, $[e]\in \overline{E}$.  There is a faithful conditional expectation $\bbF: \Phi(A, \eta) \rightarrow  A$ uniquely determined by $\bbF(X_{[e]}aX_{[f]}) = \eta_{[e], [f]}(a)$ for all $a \in A$ and 
	\begin{align*}
		\bbF&(a_0X_{[e_{i_1}]}a_{1}X_{[e_{i_2}]}\cdots a_{n-1}X_{[e_{i_n}]}a_{n})\\
			&= \sum_{k=2}^{n}a_{0}\eta_{[e_{i_1}], [e_{i_k}]}(a_{1}\bbF(X_{[e_{i_2}]}\cdots a_{k-2}X_{[e_{i_{k-1}}]})a_{k-1})\bbF(a_{k}X_{[e_{k+1}]}\cdots a_{n-1}X_{[e_{i_n}]}a_{n})
	\end{align*}
for $a_{0}, \cdots, a_{n} \in A$.

From the Fock-space picture in \cite{Shl99}, one sees that the mapping $\C\langle A, X_{[e]} \rangle \rightarrow \C\langle A, Y_{e} \rangle$ which is the identity on $A$ and is determined by
	\[	
		X_{[e]} \mapsto\begin{cases} \frac{1}{\sqrt[4]{\mu(e)}}(Y_{e} + Y_{e}^{*}) &\text{ if } e \text{ is not a loop }\\ 
							Y_{e} &\text{ if } e = e^{\op}\\
\frac{1}{\sqrt{1+\mu(e)}}Y_{e} + \sqrt{\frac{\mu(e)}{1 + \mu(e)}} Y_{e^{\op}} &\text{ if } e \text{ is a loop, } e \neq e^{\op}, \, \mu(e) \leq 1\\
							i\left(\frac{1}{\sqrt{1+\mu(e)}}Y_{e} - \sqrt{\frac{\mu(e)}{1 + \mu(e)}} Y_{e^{\op}}\right) &\text{ if } e \text{ is a loop, } e \neq e^{\op}, \, \mu(e) \geq 1\\
						\end{cases}
	\]
extends to a $*$-algebra isomorphism which intertwines $\bbF$ and $\bbE$. (Note that if $\mu(e)=1$ for $e$ a self-loop satisfying $e\neq e^{\op}$, a non-canonical choice must be made for how one maps $(X_{[e]},X_{[e^{\op}]})$.)  From \cite{Shl99}, the von Neumann algebra $(\Phi(A, \eta), \phi \circ \bbF)''$ is independent of the state $\phi$ on $A$.  Therefore, $\cM(\Gamma, \mu,\phi)$ is independent of the faithful positive linear functional, $\phi$ on $A$.  

\begin{defn}\label{notation:M}
Let $\phi$ be any faithful positive linear functional on $A$.  We define $\cM(\Gamma, \mu)$ to be the von Neumann algebra generated by $\cS(\Gamma, \mu)$ in the GNS representation under $\phi \circ \bbE$.
\end{defn} 

It will be convenient to study $\cM(\Gamma, \mu)$ under special positive linear functionals, $\vphi$.  To construct such a functional, it is helpful to set up the following notation:
\begin{nota}
\begin{itemize}
\item[]

\item\label{notation:paths} Denote $\Pi_{\Gamma}$ as the set of paths in $\Gamma$.  If $\Gamma$ is clear from context, we will simply write this set as $\Pi$.

\item\label{notation:loops} Denote $\Lambda_{\Gamma}$ as the set of loops in $\Gamma$.  If $\Gamma$ is clear from context, we will simply write this set as $\Lambda$.

\item For $\sigma = e_{1}\cdots e_{n} \in \Pi$, set $\mu(\sigma) = \mu(e_{1})\cdots \mu(e_{n})$, $s(\sigma) = s(e_{1})$, and $t(\sigma) = t(e_{n})$.

\item\label{notation:H} Denote $H(\Gamma,\mu)$ as the subgroup of $\bbR^+$ generated by $\{\mu(\sigma)\colon \sigma\in \Lambda_\Gamma\}$.
\end{itemize}
\end{nota}

We now consider the collection of subgraphs $(\Xi, \mu)$ of $(\Gamma, \mu)$ satisfying $H(\Xi,\mu)=\{1\}$; that is,
	\[
		\mu(\sigma) = 1 \qquad \forall\sigma \in \Lambda_{\Xi}.
	\]
We denote by $\Gamma_{\Tr}$\label{notation:tracial_subgraph} a maximal subgraph (under inclusion) of the collection of $(\Xi, \mu)$, which we point out need not be unique; however, note that by maximality and the symmetry $\mu(e^{\op})=\mu(e)^{-1}$, we always have $V(\Gamma_{\Tr}) = V$. We place a positive functional $\vphi$\label{notation:phi} on $A$ by fixing $* \in V$, declaring $\vphi(p_{*}) := 1$, and setting $\vphi(p_{v}) := \mu(\sigma)$ where $\sigma \in \Pi_{\Gamma_{\Tr}}$ satisfies $s(\sigma) = *$ and $t(\sigma) = v$. Note that by the above condition defining $\Gamma_{\Tr}$, this definition does not depend on the choice of $\sigma \in \Pi_{\Gamma_{\Tr}}$. Observe that for any $v, w \in V$ we have $\vphi(p_{w}) = \mu(\sigma)\vphi(p_v)$, where $\sigma \in \Pi_{\Gamma_{\Tr}}$ satisfies $s(\sigma) = v$ and $t(\sigma) = w$. In particular, this implies that---up to the choice of $\Gamma_{\Tr}$ and scaling---$\vphi$ is independent of the choice of $*\in V$.

We extend $\vphi$ to $\cM(\Gamma, \mu)$ via $\vphi \circ \bbE$, and we will also denote this extension as $\vphi$.  We have the following proposition which details the joint law of the $(Y_{e})_{e \in E}$ under $\vphi$.

\begin{prop}\label{prop:edges_elements_are_eigenops}
Let $Q \in \cM(\Gamma, \mu)$.  Then $\vphi(Y_{e}Q) = \mu(e)\mu(\sigma)\vphi(QY_{e})$ where $\sigma \in \Pi_{\Gamma_{\Tr}}$ satisfies $s(\sigma) = t(e)$ and $t(\sigma) = s(e)$.

\end{prop}

\begin{proof}

It suffices to prove this theorem when $Q = Y_{e_{1}}\cdots Y_{e_{n}}$, where $e_1\cdots e_n\in \Pi_\Gamma$, $s(e_{1}) = t(e)$, and $t(e_{n}) = s(e)$.  Let $\sigma \in \Pi_{\Gamma_{\Tr}}$ satisfy $s(\sigma) = t(e)$ and $t(\sigma) = s(e)$. For each $j=1,\ldots,n$, let $\alpha_{j-1}$ and $\beta_{j+1}$ denote scalars such that
	\[
		\mathbb{E}(Y_{e_1}\cdots Y_{e_{j-1}}) = \alpha_{j-1} p_{s(e_j)} \qquad \mathbb{E}(Y_{e_{j+1}}\cdots Y_{e_{n}}) = \beta_{j+1} p_{t(e_j)}	
	\]  
Then using a similar argument to \cite[Lemma 2.6]{MR3679687} and the definitions of $Y_{e}$ and $Y_{e^{\op}}$, we have:
	\begin{align*}
		\vphi(Y_{e}Y_{e_{1}}\cdots Y_{e_{n}}) &= \sum_{e_{j} = e^{\op}} \sqrt{\mu(e)}\vphi(\ell(e^{\op})^{*}\bbE(Y_{e_1}\cdots Y_{e_{j-1}})\ell(e^{\op})\bbE(Y_{e_{j+1}}\cdots Y_{e_n}))\\
			&= \sum_{e_{j} = e^{\op}} \sqrt{\mu(e)}\vphi(\ell(e^{\op})^{*}\alpha_{j-1}p_{t(e)}\ell(e^{\op})\beta_{j+1}p_{s(e)})\\
			&= \sum_{e_{j} = e^{\op}} \sqrt{\mu(e)}\vphi(\alpha_{j-1}\beta_{j+1}p_{s(e)})\\
			&= \mu(e)\mu(\sigma) \sum_{e_{j} = e^{\op}} \frac{1}{\sqrt{\mu(e)}}\vphi(\alpha_{j-1}\beta_{j+1}p_{t(\e)})\\
			&=  \mu(e)\mu(\sigma) \sum_{e_{j} = e^{\op}} \sqrt{\mu(e^{\op})}\vphi(\alpha_{j-1}p_{t(e)}\ell(e)^{*}\beta_{n+1}p_{s(e)}\ell(e))\\
			&= \mu(e)\mu(\sigma) \sum_{e_{j} = e^{\op}} \sqrt{\mu(e^{\op})}\vphi(\bbE(Y_{1}\cdots Y_{j-1})\ell(e)^{*}\bbE(Y_{j+1}\cdots Y_{n})\ell(e))\\
			&= \mu(e)\mu(\sigma)\vphi(Y_{e_{1}}\cdots Y_{e_{n}}Y_{e}) 
	\end{align*}
as desired.
\end{proof}

\begin{cor}
If $e \in E(\Gamma_{\Tr})$, then $Y_{e} \in \cM(\Gamma, \mu)^{\vphi}$.  More generally, each $Y_{e}$ is an eigenoperator of $\Delta_{\vphi}$ with eigenvalue $\mu(e)\mu(\sigma)$ with $\sigma \in \Pi_{\Gamma_{\Tr}}$ satisfying $s(\sigma) = t(e)$ and $t(\sigma) = s(e)$.
\end{cor}

\subsection{Some remarks about $\cM(\Gamma_{\Tr}, \mu)$}

In the case that $\mu(\sigma) = 1$ for every $\sigma \in \Lambda_{\Gamma}$, then $\vphi$ is a trace on $\cM(\Gamma, \mu)$.  Furthermore, if we define $X_{e} = \frac{1}{\sqrt[4]{\mu(e)}}Y_{e}$, then
$$
X_{e} = \frac{1}{\sqrt[4]{\mu(e)}}\ell(e) + \sqrt[4]{\mu(e)}\ell(e^{\op})^{*} = \sqrt[4]{\frac{\vphi(p_{s(e)})}{\vphi(p_{t(e)})}}\ell(e) + \sqrt[4]{\frac{\vphi(p_{t(e)})}{\vphi(p_{s(e)})}}\ell(e^{\op})^{*}
$$ 
It follows that $(\cM(\Gamma, \mu), \vphi)$ is the von Neumann algebra $(\cM(\Gamma, \overline{\mu}), \Tr)$ from \cite{MR3266249} where $\overline{\mu}: V \rightarrow \R^{+}$ is given by $\overline{\mu}(v) = \varphi(p_{v})$.  

This means that we can completely determine the structure of $(\cM(\Gamma_{\Tr}, \mu), \vphi)$, and this will be used in the upcoming sections.  In particular we see that $(\cM(\Gamma_{\Tr}, \mu), \vphi)$ is a factor if and only if $\Gamma_{\Tr}$ has at least two distinct pairs of edges and for every vertex, $v$, we must have
$$
\vphi(p_{v}) \leq \sum_{w \sim v} n_{v, w}\vphi(p_{w})
$$
where $w \sim v$ means that $w$ is connected to $e$ by at least one edge, and $n_{v, w}$ denotes the number of edges with source $v$ and target $w$.  In terms of our edge weighting, this is equivalent to
$$
1 \leq \sum_{s(e) = v} \mu(e).
$$


\subsection{Tracking centralizer central supports}

We will be applying Lemmas \ref{lem:antman1} and \ref{lem:antman2} below to study the structure of $\cM(\Gamma, \mu)$.  In the application of these lemmas, specifically Lemma~\ref{lem:antman2}, it will be crucial to track the central support in $\cM(\Gamma, \mu)^{\varphi}$ of specific projections. Recall that for a projection $p\in M$, we denote its central support in $M$ by $z(p\colon M)$.

\begin{lem}\label{lem:centralizersupport}
Assume that $\Gamma$ is connected with at least two pairs of edges, and for each $e \in E$, let $Y_{e} = u_{e}|Y_{e}|$ be the polar decomposition.  Then for any $e \in E$, $z(u_{e}u_{e}^{*}: \cM(\Gamma, \mu)^{\vphi})$ is $\displaystyle z = \bigvee_{f \in E} u_{f}u_{f}^{*}$.
\end{lem}

\begin{proof}

First note that this is immediate if $|V| = 1$, for in this case $\cM(\Gamma, \mu) = (T_{H}, \vphi_{H})$, where $H=H(\Gamma,\mu),$ see Lemma \ref{lem:base_case0} below, and $(T_{H}, \vphi_{H})^{\vphi}$ is a factor.  We therefore assume $|V| \geq 2$.

Fix $e\in E$ and set $z' = z(u_{e}u_{e}^{*}: \cM(\Gamma, \mu)^{\vphi})$.  It is clear by definition that $z \in \cM(\Gamma, \mu)^{\vphi}$ and is a central projection in $\cM(\Gamma, \mu)$ (it must commute with each $Y_{e}$ and $p_{v}$), hence $z' \leq z$.
Note that one must have $\sum_{w \in V} z'p_{w} = z'$.  
To show that $z' = z$, we will argue that for each $w \in V$, $z'p_{w} = zp_{w}$, the latter of which is $q_{w} = \bigvee_{s(f) = w} u_{f}u_{f}^{*}$.

Denote $v = s(e)$, and set $S = \{f \in E: s(f) = v\}$. Consider $\cN_{v} := p_{v}W^{*}(Y_{f}Y_{f}^{*}\, : \, f \in S)p_{v}$.  Note that $\cN_{v} \subset \cM(\Gamma, \mu)^{\varphi}$.  By freeness, together with \cite{MR1201693}, we see that 
\[
\cN_{v} \cong \begin{cases}
\overset{q_v}{L(\F_{s})} \oplus \overset{p_{v} - q_{v}}{\C} &\text{ if } |S| \geq 2\\
\overset{q_v}{L(\Z)} \oplus \overset{p_{v} - q_{v}}{\C} &\text{ if } |S| = 1
\end{cases}
\]
for some $s > 1$.  If $|S| \geq 2$, then $L(\F_{s})$ is a factor, and $q_{v}$ is the identity of $L(\F_{s})$, so we see that $z'p_{v} \geq q_v= zp_{v}$, hence we have equality. If $|S|=1$, we simply have $q_v=u_eu_e^*$ and so  $z'p_{v} = zp_{v}$ is immediate.

We now examine the vertices $w$ which are connected to $v$ through an edge $f \in E(\Gamma_{\Tr})$.  Let $s(f) = v$, and $t(f) = w$.  
From above, $z(u_{f}u_{f}^{*}: \cM(\Gamma, \mu)^{\vphi}) = z'$  since we have equality under $p_{v}$.  Since $u_f\in \cM(\Gamma,\mu)^\vphi$,
it follows that $z'=z(u_{f}^{*}u_{f}: \cM(\Gamma, \mu)^{\vphi})$.  Using the arguments in the previous paragraph applied to $f^{\op}$, it follows that $z'p_{w} = \bigvee_{s(g) = w} u_{g}u_{g}^{*}$, which is $zp_{w}$.

To finish the proof, we iterate this computation, and note that every vertex admits an edge which is in $\Gamma_{\Tr}$ (provided $|V| \geq 2$).
\end{proof}

\begin{rem}\label{rem:central_support_same_in_centralizer}
As mentioned in the above proof,  the projection $z$ is central in $\cM(\Gamma, \mu)$.  It therefore follows that $z(u_{e}u_{e}^{*}: \cM(\Gamma, \mu)) = z$, i.e. $z(u_{e}u_{e}^{*}: \cM(\Gamma, \mu)) = z(u_{e}u_{e}^{*}: \cM(\Gamma, \mu)^{\vphi})$.
\end{rem}

\begin{rem}\label{rem:atomic}
Clearly, $\cM(\Gamma, \mu) = z\cM(\Gamma, \mu) + (1-z)\cM(\Gamma, \mu)$.  Furthermore, since $1-z$ is orthogonal to every $Y_{e}$, we see that $(1-z)\cM(\Gamma, \mu)$ is purely atomic:
$$
(1-z)\cM(\Gamma, \mu) = \bigoplus_{v \in V} \underset{\vphi(p_{v})\left[1 - \sum_{s(e) = v}\mu(e)\right]}{\overset{r_{v}}{\C}}
$$
with $r_{v} \leq p_{v}$.  Therefore to understand the structure of $\cM(\Gamma, \mu)$, it is sufficient to determine $z\cM(\Gamma, \mu)$.
\end{rem}


\section{Technical Tools}\label{sec:technical_tools}


\subsection{Matricial Model}\label{subsec:matrix_model}

Let $H\subset \bbR^+$ be a subgroup and let $\lambda_0\in H\cap (0,1)$. Let $\cH$ be a separable infinite-dimensional Hilbert space and $\{e_{i,j}\}_{i,j\in \bbN_0}$ be a system of matrix units for $\cB(\cH)$. Let $\psi_{\lambda_0}$ be as in Subsection~\ref{subsec:status_quo} and fix $t\in (0,1)$. In this section we produce a matricial model for
	\[
		(T_H,\varphi_H) * \left[(\cB(\cH),\psi_{\lambda_0})\otimes ( \underset{t}{\bbC} \oplus \underset{1-t}{\bbC})\right],
	\]
which we note is isomorphic to $(T_H,\varphi_H)$ by Lemma~\ref{lem:mini_absorption}.

Let $\cK$ be a Hilbert space with orthonormal basis
	\[
		\left\{\xi(\lambda,i,j,a,b), \eta(\lambda,i,j,a,b)\colon \lambda\in H,\ i,j\in\bbN_0,\ a,b\in \{0,1\}\right\}.
	\]
The ambient von Neumann algebra that will contain our matricial model is:
	\[
		(B,\Phi):=( \cB(\cF(\cK)),\omega) \otimes \left[(\cB(\cH),\psi_{\lambda_0})\otimes \underset{t,1-t}{M_2(\bbC)}\right],
	\]
where $\omega$ is the vacuum state on $\cB(\cF(\cK))$. For each $\lambda\in H$, define operators in $B$
	\begin{align*}
		L^\xi_\lambda &:=\sum_{\substack{i,j\in\bbN_0\\ a,b\in \{0,1\}}} \sqrt{\lambda_0^i(1-\lambda_0)[(1-a)t+ a(1-t)]}\ell(\xi(\lambda, i,j,a,b) )\otimes e_{i,j}\otimes e_{a,b} \\
		L^\eta_\lambda&:=\sum_{\substack{i,j\in\bbN_0\\ a,b\in \{0,1\}}} \sqrt{\lambda_0^j(1-\lambda_0)[(1-b)t+ b(1-t)]}\ell(\eta(\lambda, i,j,a,b) )\otimes e_{j,i}\otimes e_{b,a}
	\end{align*}
We then define for each $\lambda\in H$
	\[
		Y_\lambda:= L^\xi_\lambda + \sqrt{\lambda} (L^\eta_\lambda)^*.
	\]
For each $\lambda\in H$, let $y_\lambda$, $\lambda\in H$, denote the generalized circular element generating $T_\lambda\subset T_H$. It follows from \cite[Theorem 5.2]{Shl97} that the map
	\begin{align*}
		y_\lambda&\mapsto Y_\lambda\\
		e_{i,j}\otimes (\alpha,\beta)&\mapsto 1\otimes e_{i,j} \otimes (\alpha e_{0,0} + \beta e_{1,1})
	\end{align*}
extends to a state-preserving embedding
	\[
		(T_H,\varphi_H) * \left[(\cB(\cH),\psi_{\lambda_0})\otimes ( \underset{t}{\bbC} \oplus \underset{1-t}{\bbC})\right]\hookrightarrow (B,\Phi).
	\]

\subsection{Compression and Amplification Lemmas}

\begin{lem}\label{lem:antman1}

Suppose that $(M, \varphi) \cong (T_H,\varphi_H)$ for some non-trivial, countable subgroup $H$ of $\R^+$.  Let $p \in M^\vphi$ be a projection.  Then
	\[
		\left(pMp, \vphi^p\right) \cong (T_H, \varphi_H).
	\]
\end{lem}

\begin{proof}
Fix $\lambda_0\in H\cap (0,1)$ and $t\in [0,1]$. Identify $(M,\varphi)$ with our matricial model for $(T_H,\vphi_H)$ from Subsection~\ref{subsec:matrix_model}. Suppose $\vphi(p) = \lambda_0^{l}(1 - \lambda_0^{k})$ for some $\ell\in \bbN_0$ and $k\in \bbN$. Since the centralizer of a free Araki--Woods factor with respect to its free quasi-free state is a factor, without loss of generality we may assume
	\[
		p = 1\otimes (e_{\ell,\ell} + e_{\ell+1,\ell+1} + \cdots + e_{\ell+k-1, \ell+k-1})\otimes 1 \in \cB(\cF(\cK)) \otimes \cB(\cH)\otimes M_2(\bbC)
	\]
For each $i\in\bbN_0$, define $v_i := 1\otimes (e_{ik, \ell} + e_{ik+1,\ell+1} + \cdots + e_{i(k+1)-1,\ell+k-1})\otimes 1$, so that $v_i^*v_i = p$ and
	\[
		v_iv_i^* = 1\otimes (e_{ik,ik} + e_{ik+1,ik+1} + \cdots + e_{(i+1)k-1, (i+1)k -1})\otimes 1.
	\]
Observe that $\sum_i v_iv_i^*=1$. Now, $pMp$ is generated by $p[\cB(\cH)\otimes (\bbC\oplus\bbC)]p$ and $\{v_i^* Y_\lambda v_j\colon i,j\in \bbN_0,\ \lambda\in H\}$. Moreover, from the matricial model it is even clear that
	\begin{align*}
		(pMp,\vphi^p)\cong \left( p[\cB(\cH)\otimes (\bbC\oplus\bbC)]p,\Phi^p\right) * \bigast_{\substack{i,j\in \bbN_0\\ \lambda\in H}}  \left(W^*( v_i^*Y_\lambda v_j),\Phi^p\right),
	\end{align*}
For $i,j\in \bbN_0$ and $\lambda\in H$ we have
	\[
		(v_i^* Y_\lambda v_j,\Phi^p) \stackrel{d}{\sim} \left(\sqrt{\lambda_0^i} y_{\lambda\lambda_0^{j-i}}, \vphi_{\lambda\lambda_0^{j-i}}\right),
	\]
where $\stackrel{d}{\sim} $ means equality in distribution.  Since $\< \lambda \lambda_0^{j-i}\colon \lambda \in H,\ i,j\in \bbN_0\> = H$, we have
	\[
		\bigast_{\substack{i,j\in \bbN_0\\ \lambda\in H}} \left( W^*( v_i^*Y_\lambda v_j), \Phi^p\right) \cong \bigast_{\substack{i,j\in \bbN_0\\ \lambda\in H}} \left(T_{\lambda\lambda_0^{j-i}},\vphi_{\lambda\lambda_0^{j-i}}\right) \cong (T_H,\vphi_H). 
	\]
Noting that 
	\[
		\left( p[\cB(\cH)\otimes (\bbC\oplus\bbC)]p,\Phi^p\right) \cong (M_k(\bbC),\psi_{\lambda_0})\otimes (\underset{t}{\bbC}\oplus \underset{1-t}{\bbC}),
	\]
the desired isomorphism then follows from Lemma~\ref{lem:mini_absorption}.

Next, assume that $\vphi(p)$ is not of the above form and that there exists $k\in \bbN_0$ such that $\vphi(p) < 1 - \lambda_0^{k+1}$. Define for each $\ell\in\bbN_0$ $p_\ell = 1\otimes e_{\ell,\ell}\otimes e_{0,0}$ and $q_\ell=1\otimes e_{\ell,\ell}\otimes e_{1,1}$. Set $r_\ell = q_\ell + 1\otimes [e_{\ell+1,\ell+1}+ \cdots +e_{\ell+k,\ell+k}]\otimes 1 + p_{\ell+k+1}$ which we note is in the centralizer and satisfies
	\[
		\Phi(r_\ell) = (1-t) \lambda_0^\ell(1-\lambda_0) + \lambda_0^{\ell+1}(1-\lambda_0) + \cdots + \lambda_0^{\ell+k}(1-\lambda_0) + t \lambda_0^{\ell+k+1}(1-\lambda_0).
	\] 
We visualize $r_\ell$ as follows (here $\ell=k=1$):
	
	
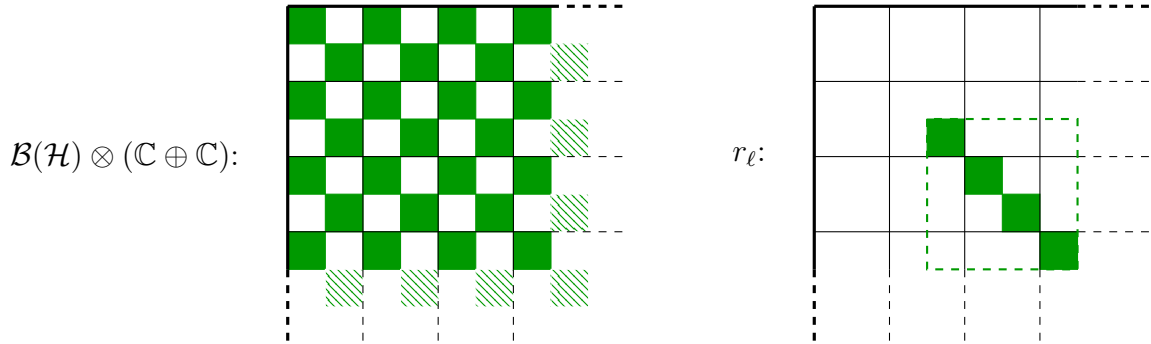
\[
		\begin{tikzpicture}
		\node[left] at (-0.5,-2) {$\cB(\cH) \otimes \left( \bbC\oplus \bbC\right)$:};
		
		\draw[green!60!black, fill= green!60!black] (0,0) rectangle (0.5,-0.5);
		\draw[green!60!black, fill= green!60!black] (0.5,-0.5) rectangle (1,-1);
		\draw[green!60!black, fill= green!60!black] (1,0) rectangle (1.5,-0.5);
		\draw[green!60!black, fill= green!60!black] (1.5,-0.5) rectangle (2,-1);
		\draw[green!60!black, fill= green!60!black] (2,0) rectangle (2.5,-0.5);
		\draw[green!60!black, fill= green!60!black] (2.5,-0.5) rectangle (3,-1);
		\draw[green!60!black, fill= green!60!black] (3,0) rectangle (3.5,-0.5);
		\draw[white, pattern=north west lines, pattern color=green!60!black] (3.5,-0.5) rectangle (4,-1);
		\draw[green!60!black, fill= green!60!black] (0,-1) rectangle (0.5,-1.5);
		\draw[green!60!black, fill= green!60!black] (0.5,-1.5) rectangle (1,-2);
		\draw[green!60!black, fill= green!60!black] (1,-1) rectangle (1.5,-1.5);
		\draw[green!60!black, fill= green!60!black] (1.5,-1.5) rectangle (2,-2);
		\draw[green!60!black, fill= green!60!black] (2,-1) rectangle (2.5,-1.5);
		\draw[green!60!black, fill= green!60!black] (2.5,-1.5) rectangle (3,-2);
		\draw[green!60!black, fill= green!60!black] (3,-1) rectangle (3.5,-1.5);
		\draw[white, pattern=north west lines, pattern color=green!60!black] (3.5,-1.5) rectangle (4,-2);
		\draw[green!60!black, fill= green!60!black] (0,-2) rectangle (0.5,-2.5);
		\draw[green!60!black, fill= green!60!black] (0.5,-2.5) rectangle (1,-3);
		\draw[green!60!black, fill= green!60!black] (1,-2) rectangle (1.5,-2.5);
		\draw[green!60!black, fill= green!60!black] (1.5,-2.5) rectangle (2,-3);
		\draw[green!60!black, fill= green!60!black] (2,-2) rectangle (2.5,-2.5);
		\draw[green!60!black, fill= green!60!black] (2.5,-2.5) rectangle (3,-3);
		\draw[green!60!black, fill= green!60!black] (3,-2) rectangle (3.5,-2.5);
		\draw[white, pattern=north west lines, pattern color=green!60!black] (3.5,-2.5) rectangle (4,-3);	
		\draw[green!60!black, fill= green!60!black] (0,-3) rectangle (0.5,-3.5);
		\draw[white, pattern=north west lines, pattern color=green!60!black] (0.5,-3.5) rectangle (1,-4);
		\draw[green!60!black, fill= green!60!black] (1,-3) rectangle (1.5,-3.5);
		\draw[white, pattern=north west lines, pattern color=green!60!black] (1.5,-3.5) rectangle (2,-4);
		\draw[green!60!black, fill= green!60!black] (2,-3) rectangle (2.5,-3.5);
		\draw[white, pattern=north west lines, pattern color=green!60!black] (2.5,-3.5) rectangle (3,-4);
		\draw[green!60!black, fill= green!60!black] (3,-3) rectangle (3.5,-3.5);
		\draw[white, pattern=north west lines, pattern color=green!60!black] (3.5,-3.5) rectangle (4,-4);
		
		\draw[black] (0,-1) --++ (3.5,0);
		\draw[black, dashed] (3.5,-1) --++ (1,0);
		\draw[black] (0,-2) --++ (3.5,0);
		\draw[black, dashed] (3.5,-2) --++ (1,0);
		\draw[black] (0,-3) --++ (3.5,0);
		\draw[black, dashed] (3.5,-3) --++ (1,0);
		
		\draw[black] (1,0) --++ (0,-3.5);
		\draw[black, dashed] (1,-3.5) --++ (0,-1);
		\draw[black] (2,0) --++ (0,-3.5);
		\draw[black, dashed] (2,-3.5) --++ (0,-1);
		\draw[black] (3,0) --++ (0,-3.5);
		\draw[black, dashed] (3,-3.5) --++ (0,-1);
		
		\draw[very thick] (0,0) --++ (3.5,0);
		\draw[very thick, dashed] (3.5,0) --++ (1,0);
		\draw[very thick] (0,0) --++ (0,-3.5);
		\draw[very thick,dashed] (0,-3.5) --++ (0,-1);

		\node[left] at (6.5,-2) {$r_\ell$:};

		\draw[green!60!black, fill= green!60!black] (8.5,-1.5) rectangle (9,-2);
		\draw[green!60!black, fill= green!60!black] (9,-2) rectangle (9.5,-2.5);
		\draw[green!60!black, fill= green!60!black] (9.5,-2.5) rectangle (10,-3);
		\draw[green!60!black, fill= green!60!black] (10,-3) rectangle (10.5,-3.5);

		\draw[black] (7,-1) --++ (3.5,0);
		\draw[black, dashed] (10.5,-1) --++ (1,0);
		\draw[black] (7,-2) --++ (3.5,0);
		\draw[black, dashed] (10.5,-2) --++ (1,0);
		\draw[black] (7,-3) --++ (3.5,0);
		\draw[black, dashed] (10.5,-3) --++ (1,0);

		\draw[black] (8,0) --++ (0,-3.5);
		\draw[black, dashed] (8,-3.5) --++ (0,-1);
		\draw[black] (9,0) --++ (0,-3.5);
		\draw[black, dashed] (9,-3.5) --++ (0,-1);
		\draw[black] (10,0) --++ (0,-3.5);
		\draw[black, dashed] (10,-3.5) --++ (0,-1);
		
		\draw[very thick] (7,0) --++ (3.5,0);
		\draw[very thick, dashed] (10.5,0) --++ (1,0);
		\draw[very thick] (7,0) --++ (0,-3.5);
		\draw[very thick,dashed] (7,-3.5) --++ (0,-1);
		
		\draw[thick, green!60!black, dashed] (8.5,-1.5) rectangle (10.5,-3.5);
		\end{tikzpicture}
	\]
Observe that $\Phi(r_\ell) \in \left( \lambda_0^{\ell+1}(1-\lambda_0^{k+1}), \lambda_0^\ell(1-\lambda_0^{k+1})\right)$. We can find $\ell\in \bbN_0$ such that
	\[
		\vphi(p)\in \left( \lambda_0^{\ell+1}(1-\lambda_0^{k+1}), \lambda_0^\ell(1-\lambda_0^{k+1})\right),
	\]  
and hence (adjusting our entire matricial model) we can pick $t\in (0,1)$ such that $\Phi(r_\ell)=\varphi(p)$. Then we may assume---without loss of generality---that $p=r_\ell$. Now, define
	\begin{align*}
		u_i := 1\otimes\left[e_{i(k+1),\ell}\otimes e_{1,1} \right.&+ \left(e_{i(k+1)+1,\ell+1}+ \cdots + e_{(i+1)(k+1)-1,\ell+k}\right)\otimes 1\\
			 &\left. +e_{(i+1)(k+1),\ell+k+1}\otimes e_{0,0}\right]
	\end{align*}
so that $u_i^*u_i=p$ and
	\[
		u_iu_i^* = q_{i(k+1)} + 1\otimes \left(e_{i(k+1)+1, i(k+1)+1} + \cdots +e_{(i+1)(k+1) - 1,(i+1)(k+1)-1}\right)\otimes 1 + p_{(i+1)(k+1)}.
	\]
Set $u_{-1}:= 1\otimes e_{0,\ell+k+1}\otimes e_{0,0}$.  We visualize the family $\{u_i\}_{i=-1}^\infty$ as follows (here $\ell=k=1$):
	
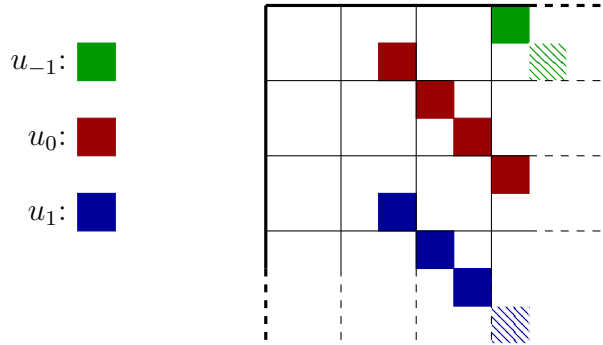
\[
		\begin{tikzpicture}
		\draw[green!60!black, fill= green!60!black] (-2.5,-.5) rectangle (-2,-1);
		\node[left] at (-2.5,-.8) {$u_{-1}$:};

		\draw[red!60!black, fill= red!60!black] (-2.5,-1.5) rectangle (-2,-2);
		\node[left] at (-2.5,-1.8) {$u_{0}$:};

		\draw[blue!60!black, fill= blue!60!black] (-2.5,-2.5) rectangle (-2,-3);
		\node[left] at (-2.5,-2.8) {$u_{1}$:};

		\draw[green!60!black, fill= green!60!black] (3,0) rectangle (3.5,-0.5);
		\draw[white, pattern=north west lines, pattern color=green!60!black] (3.5,-0.5) rectangle (4,-1);

		\draw[red!60!black, fill= red!60!black] (1.5,-.5) rectangle (2,-1);
		\draw[red!60!black, fill= red!60!black] (2,-1) rectangle (2.5,-1.5);
		\draw[red!60!black, fill= red!60!black] (2.5,-1.5) rectangle (3,-2);
		\draw[red!60!black, fill= red!60!black] (3,-2) rectangle (3.5,-2.5);

		\draw[blue!60!black, fill= blue!60!black] (1.5,-2.5) rectangle (2,-3);
		\draw[blue!60!black, fill= blue!60!black] (2,-3) rectangle (2.5,-3.5);
		\draw[blue!60!black, fill= blue!60!black] (2.5,-3.5) rectangle (3,-4);
		\draw[white, pattern=north west lines, pattern color=blue!60!black] (3,-4) rectangle (3.5,-4.5);
		
		\draw[black] (0,-1) --++ (3.5,0);
		\draw[black, dashed] (3.5,-1) --++ (1,0);
		\draw[black] (0,-2) --++ (3.5,0);
		\draw[black, dashed] (3.5,-2) --++ (1,0);
		\draw[black] (0,-3) --++ (3.5,0);
		\draw[black, dashed] (3.5,-3) --++ (1,0);
		
		\draw[black] (1,0) --++ (0,-3.5);
		\draw[black, dashed] (1,-3.5) --++ (0,-1);
		\draw[black] (2,0) --++ (0,-3.5);
		\draw[black, dashed] (2,-3.5) --++ (0,-1);
		\draw[black] (3,0) --++ (0,-3.5);
		\draw[black, dashed] (3,-3.5) --++ (0,-1);
		
		\draw[very thick] (0,0) --++ (3.5,0);
		\draw[very thick, dashed] (3.5,0) --++ (1,0);
		\draw[very thick] (0,0) --++ (0,-3.5);
		\draw[very thick,dashed] (0,-3.5) --++ (0,-1);

		\end{tikzpicture}
	\]
Observe that
	\[
		\sum_{i\in \bbN_0} u_{i-1}u_{i-1}^* = 1.
	\]
Thus, $p M p$ is generated by $p [\cB(\cH) \otimes (\bbC\oplus \bbC)] p$ and $\{u_{i-1}^* Y_\lambda u_{j-1}\colon i,j\in \bbN_0,\ \lambda\in H\}$. For each $i,j\in \bbN_0$ and $\lambda\in H$ define
	\[
		Z(\lambda,i,j):=\begin{cases} u_i^*\left[Y_\lambda - p_{(i+1)(k+1)}Y_\lambda p_{(i+1)(k+1)}\right] u_i + \sqrt{\lambda_0^{k+1}} u_{i-1}^* p_{i(k+1)} Y_\lambda p_{i(k+1)} u_{i-1} & \text{if }i=j\\
								u_i^*\left[Y_\lambda - p_{(i+1)(k+1)} Y_\lambda\right] u_j + \sqrt{\lambda_0^{k+1}} u_{i-1}^* p_{i(k+1)} Y_\lambda u_{j-1} & \text{if }i<j\\
								u_i^*\left[Y_\lambda -  Y_\lambda p_{(j+1)(k+1)}\right] u_j + \sqrt{\lambda_0^{k+1}} u_{i-1}^*  Y_\lambda p_{j(k+1)} u_{j-1} & \text{if }i>j
			\end{cases}.
	\]
We visualize $Z(\lambda,i,j)$ as follows (here $\ell=k=1$):
		\[
		\begin{tikzpicture}
		\draw[lightred, fill=lightred] (0,0) rectangle (2,2);
		\draw[thick, green!60!black, dashed] (0,0) rectangle (2,2);
		\draw[black] (0.5,0) --++ (0,2);
		\draw[black] (1.5,0) --++ (0,2);
		\draw[black] (0,.5) --++ (2,0);
		\draw[black] (0,1.5) --++ (2,0);
		\node[below] at (1,0) {$Z(\lambda, i,j)$};

		\draw [decorate,decoration={brace,amplitude=10pt},xshift=-4pt,yshift=0pt]	(3.25,-4.5) -- (3.25,6.5) node [black,midway,xshift=-0.6cm] 	{ $=$};

		\draw[lightred, fill=lightred] (5,6) rectangle (6.5,4.5);
		\draw[lightred, fill=lightred] (5,4.5) rectangle (6.5,4);
		\draw[lightred, fill=lightred] (6.5,6) rectangle (7,4.5);
		\draw[thick, green!60!black, dashed] (5,4) rectangle (7,6);
		\draw[black] (5.5,6) --++ (0,-2);
		\draw[black] (6.5,6) --++ (0,-2);
		\draw[black] (5,4.5) --++ (2,0);
		\draw[black] (5,5.5) --++ (2,0);
		\node[below] at (6,3.9) {\footnotesize$u_i^*[Y_\lambda-p_{(i+1)(k+1)} Y_\lambda p_{(i+1)(k+1)} ]u_i$};

		\node at (9,5) {$+$};
		
		\draw[lightred, fill=lightred] (12.5,4) rectangle (13,4.5);
		\draw[black] (11.5,6) --++ (0,-2);
		\draw[black] (12.5,6) --++ (0,-2);
		\draw[black] (11,4.5) --++ (2,0);
		\draw[black] (11,5.5) --++ (2,0);
		\draw[thick, green!60!black, dashed] (11,4) rectangle (13,6);
		\node[below] at (12,4.05) {\footnotesize$\sqrt{\lambda_0^{k+1}} u_{i-1}^* p_{i(k+1)} Y_\lambda p_{i(k+1)} u_{i-1}$};

		\node at (15,5) {if $i=j$};

		\draw[lightred, fill=lightred] (5,0.5) rectangle (7,2);
		\draw[black] (5.5,2) --++ (0,-2);
		\draw[black] (6.5,2) --++ (0,-2);
		\draw[black] (5,1.5) --++ (2,0);
		\draw[black] (5,.5) --++ (2,0);
		\draw[thick, green!60!black,dashed] (5,0) rectangle (7,2);
		\node[below] at (6,-0.1) {\footnotesize$u_i^*\left[Y_\lambda - p_{(i+1)(k+1)} Y_\lambda\right] u_j$};

		\node at (9,1) {$+$};

		\draw[lightred, fill=lightred] (11,0) rectangle (13,0.5);
		\draw[black] (11.5,2) --++ (0,-2);
		\draw[black] (12.5,2) --++ (0,-2);
		\draw[black] (11,1.5) --++ (2,0);
		\draw[black] (11,.5) --++ (2,0);
		\draw[thick, green!60!black, dashed] (11,0) rectangle (13,2);
		\node[below] at (12,0.05) {\footnotesize$\sqrt{\lambda_0^{k+1}} u_{i-1}^* p_{i(k+1)} Y_\lambda u_{j-1}$};

		\node at (15,1) {if $i<j$};

		\draw[lightred, fill=lightred] (5,-4) rectangle (6.5,-2);
		\draw[black] (5.5,-2) --++ (0,-2);
		\draw[black] (6.5,-2) --++ (0,-2);
		\draw[black] (5,-2.5) --++ (2,0);
		\draw[black] (5,-3.5) --++ (2,0);
		\draw[thick, green!60!black,dashed] (5,-4) rectangle (7,-2);
		\node[below] at (6,-4.1) {\footnotesize$u_i^*\left[Y_\lambda -  Y_\lambda p_{(j+1)(k+1)}\right] u_j$};

		\node at (9,-3) {$+$};

		\draw[lightred, fill=lightred] (12.5,-4) rectangle (13,-2);
		\draw[black] (11.5,-2) --++ (0,-2);
		\draw[black] (12.5,-2) --++ (0,-2);
		\draw[black] (11,-2.5) --++ (2,0);
		\draw[black] (11,-3.5) --++ (2,0);
		\draw[thick, green!60!black, dashed] (11,-4) rectangle (13,-2);
		\node[below] at (12,-3.95) {\footnotesize$\sqrt{\lambda_0^{k+1}} u_{i-1}^*  Y_\lambda p_{j(k+1)} u_{j-1}$};

		\node at (15,-3) {if $i>j$};
		\end{tikzpicture}
	\]
Then we have that
	\[
		(pMp,\vphi^p)\cong 	\left( p[\cB(\cH) \otimes (\bbC\oplus \bbC)] p,\Phi^p\right) * \bigast_{\substack{i,j\in \bbN_0\\ \lambda\in H}} \left( W^*(Z(\lambda,i,j)),\Phi^p\right).
	\]
Moreover,
	\[
		\left( Z(\lambda,i,j),\Phi^p\right) \stackrel{d}{\sim} \left( u_i^* Y_\lambda u_j,\Phi^p\right) \stackrel{d}{\sim} \left(\sqrt{\lambda_0^{i(k+1)}} y_{\lambda \lambda_0^{(j-i)(k+1)}}, \vphi_{\lambda\lambda_0^{(j-i)(k+1)}}\right).
	\]
Since $\< \lambda\lambda_0^{(j-i)(k+1)}\colon i,j\in \bbN_0,\ \lambda\in H\>= H$, we have
	\begin{align*}
		(pMp, \vphi^p) &\cong \left( p[\cB(\cH) \otimes (\bbC\oplus \bbC)] p, \Phi^p\right) * \bigast_{\substack{i,j\in \bbN_0\\ \lambda\in H}} (T_{\lambda\lambda_0^{(j-i)(k+1)}}, \vphi_{\lambda\lambda_0^{(j-i)(k+1)}}) \\
		& \cong \left( p[\cB(\cH) \otimes (\bbC\oplus \bbC)] p, \Phi^p\right) * (T_H,\vphi_H).
	\end{align*}
Noting that
	\[
		\left( p[\cB(\cH) \otimes (\bbC\oplus \bbC)] p, \Phi^p\right) \cong (M_{k+1}(\bbC),\psi_{\lambda_0})\otimes (\underset{t}{\bbC}\oplus \underset{1-t}{\bbC}),
	\]
the desired isomorphism then follows from Lemma~\ref{lem:mini_absorption}.
\end{proof}

\begin{lem}\label{lem:antman2}
Let $M$ be a von Neumann algebra with almost-periodic faithful normal state $\vphi$. Let $p\in M^\vphi$ a projection such that $(pMp, \vphi^p)\cong (T_H,\vphi_H)$ for some non-trivial, countable subgroup $H$ of $\R^+$, and such that $z:=z(p\colon M^\vphi)=z(p\colon M)$. Then
	\[
		(z M,\vphi^{z})\cong (T_H,\vphi_H).
	\]
In particular, if $z(p\colon M^\vphi)=1$ then $(M,\vphi)\cong (T_H,\vphi_H)$.
\end{lem}
\begin{proof}
Observe that since $(M^\vphi,\vphi)$ is a tracial von Neumann algebra containing $p$, there are state preserving isomorphisms
	\[
		(z M^\vphi, \vphi^z) \cong (p M^\vphi p ,\vphi^p) = ( (pMp)^{\vphi^p}, \vphi^p) \cong (T_H^{\vphi_H},\vphi_H)\cong (L(\mathbb{F}_\infty), \tau)
	\]
Consequently, $(zM^\vphi, \vphi^z)$ is a factor. Furthermore, $z M^\vphi = (z M)^{\vphi^z}$. Thus, replacing $M$ with $zM$, we may assume $z=1$ and that $M^\vphi$ is a factor.

Fix $\lambda\in H\cap (0,1)$. By the isomorphism with $(T_H,\vphi_H)$, $\lambda$ is in the point spectrum of $\Delta_{\vphi^p}$ and hence in the point spectrum of $\Delta_\vphi$. By \cite[Lemma 4.9]{Dyk97} there exists an isometry $v\in M$ which is an eigenoperator with eigenvalue $\lambda$. Let $k\in \N$ be such that $\lambda^k \leq \vphi(p)$. Then $q=v^k(v^*)^k\in M^\vphi$ with $\vphi(q)=\lambda^k$. Since $M^\vphi$ is a factor, $q$ can be conjugated to a projection $p_0$ under $p$ via an partial isometry $w\in M^\vphi$. Thus
	\[
		(M,\vphi)\cong (qMq,\vphi^q)\cong (p_0 Mp_0, \vphi^{p_0}).
	\]
But then $p_0 Mp_0$ is a compression of $pMp$, so by the previous lemma there is a state-preserving isomorphism with $(T_H,\vphi_H)$.
\end{proof}

\begin{rem}
Observe that by Remark~\ref{rem:central_support_same_in_centralizer}, the hypothesis $z(p\colon M^\vphi)=z(p\colon M)$ in the above lemma holds automatically for $M=\cM(\Gamma,\mu)$ and $p=u_eu_e^*$ for any $e\in E$.
\end{rem}

\begin{cor}\label{cor:absorption}
Let $\lambda,\lambda_1,\ldots, \lambda_d\in (0,1)$ and let $H=\<\lambda,\lambda_1,\ldots, \lambda_d\>< \bbR^+$. Let $\cH_1,\ldots, \cH_d$ be separable Hilbert spaces equipped with respective states $\psi_{\lambda_1},\ldots, \psi_{\lambda_d}$ (each relative to some choice of matrix units). For any $t_1,\ldots, t_d>0$ with $t_1+\cdots +t_d=1$ we have
	\[
		(T_\lambda,\vphi_\lambda) * \left[ \underset{t_1}{(\cB(\cH_1),\psi_{\lambda_1})} \oplus \cdots \oplus \underset{t_d}{(\cB(\cH_d), \psi_{\lambda_d})}\right] \cong (T_H,\psi_H).
	\]
\end{cor}
\begin{proof}
We proceed by induction on $d$. When $d=1$ (in which case $t_1=1$), we have by \cite[Theorem 4.8]{Shl97} that
	\[
		(L(\bbZ),\tau) * (\cB(\cH_1),\psi_{\lambda_1})\cong (T_{\lambda_1}, \vphi_{\lambda_1}).
	\]
So using free absorption we obtain
	\begin{align*}
		(T_\lambda,\vphi_\lambda) * (\cB(\cH_1),\psi_{\lambda_1}) &\cong (T_\lambda,\vphi_\lambda) * (L(\bbZ),\tau) * (\cB(\cH_1),\psi_{\lambda_1})\\
			&\cong (T_\lambda, \vphi_\lambda) * (T_{\lambda_1}, \vphi_{\lambda_1}) \cong (T_H, \lambda_H),
	\end{align*}
where $H=\< \lambda,\lambda_1\> < \bbR^+$.

Now, suppose the result holds for $d-1$. Let $t_1+\cdots+t_d=1$. Denote
	\[
		(A,\psi):=\underset{\frac{t_1}{t_1+\cdots +t_{d-1}}}{(\cB(\cH_1),\psi_{\lambda_1})}\oplus\cdots \oplus \underset{\frac{t_{d-1}}{t_1+\cdots+t_{d-1}}}{(\cB(\cH_{d-1}),\psi_{\lambda_{d-1}})}.
	\]
Consider the following intermediate algebras:
	\begin{align*}
		(N_1,\vphi_1)&:= (T_\lambda,\vphi_\lambda) * \left[\overset{p}{\underset{t_1+\cdots+t_{d-1}}{\bbC}}\oplus \overset{q}{\underset{t_d}{\bbC}}\right]\\
		(N_2,\vphi_2)&:=(T_\lambda,\vphi_\lambda) * \left[ \overset{p}{\underset{t_1+\cdots+t_{d-1}}{(A,\psi)}}\oplus \overset{q}{\underset{t_d}{\bbC}}\right]\\
		(M,\vphi) &:= (T_\lambda,\vphi_\lambda) * \left[ \overset{p}{\underset{t_1+\cdots+t_{d-1}}{(A,\psi)}}\oplus \overset{q}{\underset{t_d}{(\cB(\cH_d),\psi_{\lambda_d})}}\right].
	\end{align*}
Note that $p$ and $q$ are central projections in the right factors of each free product. Using free absorption, (the proof of) \cite[Lemma 1.6]{MR1201693}, and free absorption again we obtain $(N_1,\vphi_1)\cong (T_\lambda, \vphi_\lambda)$. Then by Lemma~\ref{lem:Dykema} we have that
	\[
		(pN_2p,\vphi_2^p) \cong (p N_1 p,\vphi_1^p) * (A,\psi).
	\]
By Lemma~\ref{lem:antman1}, $(p N_1 p,\vphi_1^p) \cong (T_\lambda,\vphi_\lambda)$. So by the induction hypothesis, $(p N_2p ,\vphi_2^p) \cong (T_K,\vphi_K)$ where $K=\<\lambda,\lambda_1,\ldots, \lambda_{d-1}\> \leq H$. From \cite[Corollary 7.1]{Nel17} we see that $N_2^{\vphi_2}$ is a factor, and so $(N_2,\vphi_2)\cong (T_K,\vphi_K)$ by Lemma~\ref{lem:antman2}.

Appealing to Lemma~\ref{lem:Dykema} again yields
	\begin{align*}
		(qMq,\vphi^q) &\cong (q N_2q, \vphi_2^q) * (\cB(\cH), \psi_{\lambda_d})\\
			&\cong (T_K,\vphi_K) * (\cB(\cH), \psi_{\lambda_d}) \cong (T_H,\vphi_H),			
	\end{align*}
where we have used Lemma~\ref{lem:antman1} and the same argument as in the base case (along with $H=\<K,\lambda\>$). Once more, \cite[Corollary 7.1]{Nel17} implies that $M^\vphi$ is a factor, and so Lemma~\ref{lem:antman2} concludes the proof.
\end{proof}


\section{Building the Graph}\label{sec:building_the_graph}

In this section we prove our main result. Let $\Gamma=(V, E)$ and $\mu$ be as in Section~\ref{sec:adding_weighting_to_edges}. Our strategy is to build up to an isomorphism for $\cM(\Gamma,\mu)$ by first establishing isomorphisms for a certain special subgraph $\Gamma_0$ and then constructing $\Gamma$ by successively adding edges (and sometimes vertices) and controlling the isomorphism at each step of this construction. What makes $\Gamma_0$ special is that $\cM(\Gamma_0,\mu)$ is a free Araki--Woods factor (up to direct sum with a finite dimensional abelian algebra). The robust absorption properties of the free Araki--Woods factors that we have seen will then ensure that the intermediate graphs between $\Gamma_0$ and $\Gamma$ will also have graph von Neumann algebras that are free Araki--Woods factors (up to direct sum with a finite dimensional abelian algebra).

To begin, we first determine $\Gamma_0$ as follows. Since $H$ is non-trivial, there exists $\sigma_0\in \Lambda_\Gamma$ such that $\mu(\sigma)\neq 1$. By taking such a loop $\sigma_0=e_1\cdots e_n$ with minimal length, we may further assume that that $\sigma_0$ visits exactly $n-1$ vertices. Furthermore, by considering the reverse loop if necessary, we may assume $\mu(\sigma_0)>1$. Finally, by the following proposition, up to a cyclic relabeling of the edges we may assume that $\mu(e_1)\cdots \mu(e_k)\geq 1$ for each $k=1,\ldots, n$ (which will be necessary in Lemma~\ref{lem:base_case3}). Then $\Gamma_0=(V_0,E_0)$ is defined as the subgraph consisting of $\sigma_0$ and $\sigma_0^{\op}$.

\begin{prop}
Let $\mu_{1},\ldots, \mu_{n}$ be positive real numbers satisfying $\mu_{1}\cdots\mu_{n} \geq 1$.  Then there exists a cyclic permutation $\sigma \in S_{n}$ satisfying $\mu_{\sigma(1)}\cdots\mu_{\sigma(k)} \geq 1$ for each $k=1,\ldots, n$.
\end{prop}

\begin{proof}
Define $\alpha = \min\{ \mu_{\sigma(1)}\cdots \mu_{\sigma(j)} \, | 1 \leq j \leq n , \, \sigma \in S_{n} \text{ cyclic }\}$.  If $\alpha \geq 1$, then clearly any cyclic permutation will satisfy $\mu_{\sigma(1)}\cdots\mu_{\sigma(k)} \geq 1$ for all $k \in \{1, \dots, n\}$.  Therefore, we suppose $\alpha < 1$.

Choose $m \in \{1, \cdots, n\}$ and cyclic $\sigma \in S_{n}$ satisfying $\mu_{\sigma(m)}\cdots\mu_{\sigma(n)} = \alpha$.  We claim that $\mu_{\sigma(1)}\cdots\mu_{\sigma(k)} \geq 1$ for all $k \in \{1, \dots, n\}$.  Indeed, if $k < m$, and $\mu_{\sigma(1)}\cdots\mu_{\sigma(k)} < 1$, then $\mu_{\sigma(m)}\cdots\mu_{\sigma(n)}\mu_{\sigma(1)}\cdots\mu_{\sigma(k)} < \alpha$, contradicting the minimality of $\alpha$.  Furthermore, note that if $k \geq m$, then $\mu_{\sigma(k+1)}\cdots \mu_{\sigma(n)} \leq 1$, for otherwise we would have $\mu_{\sigma(m)}\cdots \mu_{\sigma(k)} <  \mu_{\sigma(m)}\cdots \mu_{\sigma(n)} = \alpha$, contradicting the minimality of $\alpha$.  Since  $\mu_{1}\cdots\mu_{n} \geq 1$, it follows that $\mu_{\sigma(1)}\cdots\mu_{\sigma(k)} \geq 1$, as desired.
\end{proof}

Now, we define a positive linear functional $\vphi$ on $\cM(\Gamma,\mu)$ via the following procedure. Let $\sigma_0=e_1\cdots e_n$ be as above. As in Section~\ref{sec:adding_weighting_to_edges},  let $\Gamma_{\Tr}$ be a maximal subgraph of $\Gamma$ containing $e_1,\ldots, e_{n-1}$ and satisfying
	\[
		\mu(\sigma)=1\qquad \forall \sigma\in \Lambda_{\Gamma_{\Tr}} 
	\]
Then define $\vphi$ on $\cM(\Gamma,\mu)$ as in Section~\ref{sec:adding_weighting_to_edges} by setting $*:=s(\sigma_0)$; that is, $\vphi(p_{s(\sigma_0)}):=1$ and for $v\in V\setminus \{s(\sigma_0)\}$ set $\vphi(v):= \mu(\sigma)$ where $\sigma\in \Pi_{\Gamma_{\Tr}}$ satisfies $s(\sigma)=s(\sigma_0)$ and $t(\sigma)=v$. Extend $\vphi$ to $\cM(\Gamma,\mu)$ by letting $\vphi=\vphi\circ\bbE$. Note that since $e_1\cdots e_{n-1}\in \Pi_{\Gamma_{\Tr}}$, Proposition~\ref{prop:edges_elements_are_eigenops} implies that $Y_{e_n}$ is an eigenoperator with respect to $\varphi$ with eigenvalue $\mu(e_1)\cdots \mu(e_n)>1$.

We have carefully chosen $\Gamma_{\Tr}$ above for the purposes of our analysis of $(\cM(\Gamma_0,\mu),\vphi)$. However, as we add edges to construct $\Gamma$ from $\Gamma_0$, we might have to alter the choice of $\Gamma_{\Tr}$.  The following proposition allows us to do so.

\begin{prop}\label{prop:changestate}
Let $\Gamma_{\Tr_{1}}$ and $\Gamma_{\Tr_{2}}$ both be maximal among subgraphs $\Xi$ of $\Gamma$ satsifying $\mu(\sigma)=1$ for all $\sigma\in \Lambda_{\Xi}$.  Let $\vphi_{1}$ and $\vphi_{2}$ be the induced positive linear functionals (see Section~\ref{sec:adding_weighting_to_edges}), and assume $H:=H(\Gamma,\mu)$ is non-trivial.  Set $\displaystyle\gamma_{v} := \sum_{s(f) = v} \mu(f)$.  If
	\[
		(\cM(\Gamma, \mu), \vphi_{1}) = \overset{\bigvee_{e \in E} u_{e}u_{e}^{*}}{(T_{H}, \vphi_{H})} \oplus \bigoplus_{v\colon\gamma_{v} < 1} \underset{\vphi_{1}(p_{v})(1-\gamma_{v})}{\C},
	\]
then
	\[
		(\cM(\Gamma, \mu), \vphi_{2}) = \overset{\bigvee_{e \in E} u_{e}u_{e}^{*}}{(T_{H}, \vphi_{H})} \oplus \bigoplus_{v\colon\gamma_{v} < 1} \underset{\vphi_{2}(p_{v})(1-\gamma_{v})}{\C}.
	\]
\end{prop}
\begin{proof}
Choose $v \in V$ satisfying $\gamma_{v} \geq 1$ (so that $p_{v} \leq \bigvee_{e \in E} u_{e}u_{e}^{*}$).  Lemma~\ref{lem:antman1} implies that $(p_v\cM(\Gamma, \mu)p_v, \vphi_{1}^{p_v}) \cong (T_{H}, \vphi_{H})$.  Let $\phi$ be any positive linear functional on $\cM(\Gamma, \mu)$ preserving $\bbE$.  It is straightforward to see that the joint law of $\{Y_{e_{1}}\cdots Y_{e_{n}} : s(e_{1}) = t(e_{n}) = v\}$ in $(p_{v}\cM(\Gamma, \mu)p_{v}, \phi^{p_v})$ depends only on $\mu$ and not on $\phi$.  It follows that $(p_v\cM(\Gamma, \mu)p_v, \vphi_{2}^{p_v}) \cong (T_{H}, \vphi_{H})$.  The rest follows from Lemma~\ref{lem:antman2} as well as Lemma~\ref{lem:centralizersupport} and Remark~\ref{rem:atomic}.
\end{proof}

In the following Subsection we consider the various possibilities (and some additional relevant cases) for $\Gamma_0$. We show that in each case $\cM(\Gamma_0,\mu)$ is a free Araki--Woods factor (up to direct sum with a finite dimensional abelian algebra). Later, in Subsection~\ref{subsec:under_construction}, we demonstrate how to build up from this subgraph while controlling the (state-preserving) isomorphism class.


\subsection{Establishing a foundation}

$\hfill$\textit{The wheel turns and a black hole is born.}

{\color{white}a}

We consider $\Gamma_0= (V_0, E_0)$ as a subgraph of $\Gamma$ and therefore $\cM(\Gamma_0,\mu)$ as a subalgebra of $\cM(\Gamma,\mu)$, which we endow with the positive linear function $\vphi$ defined as above. The following lemma only corresponds to a possible case for $\Gamma_0$ when $n=1$, but we are able to prove a more general version here with no additional effort.

\begin{lem}\label{lem:base_case0}
Let $\Gamma_0=(V_0,E_0)$, $\mu$, and $\vphi$ satisfy
	
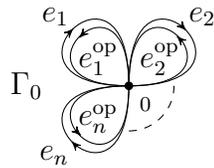
\[
		\begin{tikzpicture}
		\node[left] at (-1,0) {$\Gamma_0$};
		
		\node[below right] at (0,0) {\scriptsize$0$};
		\draw[fill=black] (0,0) circle (.05);
		
		\node[above left] at (-0.65,0.65) {$e_1$};
		\node[above left] at (0,0) {$e_1^{\op}$};
		
		\node[above right] at (0.65,0.65) {$e_2$};
		\node[above right] at (0,0) {$e_2^{\op}$};
		
		\node[below left] at (-0.65,-0.65) {$e_n$};
		\node[below left] at (0,-0.1) {$e_n^{\op}$};
		
		\draw (0,0) .. controls (-2,0) and (0,2) .. (0,0);
		\draw[->] (-0.76,0.74) -- (-0.74,0.76);
		\draw (0,0) .. controls (-1.75,0) and (0,1.75) .. (0,0);
		\draw[->] (-0.64,0.66) -- (-0.66,0.64); 
		\draw (0,0) .. controls (2,0) and (0,2) .. (0,0);
		\draw[->]  (0.74,0.76) -- (0.76,0.74);
		\draw (0,0) .. controls (1.75,0) and (0,1.75) .. (0,0);
		\draw[->] (0.66,0.64) -- (0.64,0.66);
		\draw (0,0) .. controls (-2,0) and (0,-2) .. (0,0);
		\draw[->]  (-0.74,-0.76) -- (-0.76,-0.74);
		\draw (0,0) .. controls (-1.75,0) and (0,-1.75) .. (0,0);
		\draw[->] (-0.66,-0.64) -- (-0.64,-0.66);

		\draw[dashed] (0,-0.6) arc (-90:0:0.6);

		\node[right] at (2,0.5) {$\bullet$ $V_0=\{0\}$ with $\vphi(p_0)=1$};
		
		\node[right] at (2,-0.5) {$\bullet$ $E_0=\{e_1,e_1^{\op},\ldots, e_n,e_n^{\op}\}$ with $s(e_i)=t(e_i)=0$};
		
		\end{tikzpicture}
	\]
Let $H=\< \mu(e_i)\colon i=1,\ldots, n\> < \R^+$. Then
	\[
		(\cM(\Gamma_0,\mu),\vphi) \cong (T_H, \vphi_H).
	\]
\end{lem}
\begin{proof}
With respect to $\vphi$, $Y_{e_1},\ldots, Y_{e_n}$ are freely independent generalized circular elements with eigenvalues $\mu(e_1),\ldots, \mu(e_n)$, respectively. Hence
	\[
		(\cM(\Gamma_0,\mu),\vphi) \cong (T_{\mu(e_1)},\vphi_{\mu(e_1)}) *\cdots * (T_{\mu(e_n)}, \vphi_{\mu(e_n)}) \cong (T_H,\vphi_H). \qedhere
	\]
\end{proof}

\begin{lem}\label{lem:base_case1}
Let $\Gamma_0=(V_0,E_0)$, $\mu$, and $\vphi$ satisfy
	
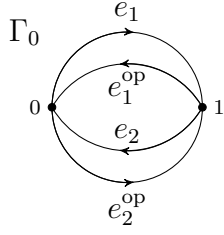
\[
		\begin{tikzpicture}
		\node[left] at (-1,1) {$\Gamma_0$};
		
		\node[left] at (-1,0) {\scriptsize$0$}; 
		\draw[fill=black] (-1,0) circle (0.05);
		
		\node[right] at (1,0) {\scriptsize$1$}; 
		\draw[fill=black] (1,0) circle (0.05);
		
		\node[below] at (0, 0.62) {$e_1^{\op}$}; 
		\draw (-1,0) arc (150:30:1.1547);
		\draw[->] (1,0) arc (30:95:1.1547);

		\node[above] at (0,1) {$e_1$}; 
		\draw (-1,0) arc (180:0:1);
		\draw[->] (-1,0) arc (180:85:1);
		
		\node[above] at (0,-0.62) {$e_2$}; 
		\draw (1,0) arc (-30:-150:1.1547);
		\draw[->] (1,0) arc (-30:-95:1.1547);

		\node[below] at (0,-1) {$e_2^{\op}$}; 
		\draw (-1,0) arc (-180:0:1);
		\draw[->] (-1,0) arc (-180:-85:1);

		\node[right] at (2,1) {$\bullet$ $V_0=\{0,1\}$ with $\vphi(p_0)=1$ and $\vphi(p_1)=\mu(e_1)$};
		
		\node[right] at (2,0) {$\bullet$ $E_0=\{e_1,e_1^{\op}, e_2, e_2^{\op}\}$,  with $s(e_1)=t(e_2)=0$ and $t(e_1)=s(e_2)=1$};
		
		\node[right] at (2,-1) {$\bullet$ If $\mu_i=\mu(e_i)$, $i=1,2$, then $\mu_1>1$};
	
		\end{tikzpicture}
	\]
Let $\lambda=\frac{1}{\mu_1\mu_2}$ , and $z = \bigvee_{e \in E(\Gamma_{0})}u_{e}^{*}u_{e}$. Then
	\[
		(\cM(\Gamma_0,\mu),\vphi) \cong \begin{cases} (T_\lambda,\vphi_\lambda) &\text{if }\mu_1^{-1}+\mu_2\geq 1\\
														(T_\lambda,\vphi_\lambda)\oplus \underset{\mu_1[1-\mu_1^{-1}-\mu_2]}{\overset{r_1}{\bbC}} &\text{otherwise}\end{cases},
	\]
where $r_1 = 1 - z \leq p_{1}$.
\end{lem}
\begin{proof}
Let $u_{i}$ be the polar part of $Y_{e_{i}}$ and denote $D:=\overset{p_0}{\bbC}\oplus \overset{p_1}{\bbC}$. Notice that $\cM(\Gamma_0,\mu)$ must be the following:
	\[
		\cM(\Gamma_0,\mu) \cong \begin{cases}
\left(\underset{1, 1}{\overset{p_{0}, u_{1}^{*}u_{1}}{M_{2}(L(\Z))}} \oplus \underset{\mu_{1}-1}{\overset{p_{1} - u_{1}^{*}u_{1}}{\C}}\right) \underset{D}{\Asterisk} \left(\underset{1-\frac{1}{\mu_{2}}}{\overset{p_{0} - u_{2}^{*}u_{2}}{\C}} \oplus \underset{\frac{1}{\mu_{2}}, \mu_{1}}{\overset{u_{2}^{*}u_{2}, p_{1}}{M_{2}(L(\Z))}}\right) &\text{ if } \mu_{2} \geq 1  \\
\left(\underset{1, 1}{\overset{p_{0}, u_{1}^{*}u_{1}}{M_{2}(L(\Z))}} \oplus \underset{\mu_{1}-1}{\overset{p_{1} - u_{1}^{*}u_{1}}{\C}}\right) \underset{D}{\Asterisk} \left( \underset{1, \mu_{1}\mu_2}{\overset{p_{0}, u_{2}u_{2}^{*}}{M_{2}(L(\Z))}} \oplus \underset{\mu_{1}-\mu_{1}\mu_{2}}{\overset{p_{1} - u_{2}u_{2}^{*}}{\C}} \right) &\text{ if } \mu_{2} < 1
\end{cases}.		
	\]
Note by \cite[Theorem 1.1]{MR1201693} that we have a trace preserving isomorphism
	\[
		\underset{1, 1}{M_{2}(L(\Z))} \oplus \underset{\mu_{1}-1}{\C} \cong \left[\overset{p_{0}'}{\underset{1}{\C}} \oplus \overset{p_{1}'}{\underset{\mu_{1}}{\C}}\right] *  \left[\underset{1}{\C} \oplus \underset{\mu_{1}}{\C}\right].
	\]
Furthermore, in this isomorphism,
	\[
		p_{0}' = \begin{pmatrix} 1 & 0 \\ 0 & 0\end{pmatrix} \oplus 0 \qquad\text{and}\qquad p_{1}' = \begin{pmatrix} 0 & 0 \\ 0 & 1\end{pmatrix} \oplus 1.
	\]
This means that there exists a trace preserving isomorphism
	\begin{align*}
		\underset{1, 1}{\overset{p_{0}, u_{1}^{*}u_{1}}{M_{2}(L(\Z))}} \oplus \underset{\mu_{1}-1}{\overset{p_{1} - u_{1}^{*}u_{1}}{\C}} &\to \left[\overset{p_{0}'}{\underset{1}{\C}} \oplus \overset{p_{1}'}{\underset{\mu_{1}}{\C}}\right] *  \left[ \underset{1}{\C} \oplus  \underset{\mu_{1}}{\C}\right]\\
		p_0&\mapsto p_0'\\
		p_1&\mapsto p_1',
	\end{align*}
that is, $D$ is freely complemented in this von Neumann algebra.  Consequently, by \cite[Proposition 4.1]{Hou07}, we have
	\[
		\cM(\Gamma_0,\mu) \cong \begin{cases}
\left(\underset{1}{\C} \oplus \underset{\mu_{1}}{\C} \right) * \left(\underset{1-\frac{1}{\mu_{2}}}{\overset{p_{0} - u_{2}^{*}u_{2}}{\C}} \oplus \underset{\frac{1}{\mu_{2}}, \mu_{1}}{\overset{u_{2}^{*}u_{2}, p_{1}}{M_{2}(L(\Z))}}\right) &\text{ if } \mu_{2} \geq 1  \\
\left(\underset{1}{\C} \oplus \underset{\mu_{1}}{\C} \right) * \left( \underset{1, \mu_{1}\mu_2}{\overset{p_{0}, u_{2}u_{2}^{*}}{M_{2}(L(\Z))}} \oplus \underset{\mu_{1}-\mu_{1}\mu_{2}}{\overset{p_{1} - u_{2}u_{2}^{*}}{\C}} \right) &\text{ if } \mu_{2} < 1
\end{cases}.
	\]

\noindent
\underline{\textbf{Case 1:}} Assume $\mu_2 \geq 1$. Consider the following von Neumann subalgebras of $\cM(\Gamma_0,\mu)$:
	\begin{align*}
		(\cN_{1},\vphi) &:= \left(\underset{1}{\C} \oplus \underset{\mu_{1}}{\C} \right) * \left(\underset{1 - \frac{1}{\mu_{2}}}{\overset{p_{0} - u_{2}^{*}u_{2}}{\C}} \oplus \underset{\frac{1}{\mu_{2}} + \mu_{1}}{\overset{P}{\C}} \right)\\
		(\cN_{2},\vphi) &:= \left(\underset{1}{\C} \oplus \underset{\mu_{1}}{\C} \right) * \left(\underset{1 - \frac{1}{\mu_{2}}}{\overset{p_{0} - u_{2}^{*}u_{2}}{\C}} \oplus \underset{\frac{1}{\mu_{2}}, \mu_{1}}{\overset{u_{2}^{*}u_{2}, p_{1}}{M_{2}(\C)}}\right)
	\end{align*}
where $P = p_{1} + u_{2}^{*}u_{2}$.  Note that by \cite[Theorem 1.1]{MR1201693}, we have
	\[
		(P\cN_{1}P,\vphi^P) \cong \underset{\frac{1}{\mu_{2}}}{\C} \oplus \underset{1 - \frac{1}{\mu_{2}}}{L(\Z)} \oplus \underset{\mu_{1} - 1 + \frac{1}{\mu_{2}}}{\C}.
	\]
By Lemma~\ref{lem:Dykema}, we see that 
	\[
		(P\cN_{2}P,\vphi^P) = (P\cN_{1}P,\vphi^P) * \underset{\frac{1}{\mu_{2}}, \mu_{1}}{\overset{u_{2}^{*}u_{2}, p_{1}}{M_{2}(\C)}}.
	\]
Note that there is a trace preserving inclusion of $\underset{\frac{1}{\mu_{2}}}{\C} \oplus \underset{\mu_{1}}{\C}$ into $P\cN_{1}P$.  By \cite[Theorem 4.3]{Hou07}, it follows that 
	\[
		(P\cN_{2}P,\vphi^P) \cong (P\cN_{2}P,\vphi^P) * (T_{\lambda}, \vphi_{\lambda}).
	\]
 Corollary~\ref{cor:absorption} yields $(P\cN_{2}P,\vphi^P) \cong (T_{\lambda}, \vphi_{\lambda})$.  Using Lemma~\ref{lem:antman1}, we see that $(p_{1}\cN_{2}p_{1},\vphi^{p_1}) \cong (T_{\lambda}, \vphi_{\lambda})$.  

Using Lemma~\ref{lem:Dykema} again along with free absorption, we obtain
	\[
		(p_{1}\cM(\Gamma_{0}, \mu)p_{1},\vphi^{p_1}) \cong (p_{1}\cN_{2}p_{1},\vphi^{p_1}) * (L(\Z),\tau) \cong (T_{\lambda}, \vphi_{\lambda}).
	\]
Since $z=z(p_1\colon \cM(\Gamma_0,\mu))$ by Remark~\ref{rem:central_support_same_in_centralizer}, it follows from Lemma~\ref{lem:antman2} that $(z\cM(\Gamma_{0}, \mu),\vphi^z) \cong (T_{\lambda}, \vphi_{\lambda})$. The formula for $(\cM(\Gamma_0,\mu),\vphi)$ then follows from Remark~\ref{rem:atomic}.\\

\noindent
\underline{\textbf{Case 2:}} Assume $\mu_2 < 1$. Consider the following von Neumann subalgebras of $\cM(\Gamma_0,\mu)$:
	\begin{align*}
		(\cN_{1},\vphi) &:= \left(\underset{1}{\C} \oplus \underset{\mu_{1}}{\C} \right) * \left( \underset{1 + \mu_{1}\mu_{2}}{\overset{Q}{\C}} \oplus \underset{\mu_{1} - \mu_{1}\mu_{2}}{\overset{p_{1} - u_{2}u_{2}^{*}}{\C}} \right)\\
		(\cN_{2},\vphi) &:= \left(\underset{1}{\C} \oplus \underset{\mu_{1}}{\C} \right) * \left( \underset{1, \mu_{1}\mu_{2}}{\overset{p_{0}, u_{2}u_{2}^{*}}{M_{2}(\C)}} \oplus \underset{\mu_{1} - \mu_{1}\mu_{2}}{\overset{p_{1} - u_{2}u_{2}^{*}}{\C}}\right) 
	\end{align*}
where $Q = p_{0} + u_{2}u_{2}^{*}$.  Note that by \cite[Theorem 1.1]{MR1201693}, we have
	\[
		(Q\cN_{1}Q,\vphi^Q) \cong \underset{1 + \mu_{1}\mu_{2} - \mu_{1}}{\C} \oplus \underset{\mu_{1} - \mu_{1}\mu_{2}}{L(\Z)} \oplus \underset{\mu_{1}\mu_{2}}{\C}.
	\]
Proceeding---\textit{mutatis mutandis}---as in Case 1 finishes the proof.
\end{proof}

The following represents a higher order version of the previous lemma. It is not necessary as a foundational step, but it will be used in simplifying the proofs in our construction steps. Since the proof utilizes the previous lemma as a base case for an induction argument, we present it presently.

\begin{lem}\label{lem:switcheroo}
Let $\Gamma_0=(V_0,E_0)$, $\mu$, and $\vphi$ satisfy
	
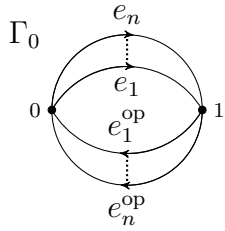
\[
		\begin{tikzpicture}
		\node[left] at (-1,1) {$\Gamma_0$};
		
		\node[left] at (-1,0) {\scriptsize$0$}; 
		\draw[fill=black] (-1,0) circle (0.05);
		
		\node[right] at (1,0) {\scriptsize$1$}; 
		\draw[fill=black] (1,0) circle (0.05);
		
		\node[below] at (0, 0.57735) {$e_1$}; 
		\draw (-1,0) arc (150:30:1.1547);
		\draw[->] (-1,0) arc (150:85:1.1547);
		
		\draw[densely dotted,thick] (0,0.6) -- (0,1);
		
		\node[above] at (0,1) {$e_n$}; 
		\draw (-1,0) arc (180:0:1);
		\draw[->] (-1,0) arc (180:85:1);
		
		\node[above] at (0,-0.57735) {$e_1^{\op}$}; 
		\draw (1,0) arc (-30:-150:1.1547);
		\draw[->] (1,0) arc (-30:-95:1.1547);
		
		\draw[densely dotted, thick] (0,-0.6) -- (0,-1);
		
		\node[below] at (0,-1) {$e_n^{\op}$}; 
		\draw (-1,0) arc (-180:0:1);
		\draw[->] (1,0) arc (0:-95:1);

		\node[right] at (2,1.25) {$\bullet$ $V_0=\{0,1\}$ with $\vphi(p_0)=1$ and $\vphi(p_1)=\mu(e_1)$};
		
		\node[right] at (2,0.25) {$\bullet$ $E_0=\{e_1,e_1^{\op},\ldots, e_n, e_n^{\op}\}$, $n\geq 2$, with $s(e_i)=0$ and $t(e_i)=1$};
		\node[right] at (2,-.25) {{\color{white}$\bullet$} for each $i=1,\ldots, n$}; 
		
		\node[right] at (2,-1.25) {$\bullet$ $\mu(e_1)>1$ and $\frac{\mu(e_i)}{\mu(e_j)} \neq 1$ for at least one pair $i,j=1,\ldots,n$};
		\end{tikzpicture}
	\]
 Then
 	\[
 		(\cM(\Gamma_0, \mu),\vphi) \cong \underset{\mu(e_1)\sum_{i=1}^n \frac{1}{\mu(e_i)}}{(T_H,\vphi_H)}\oplus \underset{\mu(e_1)\left[1- \sum_{i=1}^n \frac{1}{\mu(e_i)}\right]}{\overset{r_1}{\bbC}},
 	\]
where $H=\< \frac{\mu(e_i)}{\mu(e_j)} \colon i,j=1,\ldots, n\>$ and $r_1\leq p_1$ exists if and only if $\sum_{i=1}^n \frac{1}{\mu(e_i)} < 1$.
\end{lem}
\begin{proof}
We proceed by induction on $n$. Observe that the case $n=2$ follows from Lemma~\ref{lem:base_case1}. Suppose the result holds for $n-1$. By relabeling $e_n$ if necessary, we may assume $H':=\<\frac{\mu(e_i)}{\mu(e_j)}\colon i,j=1,\ldots, n-1\>$ is non-trivial. Denote $D:=\overset{p_0}{\bbC}\oplus\overset{p_1}{\bbC}$ and $\Gamma_1 := (V,E\setminus\{e_n,e_n^{\op}\}, \mu)$. 

First suppose $\displaystyle \sum_{i=1}^{n-1} \frac{1}{\mu(e_i)} \geq 1$ so that we also have $\displaystyle \sum_{i=1}^{n} \frac{1}{\mu(e_i)} \geq 1$. Then the induction hypothesis implies
	\begin{align*}
		\cM(\Gamma_0,\mu) &= \cM(\Gamma_1,\mu) \Asterisk_D W^*(Y_{e_n},D)\\
			&= (T_{H'},\vphi_{H'}) \Asterisk_D \begin{cases} \underset{1-\mu(e_n)}{\overset{q_0}{\bbC}} \oplus \underset{\mu(e_n), \mu(e_1)}{\overset{p_0-q_0, p_1}{M_2(L(\bbZ))}} &\text{if }\mu(e_n)<1\\
			&\\
			 \underset{1, \frac{\mu(e_1)}{\mu(e_n)}}{\overset{p_0, p_1-q_1}{M_2(L(\bbZ))}} \oplus \underset{\mu(e_1)\left(1-\frac{1}{\mu(e_n)}\right)}{\overset{q_1}{\bbC}} & \text{otherwise}\end{cases},
	\end{align*}
where $q_0=p_0-u_nu_n^*$ and $q_1=p_1-u_n^*u_n$. By Corollary~\ref{cor:absorption} we have $(T_{H'},\vphi_{H'}) \cong D* (T_{H'},\vphi_{H'})$, and so we can convert the above into a scalar-valued free product with respect to $\vphi$ using \cite[Proposition 4.1]{Hou07}. Then applying Corollary~\ref{cor:absorption} again yields
	\[
		(\cM(\Gamma_0,\mu),\vphi) \cong (T_H,\vphi_H).
	\]

Next, assume $\displaystyle \sum_{i=1}^{n-1} \frac{1}{\mu(e_i)} < 1$. Now the induction hypothesis implies
	\begin{align*}
		\cM(\Gamma_0,\mu) &= \cM(\Gamma_1,\mu) \Asterisk_D W^*(Y_{e_n},D)\\
			&= \left[ \underset{1+\mu(e_1)\sum_{i=1}^{n-1} \frac{1}{\mu(e_i)}}{(T_{H'}, \vphi_{H'})} \oplus \underset{\mu(e_1)\left[1- \sum_{i=1}^{n-1} \frac{1}{\mu(e_i)}\right]}{\bbC}\right] \Asterisk_D \begin{cases} \underset{1-\mu(e_n)}{\overset{q_0}{\bbC}} \oplus \underset{\mu(e_n), \mu(e_1)}{\overset{p_0-q_0, p_1}{M_2(L(\bbZ))}} &\text{if }\mu(e_n)<1\\
			&\\
			 \underset{1, \frac{\mu(e_1)}{\mu(e_n)}}{\overset{p_0, p_1-q_1}{M_2(L(\bbZ))}} \oplus \underset{\mu(e_1)\left(1-\frac{1}{\mu(e_n)}\right)}{\overset{q_1}{\bbC}} & \text{otherwise}\end{cases}.
	\end{align*}
We claim that
	\[
		 \underset{1+\mu(e_1)\sum_{i=1}^{n-1} \frac{1}{\mu(e_i)}}{(T_{H'}, \vphi_{H'})} \oplus \underset{\mu(e_1)\left[1- \sum_{i=1}^{n-1} \frac{1}{\mu(e_i)}\right]}{\bbC} \cong D * \left[\underset{1+\mu(e_1)-\alpha}{(T_{H'},\vphi_{H'}) }\oplus \overset{p}{\underset{\alpha}{\bbC}}\right],
	\]
where
	\[
		\alpha:= 1+\mu(e_1)\left[1-\sum_{i=1}^{n-1} \frac{1}{\mu(e_i)}\right] = \mu(e_{1})\left[1-\sum_{i=2}^{n-1} \frac{1}{\mu(e_i)}\right]  .
	\]
First note that $1\leq \alpha <\mu(e_1)$. Set
	\begin{align*}
		(N,\phi)&:= D * \left[\underset{1+\mu(e_1)-\alpha}{\bbC} \oplus \overset{p}{\underset{\alpha}{\bbC}}\right]\\
		(M,\phi)&:= D * \left[\underset{1+\mu(e_1)-\alpha}{(T_{H'},\vphi_{H'})} \oplus \overset{p}{\underset{\alpha}{\bbC}}\right].
	\end{align*}
Then by Lemma~\ref{lem:Dykema}
	\[
		\left( (1-p) M (1-p), \phi^{1-p}\right) \cong \left( (1-p)N (1-p), \phi^{1-p}\right) * (T_{H'},\vphi_{H'}),
	\]
and $z(p\colon M) = z(p\colon N)$. Since $(1-p)N (1-p)$ is abelian, we have that $\left( (1-p) M (1-p), \phi^{1-p}\right) \cong (T_{H'},\vphi_{H'})$. Since $z(p\colon N) = 1-p_1\wedge p$, Lemma~\ref{lem:antman2} gives
	\[
		\left( (1-p_1\wedge p ) M (1-p_1\wedge p), \phi^{1-p_1\wedge p}\right) \cong (T_{H'},\vphi_{H'}),
	\]
and $(p_1\wedge p) M (p_1\wedge p)$ is precisely the claimed atomic piece.

Now, returning to our amalgamated free product above, we have by \cite[Proposition 4.1]{Hou07} that
	\begin{align*}
		(\cM(\Gamma_0,\mu),\vphi)  \cong \left[\underset{1+\mu(e_1)-\alpha}{(T_{H'},\vphi_{H'})} \oplus \overset{p}{\underset{\alpha}{\bbC}}\right] * \begin{cases} \underset{1-\mu(e_n)}{\overset{q_0}{\bbC}} \oplus \underset{\mu(e_n), \mu(e_1)}{\overset{p_0-q_0, p_1}{M_2(L(\bbZ))}} &\text{if }\mu(e_n)<1\\
			&\\
			 \underset{1, \frac{\mu(e_1)}{\mu(e_n)}}{\overset{p_0, p_1-q_1}{M_2(L(\bbZ))}} \oplus \underset{\mu(e_1)\left(1-\frac{1}{\mu(e_n)}\right)}{\overset{q_1}{\bbC}} & \text{otherwise}\end{cases}.
	\end{align*}

\noindent
\underline{\textbf{Case 1:}} We first consider the case $\mu(e_n)<1$. Consider the following von Neumann subalgebras of $\cM(\Gamma_0,\mu)$:
	\begin{align*}
		(\cN_1,\vphi)&:=\left[\underset{1+\mu(e_1)-\alpha}{(T_{H'},\vphi_{H'})} \oplus \overset{p}{\underset{\alpha}{\bbC}}\right] * \left[ \underset{1-\mu(e_n)}{\overset{q_0}{\bbC}} \oplus \overset{p_0-q_0+p_1}{\underset{\mu(e_n)+\mu(e_1)}{\bbC}}\right]\\
		(\cN_2,\vphi)&:=\left[\underset{1+\mu(e_1)-\alpha}{(T_{H'},\vphi_{H'})} \oplus \overset{p}{\underset{\alpha}{\bbC}}\right] * \left[\underset{1-\mu(e_n)}{\overset{q_0}{\bbC}} \oplus \underset{\mu(e_n), \mu(e_1)}{\overset{p_0-q_0, p_1}{M_2(\bbC)}}\right]\\
	\end{align*}
Then an application of Lemma~\ref{lem:Dykema} yields
	\[
		(\cN_1,\vphi)\cong ( T_{H'}, \vphi_{H'}) \oplus \overset{p\wedge(p_0-q_0+p_1)}{\underset{\alpha + \mu(e_n) - 1}{\bbC}},
	\]
and so by Lemma~\ref{lem:antman1}
	\[
		\left(P\cN_1P, \vphi^{P}\right) \cong \underset{1+\mu(e_1) - \alpha}{(T_{H'},\vphi_{H'})} \oplus \overset{p\wedge(p_0-q_0+p_1)}{\underset{\alpha+\mu(e_n)-1}{\bbC}},
	\]
where $P:=p_0-q_0+p_1$. Thus by Lemma~\ref{lem:Dykema} we have
	\[
		(P\cN_2P,\vphi^{P}) \cong \left[ \underset{1+\mu(e_1)- \alpha}{(T_{H'},\vphi_{H'})} \oplus \underset{\alpha+\mu(e_n)- 1}{\bbC} \right] * \underset{\mu(e_n), \mu(e_1)}{M_2(\bbC)}.
	\]
Since $\mu(e_n) \leq \min\{ 1+\mu(e_1)-\alpha, \alpha  + \mu(e_n) - 1\}$, \cite[Theorem 4.3]{Hou07} implies for $\lambda=\frac{\mu(e_n)}{\mu(e_1)}$
	\begin{align*}
		(P\cN_2P,\vphi^{P})&\cong (P\cN_2P,\vphi^{P}) * (T_\lambda,\vphi_\lambda)\\
			&\cong \left[ (T_{H'},\vphi_{H'}) \oplus \bbC\right] * (T_\lambda,\vphi_\lambda)
	\end{align*}
Now by Corollary~\ref{cor:absorption}
	\begin{align*}
		(\cA,\vphi^{P}) &:= [\overset{q}{\bbC} \oplus \bbC ] * (T_\lambda,\vphi_\lambda) \cong (T_\lambda,\vphi_\lambda).
	\end{align*}
By Lemma~\ref{lem:antman1}, $(q\mc{A} q,\vphi^q)\cong (T_\lambda,\vphi_\lambda)$. Thus by Lemma~\ref{lem:Dykema} we have
	\begin{align*}
		(q\cN_2q,\vphi^q) &\cong (q\mc{A} q,\vphi^q) * (T_{H'},\vphi_{H'})\cong (T_\lambda,\vphi_\lambda) *(T_{H'},\vphi_{H'})\cong (T_H,\vphi_H).
	\end{align*}
Since $z(q\colon \cA^{\vphi^P}) = P$ and since $z(P\colon \cM(\Gamma_0,\mu)^\vphi)\geq z(p_1\colon \cM(\Gamma_0,\mu)^\vphi) =1$ by Lemma \ref{lem:centralizersupport}, we have $z(q\colon \cM(\Gamma_0,\mu)^\vphi)=1$. Then  Lemma~\ref{lem:antman2} implies that $(\cN_2,\vphi)\cong (T_H,\vphi_H)$. By Lemmas \ref{lem:Dykema} and \ref{lem:antman1} and free absorption we have
	\[
		(p_1\mc{M}(\Gamma_0,\mu) p_1, \vphi^{p_1}) \cong (p_1 \cN_2 p_1,\vphi^{p_1}) * (L(\bbZ),\tau) \cong (T_H,\vphi_H).
	\]
Then by Lemma~\ref{lem:antman2} we obtain $(\cM(\Gamma_0,\mu), \vphi) \cong (T_H,\vphi_H)$.

\noindent
\underline{\textbf{Case 2:}} Next, we consider the case $\mu(e_n)\geq 1$. Consider the following von Neumann subalgebras of $\cM(\Gamma_0,\mu)$:
	\begin{align*}
		(\cN_1,\vphi) &:=\left[(T_{H'},\vphi_{H'}) \oplus \overset{p}{\underset{\alpha}{\bbC}}\right] * \left[  \overset{p_0+p_1-q_1}{\underset{1+\frac{\mu(e_1)}{\mu(e_n)}}{\bbC}} \oplus \underset{\mu(e_1)\left(1- \frac{1}{\mu(e_n)}\right)}{\overset{q_1}{\bbC}} \right]\\
		(\cN_2,\vphi)&:=\left[(T_{H'},\vphi_{H'}) \oplus \overset{p}{\underset{\alpha}{\bbC}}\right] * \left[ \underset{1, \frac{\mu(e_1)}{\mu(e_n)}}{\overset{p_0, p_1-q_1}{M_2(\bbC)}} \oplus \underset{\mu(e_1)\left(1- \frac{1}{\mu(e_n)}\right)}{\overset{q_1}{\bbC}} \right]\\
	\end{align*}
Set
	\begin{align*}
		\beta&:= \left(\alpha+ 1 + \frac{\mu(e_1)}{\mu(e_n)}\right) - 1- \mu(e_1) = \alpha + \frac{\mu(e_1)}{\mu(e_n)} - \mu(e_1)\\
		\gamma&:= \left(\alpha + \mu(e_1)-\frac{\mu(e_1)}{\mu(e_n)}\right) - 1 - \mu(e_1) = \mu(e_1)\left( 1- \sum_{i=1}^n \frac{1}{\mu(e_i)}\right).
	\end{align*}
Then
	\[
		(\cN_1,\vphi)\cong (T_{H'},\vphi_{H'}) \oplus \overset{ p\wedge(p_0+p_1 - q_1)}{\underset{\beta}{\bbC}} \oplus \overset{p\wedge q_1}{\underset{\gamma}{\bbC}},
	\]
so that by Lemma~\ref{lem:Dykema}
	\[
		(Q \cN_2Q,\vphi^{Q}) \cong (Q\cN_1Q, \vphi^{Q}) * \underset{1, \frac{\mu(e_1)}{\mu(e_n)}}{\overset{p_0, p_1-q_1}{M_2(\bbC)}} \cong \left[(T_{H'}, \vphi_{H'}) \oplus  \overset{ p\wedge(p_0+p_1 - q_1)}{\underset{\beta}{\bbC}}\right] * \underset{1, \frac{\mu(e_1)}{\mu(e_n)}}{\overset{p_0, p_1-q_1}{M_2(\bbC)}},
	\]
where $Q:=p_0+p_1-q_1$. Now, observe that $\beta < \frac{\mu(e_1)}{\mu(e_n)}$ (since $\alpha < \mu(e_1)$). Denote $\lambda=\frac{\mu(e_1)}{\mu(e_n)}$. If $\lambda\neq 1$, then (noting that $(T_{H'})^{\vphi_{H'}}$ is diffuse) \cite[Theorem 4.3]{Hou07} implies
	\begin{align*}
		(Q\cN_2Q,\vphi^Q) \cong (Q\cN_2 Q,\vphi^{Q})  * (T_\lambda,\vphi_\lambda).
	\end{align*}
Note that if $\beta\leq 0$, then the above is immediate regardless of the value of $\lambda$. In any case, by the same argument as in Case 1 we obtain
	\[
		(\cM(\Gamma_0,\mu),\vphi)\cong (T_H,\vphi_H) \oplus \overset{p\wedge q_1}{\underset{\gamma}{\bbC}}.
	\]
If, however, $\beta>0$ and $\lambda=1$, then we take the following approach: by \cite[Proposition 3.2]{MR1201693}
	\[
		(\cA,\vphi^{Q}):= [\overset{q}{\bbC}\oplus \underset{\beta}{\bbC}] * \underset{1,1}{M_2(\bbC)}\cong \begin{cases} (L(\bbF_{t_1}),\tau) & \text{ if } \beta \in [1/2, 3/2]\\
						(L(\bbF_{t_2}),\tau)\oplus M_2(\bbC) & \text{ otherwise }\end{cases},
	\]
for some $t_1,t_2\geq 1$. In the second isomorphism, $q$ majorizes the identity of $M_2(\C)$ (when $\beta < 1/2$) and a non-zero projection in $L(\bbF_{t_2})$. It follows that
	\[
		(q\mc{A} q,\vphi^q) \cong \begin{cases} (L(\bbF_{s_1}), \tau) & \text{if }\frac32\geq \beta\\
					(L(\bbF_{s_2}),\tau) \oplus M_2(\bbC) & \text{if }\frac32 <\beta\end{cases},
	\]
for some $s_1,s_2\geq 1$. In both cases Lemma~\ref{lem:Dykema} and free absorption yields
	\begin{align*}
		(q\cN_2q,\vphi^q) &\cong \begin{cases} (L(\bbF_{s_1}),\tau) * (T_{H'},\vphi_{H'}) & \text{if }\frac32\geq \beta\\
			\left[(L(\bbF_{s_2}),\tau) \oplus M_2(\bbC)\right] * (T_{H'},\vphi_{H'}) & \text{if }\frac32 <\beta\end{cases}\\
			&\cong (T_{H'},\vphi_{H'}).
	\end{align*}
Note that since $\lambda=1$ we have $H'=H$. Following the same procedure as in Case 1, we have
	\[
		(\cM(\Gamma_0,\mu),\vphi)\cong (T_H,\vphi_H) \oplus \overset{p\wedge q_1}{\underset{\gamma}{\bbC}}.
	\]
Observe that the proof is unaffected if $\gamma\leq 0$.
\end{proof}

The following can also be thought of as a higher order version of Lemma~\ref{lem:base_case1}. However, the methods used to establish the isomorphism class are significantly different. Thus we consider this case separately.

\begin{lem}\label{lem:base_case3}
Let $\Gamma_0=(V_0,E_0)$, $\mu$, and $\vphi$ satisfy
	
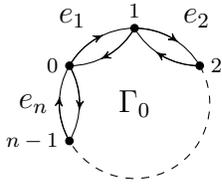
\[
		\begin{tikzpicture}
		\node at (0,0) {$\Gamma_0$};
		
		\node[left] at (-0.866,0.5) {\scriptsize$0$}; 
		\draw[fill=black] (-0.866,0.5) circle (0.05);
		
		\node[above left] at (-0.5,0.866) {$e_1$}; 
		\draw (-0.866,0.5) arc (150:90:1);
		\draw[->] (-0.866,0.5) arc (150:115:1);
		\draw (-0.866,0.5) arc (-90:-30:1);
		\draw[->] (0,1) arc (-30:-65:1);
		
		\node[above] at (0,1) {\scriptsize$1$}; 
		\draw[fill=black] (0,1) circle (0.05);
		
		\node[above right] at (0.5,0.866) {$e_2$}; 
		\draw (0,1) arc (90:30:1);
		\draw[->] (0,1) arc (90:55:1);
		\draw (0,1) arc (-150:-90:1);
		\draw[->] (0.866,0.5) arc (-90:-125:1);
		
		\node[right] at (0.866,0.5) {\scriptsize$2$}; 
		\draw[fill=black] (0.866,0.5) circle (0.05);
		
		\draw[dashed] (0.866,0.5) arc (30:-150:1); 
		
		\node[left] at (-0.866,-0.5) {\scriptsize$n-1$}; 
		\draw[fill=black] (-0.866,-0.5) circle (0.05);
		
		\node[left] at (-1,0) {$e_n$}; 
		\draw (-0.866,-0.5) arc (210:150:1);
		\draw[->] (-0.866,-0.5) arc (210:175:1);
		\draw (-0.866,0.5) arc (30:-30:1);
		\draw[->] (-0.866,0.5) arc (30:-5:1);

		\node[right] at (2,1.25) {$\bullet$ $V_0=\{0,1,2,\ldots, n-1\}$, $n\geq 3$, with $\vphi(p_0)=1$};
		\node[right] at (2,0.75) {{\color{white}$\bullet$} and $\vphi(p_i)=\mu(e_1)\cdots \mu(e_i)$ for $i=1,\ldots,n$}; 
		
		\node[right] at (2,-0.25) {$\bullet$ $E_0=\{e_i:=(i-1,i), e_i^{\op}\colon i=1,\ldots, n\}$ (vertices listed modulo $n$)};
		
		\node[right] at (2,-1.25) {$\bullet$ If $\mu_i:=\mu(e_i)$ then $\mu_1\cdots \mu_{k}\geq 1$, and for all $k \in \{1, \cdots, n\}$.};
		\end{tikzpicture}
	\]
Let $\lambda:=\frac{1}{\mu_1\cdots \mu_n}<1$. Then
	\[
		(\cM(\Gamma_0,\mu),\vphi)\cong ( T_\lambda,\vphi_\lambda) \oplus \bigoplus_{i=1}^{n-1} \underset{\mu_1\cdots \mu_i(1-\mu_{i+1} - \mu_{i}^{-1})}{\overset{r_i}{\bbC}},
	\]
where $r_i\leq p_i$ and the atomic terms in the direct sum may vanish if the indicated mass of $r_i$ is non-positive.
\end{lem}
\begin{proof}
We first compute the compression of $\cM(\Gamma_0,\mu)$ by $p_0+p_{n-1}$. Set
	\[
		D:= (p_0+p_{n-1}) A (p_0 +p_{n-1}) = \underset{1}{\overset{p_0}{\bbC}}\oplus \underset{\mu_1\cdots\mu_{n-1}}{\overset{p_{n-1}}{\bbC}}.
	\]
Let $\Gamma_1=(V_1,E_1)$ be the subgraph with $V_1=V$ and $E_1=E\setminus \{e_n\}$; that is, $\Gamma_1$ is a subgraph of $\Gamma_{\Tr}$. Then
	\begin{align*}
		(p_0+p_{n-1})&\cM(\Gamma_0,\mu)(p_0+p_{n-1}) \\
			&\cong \left[(p_0+p_{n-1})\cM(\Gamma_1,\mu)(p_0+p_{n-1})\right] \Asterisk_D \left[(p_0+p_{n-1}) W^*(Y_n, A) (p_0+p_{n-1})\right],
	\end{align*}
and $\varphi$ is a trace on $\cM(\Gamma_1,\mu)$. Now, by \cite{MR3110503} this amalgamated free product is isomorphic to
	\begin{align*}
		 \left[ \underset{1+\mu_1\cdots\mu_{n-2}}{L(\bbF_t)}\oplus \overset{s}{\underset{\mu_1\cdots\mu_{n-2}(\mu_{n-1}-1)}{\C}} \right]  \Asterisk_D \begin{cases} \overset{s_0}{\underset{1-\mu_n^{-1}}{\bbC}} \oplus \underset{\mu_n^{-1},\mu_1\cdots \mu_{n-1}}{M_2(L(\bbZ))}   & \text{if }\mu_n\geq 1\\
					&\\
			  \underset{1,\mu_1\cdots \mu_n}{M_2(L(\bbZ))} \oplus \overset{s_{n-1}}{\underset{\mu_1\cdots \mu_{n-1}(1-\mu_n)}{\bbC}} & \text{otherwise}\end{cases}
	\end{align*}
for some $t>1$, and projections $s\leq p_{n-1}$, $s_0\leq p_0$, and $s_{n-1}\leq p_{n-1}$. We note that when $\mu_n\geq 1$, the amalgam $D$ sits inside the second factor above as:
	\[
		p_0 = 1\oplus \left(\begin{array}{cc} 1 & 0 \\ 0 & 0 \end{array}\right)\qquad p_{n-1} = 0 \oplus \left(\begin{array}{cc} 0 & 0 \\ 0 & 1 \end{array}\right),
	\]
and otherwise $D$ sits inside the second factor as:
	\[
		p_0 = \left(\begin{array}{cc} 1 & 0 \\ 0 & 0 \end{array}\right) \oplus 0 \qquad p_{n-1} =\left(\begin{array}{cc} 0 & 0 \\ 0 & 1 \end{array}\right) \oplus 1.
	\]
We will argue that $D$ is always freely complemented in $(p_0+p_{n-1})\cM(\Gamma_1,\mu)(p_0+p_{n-1})$.  Our analysis depends whether or not $\mu_{n-1} \geq 1$.	

\begin{claim}
Assume that $\mu_{n-1} \geq 1$. Then there exists a finite diffuse von Neumann algebra, $B$, which is either $L(\Z)$ or an interpolated free group factor, and satisfies
\[
\left[(p_0+p_{n-1})\cM(\Gamma_1,\mu)(p_0+p_{n-1})\right] = D * \left[\underset{\mu_1\cdots\mu_{n-2}}B\oplus \underset{\mu_1\cdots \mu_{n-1} - (\mu_1\cdots \mu_{n-2} -1)}{\bbC}\right]
\]
i.e.
\[
		\underset{1+\mu_1\cdots\mu_{n-2}}{L(\bbF_t)}\oplus \overset{s}{\underset{\mu_1\cdots\mu_{n-2}(\mu_{n-1}-1)}{\C}} \cong \left[ \underset{1}{\bbC}\oplus \underset{\mu_1\cdots \mu_{n-1}}{\bbC}\right] * \left[\underset{\mu_1\cdots\mu_{n-2}}B\oplus \underset{\mu_1\cdots \mu_{n-1} - (\mu_1\cdots \mu_{n-2} -1)}{\bbC}\right]
\]
\end{claim}

\begin{proof}[Proof of Claim]
We first note that if such a $B$ were to exist, then the mass on the projection $s$ would be as claimed by \cite[Proposition 2.4]{MR1201693}.  To get the existence of $B$, we will argue that our parameter $t$ is sufficiently large.

For notational simplicity, we let $a$, $b$, and $c$ be the following parameters:
\[
	a = \frac{1}{1 + \mu_{1}\cdots \mu_{n-1}} \qquad b = \frac{\mu_{1}\cdots\mu_{n-2}}{1 + \mu_{1}\cdots \mu_{n-1}} \qquad c = \frac{\mu_{1}\cdots\mu_{n-2}(\mu_{n-1} - 1)}{1 + \mu_{1}\cdots\mu_{n-1}}
\]
Note that $a + b + c = 1$, and we are seeking $B$ satisfying
\[
\underset{a+b}{L(\F_{t})} \oplus \underset{c}{\C} = \left[ \underset{a}{\C} \oplus \underset{b+c}{\C} \right] *  \left[ \underset{b}{B} \oplus \underset{a+c}{\C} \right]
\]
Assume that $B$ has a generating set of free dimension $t_{B}$.  Using Dykema's free dimension calculations in \cite{MR1201693}, the left hand side of the equation has a generating set of free dimension $t(a+b)^{2} + 2(a+b)c = t(a+b)^{2} + 2(a+b)(1-(a+b))$ and the right hand side has a generating set of free dimension $2a(1-a) + 2b(1-b) + t_{B}b^{2}$.  Therefore, if we set 
\[
 t(a+b)^{2} + 2(a+b)(1-(a+b)) = 2a(1-a) + 2b(1-b) + t_{B}b^{2},
\]
we obtain that
\[
t_{B} = \frac{1}{b^{2}}[t(a+b)^{2} - 4ab].
\]
If we can show that $t_{B}$ bounded below by 1, then $B$ can be chosen to be $L*(\Z)$ or an interpolated free group factor, therefore eliminating the possibility of more than one atomic component in the free product.  

Our assumption on $\vphi$ is that (up to a scalar multiple) $\vphi(p_{0}) = a$, $\vphi(p_{n-2}) = b$, and $\vphi(p_{n-1}) = b+c$.  From our assumptions on $\mu_{1},\cdots, \mu_{n}$, it follows that $\vphi(p_{k}) \geq a$ for all $k$.  Note that 
	\[
		p_{n-2}\cM(\Gamma_{1}, \mu)p_{n-2} \cong \left[ \underset{\vphi(p_{n-3})}{L(\F_{t'})} \oplus \underset{b - \vphi(p_{n-3})}{\C}\right] * L(\Z)
	\]
where $t' \geq 1$, and it is understood that the $\C$ disappears if $b - \vphi(p_{n-3}) \leq 0$.  It follows that $p_{n-2}\cM(\Gamma_{1}, \mu)p_{n-2}$ is an interpolated free group factor, and its parameter is minimized when $\vphi(p_{n-3}) = a$, and $t' = 1$ (Note that this occurs when $\Gamma_0$ has exactly three vertices.).  In this case, $p_{n-2}\cM(\Gamma_{1}, \mu)p_{n-2} \cong L(\F_{t''})$ where 
$$
t'' = 1 + \left(1 - \frac{(b-a)^{2}}{b^{2}}\right).
$$
The $L(\F_{t})$ component of $(p_{0} + p_{n-1})\cM(\Gamma_{1}, \mu)(p_{0} + p_{n-1})$ is obtained by amplifying $p_{n-2}\cM(\Gamma_{1}, \mu)p_{n-2}$ by a projection of trace $\frac{a+b}{b}$.  The amplification formula therefore yields
$$
t = 1 + \frac{b^{2}}{(a+b)^{2}} \left(1 - \frac{(b-a)^{2}}{b^{2}}\right) = 1 + \frac{b^{2}}{(a+b)^{2}} - \frac{(b-a)^{2}}{(a+b)^{2}}.
$$ 
The minimal value of $t_{B}$ is therefore
	\begin{align*}
		t_{B} &= \frac{1}{b^{2}}\left( \left[1 + \frac{b^{2}}{(a+b)^{2}} - \frac{(b-a)^{2}}{(a+b)^{2}}\right](a+b)^{2} - 4ab \right)\\
			&= \frac{1}{b^{2}}((a+b)^{2} +b^{2} - (b-a)^{2} - 4ab) = 1,
	\end{align*}
as claimed.\renewcommand{\qedsymbol}{$\blacksquare$}
\end{proof}

\noindent{\underline{\textbf{Case 1:}}} Assume $\mu_{n-1}\geq 1$. Claim 1 implies that---up to an inner-automorphism---we have
	\[
		\underset{\mu_1\cdots\mu_{n-2}}{L(\bbF_t)}\oplus \overset{s}{\underset{\mu_1\cdots\mu_{n-2}(\mu_{n-1}-1)}{\C}} \cong D * \left[\underset{\mu_1\cdots\mu_{n-2}}B\oplus \underset{\mu_1\cdots \mu_{n-1} - (\mu_1\cdots \mu_{n-2} -1)}{\bbC}\right].
	\]
Thus, by \cite[Proposition 4.1]{Hou07} the above amalgmated free product collapses to the following scalar free product:
	\[
		\left[\underset{\mu_1\cdots\mu_{n-2}}B\oplus \underset{\mu_1\cdots \mu_{n-1} - (\mu_1\cdots \mu_{n-2} -1)}{\bbC}\right] * \begin{cases} \overset{s_0}{\underset{1-\mu_n^{-1}}{\bbC}} \oplus \underset{\mu_n^{-1},\mu_1\cdots \mu_{n-1}}{M_2(L(\bbZ))}   & \text{if }\mu_n\geq 1\\
					&\\
			  \underset{1,\mu_1\cdots \mu_n}{M_2(L(\bbZ))} \oplus \overset{s_{n-1}}{\underset{\mu_1\cdots \mu_{n-1}(1-\mu_n)}{\bbC}} & \text{otherwise}\end{cases}.
	\]
We shall consider the above cases separately, and we will understand these free products using the standard Dykema method of building up via various inclusions.\\

\noindent\underline{\textbf{Case 1.a:}} Assume $\mu_{n-1}\geq 1$ and $\mu_n\geq 1$. Consider the following von Neumann subalgebras of $\cM(\Gamma_0,\mu)$:
	\begin{align*}
		(\cN_1,\vphi)&:= \left[\underset{\mu_1\cdots\mu_{n-2}}B\oplus \underset{\mu_1\cdots \mu_{n-1} - (\mu_1\cdots \mu_{n-2} -1)}{\bbC}\right] * \left[ \overset{s_0}{\underset{1-\mu_n^{-1}}{\bbC}} \oplus \overset{r}{\underset{\mu_{n}^{-1} + \mu_1\cdots \mu_{n-1}}{\bbC}}\right]\\
		(\cN_2,\vphi)&:= \left[\underset{\mu_1\cdots\mu_{n-2}}B\oplus \underset{\mu_1\cdots \mu_{n-1} - (\mu_1\cdots \mu_{n-2} -1)}{\bbC}\right] *   \left[ \overset{s_0}{\underset{1-\mu_n^{-1}}{\bbC}} \oplus \overset{r-p_{n-1},p_{n-1}}{\underset{\mu_n^{-1}, \mu_1\cdots \mu_{n-1}}{M_2(\bbC)}}\right].
	\end{align*}
By \cite[Proposition 2.4]{MR1201693}, for some $t>1$ we have
	\[
		(\cN_1,\vphi) \cong \underset{\mu_1\cdots \mu_{n-2} + 1 -\mu_n^{-1}}{(L(\bbF_t),\tau)} \oplus \underset{\mu_n^{-1} + \mu_1\cdots \mu_{n-1} - \mu_1\cdots \mu_{n-2}}{\bbC}.
	\]
Thus for some (other) $t>1$ we have
	\[
		(r\cN_1 r,\vphi^r) \cong \underset{\mu_1\cdots \mu_{n-2}}{(L(\bbF_t),\tau)} \oplus \underset{\mu_n^{-1} + \mu_1\cdots \mu_{n-1} - \mu_1\cdots \mu_{n-2}}{\bbC}.
	\]
Now, this along with Lemma~\ref{lem:Dykema} and \cite[Theorem 3.1]{Hou07} implies
	\[
	 	(r\cN_2 r,\vphi^r) \cong (r \cN_1 r,\vphi^r) * \underset{\mu_n^{-1}, \mu_1\cdots \mu_{n-1}}{M_2(\bbC)} \cong (T_\lambda,\vphi_\lambda),
	\]
with $\lambda = (\mu_{1}\cdots \mu_{n})^{-1}$. Since $p_{n-1}\leq r$, it follows from Lemma~\ref{lem:antman1} that $(p_{n-1} \cN_2 p_{n-1},\vphi^{p_{n-1}})\cong (T_\lambda, \vphi_\lambda)$. Thus, applying Lemma~\ref{lem:Dykema} again along with free absorption yields
	\[
		(p_{n-1}\cM(\Gamma_0,\mu)p_{n-1}, \vphi^{p_{n-1}}) \cong (p_{n-1} \cN_2 p_{n-1},\vphi^{p_{n-1}}) * (L(\bbZ),\tau) \cong (T_\lambda, \vphi_\lambda).
	\]
Setting $z = \bigvee_{e \in E_0} u_{e}u_{e}^{*}$ so that $z = z(p_{n-1}\colon \cM(\Gamma_0, \mu)^\vphi)=z(p_{n-1}\colon \cM(\Gamma_0, \mu))$ by Remark~\ref{rem:central_support_same_in_centralizer},  we deduce $(z\cM(\Gamma_0, \mu),\vphi^z) \cong (T_{\lambda}, \vphi_{\lambda})$ from Lemma~\ref{lem:antman2}.  The formula for  $(\cM(\Gamma_0, \mu),\vphi)$ follows from Remark~\ref{rem:atomic}.\\

\noindent
\underline{\textbf{Case 1.b:}} Assume  $\mu_{n-1}\geq1$ and $\mu_n<1$. Consider the following subalgebras of $\cM(\Gamma_0,\mu)$:
	\begin{align*}
		(\cN_1,\vphi)&:= \left[\underset{\mu_1\cdots\mu_{n-2}}B\oplus \underset{\mu_1\cdots \mu_{n-1} - (\mu_1\cdots \mu_{n-2} -1)}{\bbC}\right] * \left[ \overset{r}{\underset{1+\mu_1\cdots\mu_n}{\bbC}}\oplus\overset{s_{n-1}}{\underset{\mu_1\cdots\mu_{n-1} - \mu_1\cdots \mu_n}{\bbC}}\right]\\
		(\cN_2,\vphi)&:= \left[\underset{\mu_1\cdots\mu_{n-2}}B\oplus \underset{\mu_1\cdots \mu_{n-1} - (\mu_1\cdots \mu_{n-2} -1)}{\bbC}\right] *   \left[\overset{p_0, r-p_0}{\underset{1,\mu_1\cdots \mu_n}{M_2(\bbC)}} \oplus \overset{s_{n-1}}{\underset{\mu_1\cdots\mu_{n-1} - \mu_1\cdots \mu_n}{\bbC}}\right].
	\end{align*}
We see that $(r\cN_{2}r,\vphi^r) \cong (r\cN_{1} r,\vphi^r) * \underset{1,\mu_1\cdots \mu_n}{M_2(\bbC)}$ from Lemma~\ref{lem:Dykema}.  Consider the four weightings 
	\[
		\mu_1\cdots\mu_{n-2},\ \ \ \mu_1\cdots \mu_{n-1} - (\mu_1\cdots \mu_{n-2} -1),\ \ \  1+\mu_1\cdots\mu_n,\ \ \  \text{and } \mu_1\cdots\mu_{n-1} - \mu_1\cdots \mu_n
	\]
in the summands of $\cN_{1}$.  

\underline{\textbf{Case 1.b.i:}} Assume $\mu_{n-1}\geq1$, $\mu_n <1$, and either $\mu_1\cdots\mu_{n-2}$ or $\mu_1\cdots\mu_{n-1} - \mu_1\cdots \mu_n$ is maximal amongst these weightings. Then $(r\cN_{1} r,\vphi^r) \cong (L(\F_t),\tau)$ for some $t\geq1$.  In this case, \cite[Theorem 6.7]{Shl97} and free absorption yields
	\[
		(r\cN_{2}r,\vphi^r) \cong L(\F_{t}) * \underset{1,\mu_1\cdots \mu_n}{M_2(\bbC)} \cong (T_{\lambda}, \vphi_{\lambda})
	\]
with $\lambda = (\mu_{1}\cdots \mu_{n})^{-1}$. Since $p_0\leq r$, by Lemma~\ref{lem:antman1} we have $(p_{0}\cN_{2}p_{0}, \vphi^{p_{0}}) \cong (T_{\lambda}, \vphi_{\lambda})$.  From Lemma~\ref{lem:Dykema} and free absorption it follows that 
	\[
		(p_{0}\cM(\Gamma_0,\mu)p_{0}, \vphi^{p_{0}}) \cong (p_{0}\cN_{2}p_{0}, \vphi^{p_{0}}) * (L(\Z),\tau) \cong (T_{\lambda}, \vphi_{\lambda})
	\]
Again, setting $z = \bigvee_{e \in E_0} u_{e}u_{e}^{*}$ so that $z =z(p_0\colon\cM(\Gamma_0,\mu)^\vphi)= z(p_{0}\colon \cM(\Gamma_0, \mu))$ by Remark~\ref{rem:central_support_same_in_centralizer}, we deduce that $(z\cM(\Gamma_0, \mu),\vphi^z) \cong (T_{\lambda}, \vphi_{\lambda})$ from Lemma~\ref{lem:antman2}.  The formula for  $(\cM(\Gamma_0, \mu),\vphi)$ follows from Remark~\ref{rem:atomic}. 

\underline{\textbf{Case 1.b.ii:}} Assume $\mu_{n-1}\geq 1$, $\mu_n <1$, and either $\mu_1\cdots \mu_{n-1} - (\mu_1\cdots \mu_{n-2} -1)$ or $1+\mu_1\cdots \mu_n$ is maximal amongst the weightings.
 Then $r\cN_{1}r \cong L(\F_{t}) \oplus \underset{\mu_{1}\cdots\mu_{n} + 1 - \mu_{1}\cdots\mu_{n-2}}{\C}$ for some $t\geq1$. 
  Note that $\mu_{1}\cdots\mu_{n} + 1 - \mu_{1}\cdots\mu_{n-2} \leq \mu_{1}\cdots\mu_{n}$. 
   It follows that there exists a trace preserving inclusion of $\underset{\alpha}{\C} \oplus \underset{\beta}{\C}$ into $r\cN_{1} r$ satisfying  $\alpha+\beta=1+\mu_1\cdots \mu_n$ and $1 \leq \alpha \leq \beta \leq \mu_{1}\cdots\mu_{n}$. 
              Then \cite[Theorem 4.3]{Hou07} gives
	\begin{align*}
		(r\cN_{2}r,\vphi^r) &\cong (r\cN_{1}r,\vphi^r) * \underset{1,\mu_1\cdots \mu_n}{M_2(\bbC)}\cong (r\cN_{1}r,\vphi^r) * \underset{1,\mu_1\cdots \mu_n}{M_2(\bbC)} * (T_{\lambda}, \vphi_{\lambda})\\
			&\cong \left( \underset{\mu_1\cdots\mu_{n-2}}{L(\F_{s_{2}})} \oplus \underset{\mu_{1}\cdots\mu_{n} + 1 - \mu_{1}\cdots\mu_{n-2}}{\C}\right) * \underset{1,\mu_1\cdots \mu_n}{M_2(\bbC)} * (T_{\lambda}, \vphi_{\lambda})\cong (T_{\lambda}, \vphi_{\lambda})
	\end{align*}
Continuing as in Case 1.b.i gives $(p_{0}\cM(\Gamma_0,\mu)p_{0}, \vphi^p_{0}) \cong (T_{\lambda}, \vphi_{\lambda})$, hence $(z\cM(\Gamma_0, \mu), \vphi^z) \cong (T_{\lambda}, \vphi_{\lambda})$, and so the formula for $(\cM(\Gamma_0, \mu), \vphi)$ follows from Remark~\ref{rem:atomic}.\\

This concludes Case 1. To address the case when $\mu_{n-1} < 1$, we first require the following claim:

\begin{claim}
Assume that $\mu_{n-1} < 1$.  There exists a finite diffuse von Neumann algebra, $B$, which is either $L(\Z)$ or an interpolated free group factor, and satisfies
	\[
		\left[(p_0+p_{n-1})\cM(\Gamma_1,\mu)(p_0+p_{n-1})\right] = D * \left[\underset{\mu_1\cdots\mu_{n-1}}B\oplus \underset{1}{\bbC}\right]
	\]
\end{claim}

\begin{proof}
Note that by hypothesis, we have that $\left[(p_0+p_{n-1})\cM(\Gamma_1,\mu)(p_0+p_{n-1})\right] \cong L(\F_{t})$ for some $t > 1$.  We are therefore seeking $B$ having the property
	\[
		L(\F_{t}) = \left[\underset{1}{\C} \oplus \underset{\mu_{1}\cdots\mu_{n-1}}{\C}\right] *  \left[\underset{\mu_1\cdots\mu_{n-1}}B\oplus \underset{1}{\bbC}\right].
	\]
We note that if such a $B$ were to exist, then the free product would contain no minimal projections, and so the above formula holds automatically.  Hence, in order to guarantee existence of this $B$ it suffices to show that $t$ is sufficiently large.
Assume that $B$ has a set of generators of free dimension $t_{B}$, and set $a = \frac{1}{1 + \mu_{1}\cdots\mu_{n-1}}$ and $b = \frac{\mu_{1}\cdots\mu_{n-1}}{1 + \mu_{1}\cdots\mu_{n-1}}$.  If the above free product were to hold, then a value for $t$ satisfies
	\[
		t = (1 - a^{2} - b^{2}) + (1 + (t_{B} - 1)b^{2} - a^{2});
	\]
that is, 
	\[
		t_{B} = \frac{t + 2(a^{2} + b^{2} - 1)}{b^{2}}.
	\]
We will now show that $t$ will be large enough to ensure that $t_{B} \geq 1.$

Note that
	\[
		p_{n-2}\cM(\Gamma_{1}, \mu)p_{n-2} \cong \left[\underset{\vphi(p_{n-3})}{L(\F_{s})} \oplus \underset{\vphi(p_{n-2}) - \vphi(p_{n-3})}{\C}\right] *  \left[\underset{\vphi(p_{n-1})}{L(\Z)} \oplus \underset{\vphi(p_{n-2}) - \vphi(p_{n-1})}{\C}\right],
	\]
so the above value for $t$ is minimized in the case that $s=1$ and $\vphi(p_{n-3}) = \vphi(p_0)=1$.  In this case, Dykema's free dimension calculations give
	\begin{align*}
		p_{n-2}&\cM(\Gamma_{1}, \mu)p_{n-2}\\
		 &\cong \begin{cases}
\underset{1 + \mu_{1}\cdots\mu_{n-1}}{L(\F(1 + 2ab))} \oplus \underset{\mu_{1}\cdots\mu_{n-2} - (1 + \mu_{1}\cdots\mu_{n-1})}{\C} &\text{ if } 1 + \mu_{1}\cdots\mu_{n-1} \leq \mu_{1}\cdots\mu_{n-2} \\
L\left(\F\left[2 - \frac{((\vphi(p_{n-2}) - 1)^{2} + (\vphi(p_{n-2}) - \mu_{1}\cdots\mu_{n-1})^{2}}{\vphi(p_{n-2})^{2}}\right]\right) &\text{ otherwise }
\end{cases}.
	\end{align*}
If the top case occurs, then $t=1 + 2ab$ since the identity of the interpolated free group factor component has mass $1+\mu_1\cdots \mu_{n-1}$.  In this case,
	\[
		t_{B} = \frac{1 + 2ab + 2(a^{2} + b^{2} - 1)}{b^{2}} = \frac{(a + b)^{2} - 1 + a^{2} + b^{2}}{b^{2}} = \frac{a^{2} + b^{2}}{b^{2}} > 1.
	\]
If the bottom case occurs, then we find $t$ by taking a $\frac{1 + \mu_{1}\cdots\mu_{n-1}}{\vphi(p_{n-2})}$ amplification of $p_{n-2}\cM(\Gamma_{1}, \mu)p_{n-2}$.  The amplification formula gives
	\begin{align*}
		t &= 1 + \frac{\vphi(p_{n-2})^{2}}{(1 + \mu_{1}\cdots\mu_{n})^{2}} - \frac{((\vphi(p_{n-2}) - 1)^{2} + (\vphi(p_{n-2}) - \mu_{1}\cdots\mu_{n-1})^{2}}{(1 + \mu_{1}\cdots\mu_{n})^{2}}
	\end{align*}
The values for $\vphi(p_{n-2})$ lie in $[\mu_{1}\cdots\mu_{n}, \mu_{1}\cdots\mu_{n} + 1]$, and one checks that this expression is minimized when $\vphi(p_{n-2}) = \mu_{1}\cdots\mu_{n-1}$.  In this case, we have
	\[
		t = 1 + \frac{2\mu_{1}\cdots\mu_{n-1} - 1}{(1 + \mu_{1}\cdots\mu_{n-1})^{2}} = 1 + 2ab - a^{2}.
	\]
This means that 
	\[
t_{B} =  \frac{1 + 2ab - a^{2} + 2(a^{2} + b^{2} - 1)}{b^{2}} = \frac{b^{2}}{b^{2}} = 1.
	\]
Therefore, in either case, $B$ exists.\renewcommand{\qedsymbol}{$\blacksquare$}
\end{proof}

\noindent
\underline{\textbf{Case 2:}} Assume $\mu_{n-1}<1$. By Claim 2 and \cite[Proposition 4.1]{Hou07}, we have that
	\begin{align*}
		((p_0+p_{n-1})\cM(\Gamma_0,\mu)&(p_0+p_{n-1}), \vphi^{p_0+p_{n-1}}) \\
			&\cong \left[ \underset{\mu_1\cdots\mu_{n-1}}{B}\oplus \underset{1}{\bbC}\right] *  
			\begin{cases} 
				\overset{s_0}{\underset{1-\mu_n^{-1}}{\bbC}} \oplus \overset{r-p_{n-1},p_{n-1}}{\underset{\mu_n^{-1}, \mu_1\cdots \mu_{n-1}}{M_2(L(\bbZ))}} & \text{if }\mu_n\geq 1\\
				\overset{p_0, r-p_0}{\underset{1,\mu_1\cdots \mu_n}{M_2(L(\bbZ))}} \oplus \overset{s_{n-1}}{\underset{\mu_1\cdots \mu_{n-1}(1-\mu_n)}{\bbC}} &\text{ otherwise}
			\end{cases}.
	\end{align*}
As in Case 1, our analysis will depend on $\mu_n$.\\

\noindent
\underline{\textbf{Case 2.a:}} Assume $\mu_{n-1}<1$ and $\mu_n\geq 1$. Consider the following subalgebras of $\cM(\Gamma_0,\mu)$:
	\begin{align*}
		(\cN_1,\vphi)&:=\left[\underset{\mu_1\cdots\mu_{n-1}}B\oplus \underset{1}{\bbC}\right]* \left[ \overset{s_0}{\underset{1-\mu_n^{-1}}{\bbC}} \oplus \overset{r}{\underset{\mu_{n}^{-1} + \mu_1\cdots \mu_{n-1}}{\bbC}}\right]\\
		(\cN_2,\vphi)&:= \left[\underset{\mu_1\cdots\mu_{n-1}}B\oplus \underset{1}{\bbC}\right] *   \left[ \overset{s_0}{\underset{1-\mu_n^{-1}}{\bbC}} \oplus \overset{r-p_{n-1},p_{n-1}}{\underset{\mu_n^{-1}, \mu_1\cdots \mu_{n-1}}{M_2(\bbC)}}\right].
	\end{align*}
From \cite{MR1201693}, we see that for some $t\geq 1$
	\[
		r\cN_{1}r \cong \underset{1 + \mu_{1}\cdots \mu_{n-1} - \mu_{n}^{-1}}{L(\F_{t})} \oplus \underset{\mu_{n}^{-1}}{\C}
	\]
so that there is a state preserving inclusion of $\underset{\alpha}{\C} \oplus \underset{\beta}{\C}$ into $r\cN_{1}r$ with $\mu_{n}^{-1} \leq \alpha \leq \beta \leq \mu_{1}\cdots\mu_{n-1}$.  Using the arguments in Case 1.b.ii above, we see that
	\[
		(r\cN_{2}r, \vphi^{r})  \cong \left[\underset{1 + \mu_{1}\cdots \mu_{n-1} - \mu_{n}^{-1}}{L(\F_{t})} \oplus \underset{\mu_{n}^{-1}}{\C}\right] * \underset{\mu_{1}\cdots\mu_{n-1}, \mu_{n}^{-1}}{M_{2}(\C)} \cong (T_{\lambda}, \vphi_{\lambda})
	\]
Consequently $(p_{n-1}\cN_{2}p_{n-1}, \vphi^{p_{n-1}}) \cong (T_{\lambda}, \vphi_{\lambda})$ by Lemma~\ref{lem:antman1}.  Then by Lemma~\ref{lem:Dykema} it follows that 
	\[
		(p_{n-1}\cM(\Gamma_0,\mu)p_{n-1}, \vphi^{p_{n-1}}) \cong (p_{n-1}\cN_{2}p_{n-1}, \vphi^{p_{n-1}}) * (L(\Z),\tau) \cong (T_{\lambda}, \vphi_{\lambda})
	\]
Proceeding as in the previous cases, we obtain $(z\cM(\Gamma_0, \mu), \vphi^{z})  \cong (T_{\lambda}, \vphi_{\lambda})$ and the formula for $(\cM(\Gamma_0, \mu), \vphi)$ follows from Remark~\ref{rem:atomic}.\\

\noindent
\underline{\textbf{Case 2.b:}} Assume $\mu_{n-1}<1$ and $\mu_{n} < 1$. Consider the following subalgebras of $\cM(\Gamma_0,\mu)$:
\begin{align*}
		(\cN_1,\vphi)&:= \left[\underset{\mu_1\cdots\mu_{n-1}}B\oplus \underset{1}{\bbC}\right] * \left[ \overset{r}{\underset{1+\mu_1\cdots\mu_n}{\bbC}}\oplus\overset{s_{n-1}}{\underset{\mu_1\cdots\mu_{n-1} - \mu_1\cdots \mu_n}{\bbC}}\right]\\
		(\cN_2,\vphi)&:= \left[\underset{\mu_1\cdots\mu_{n-1}}B\oplus \underset{1}{\bbC}\right] *   \left[\overset{p_0, r-p_0}{\underset{1,\mu_1\cdots \mu_n}{M_2(\bbC)}} \oplus \overset{s_{n-1}}{\underset{\mu_1\cdots \mu_{n-1}(1-\mu_n)}{\bbC}}\right]\\
		(\cN_3,\vphi)&:=\left[\underset{\mu_1\cdots\mu_{n-1}}B\oplus \underset{1}{\bbC}\right] * \left[\overset{p_0, r-p_0}{\underset{1,\mu_1\cdots \mu_n}{M_2(L(\bbZ))}} \oplus \overset{s_{n-1}}{\underset{\mu_1\cdots \mu_{n-1}(1-\mu_n)}{\bbC}}\right].
	\end{align*}
From \cite{MR1201693}, we see that for some $t\geq 1$
	\[
		r\cN_{1}r \cong \begin{cases}
L(\F_{t}) &\text{ if } 1 + \mu_{1}\dots\mu_{n} \geq \mu_{1}\cdots\mu_{n-1}\\
\underset{\mu_{1}\cdots \mu_{n-1}}{L(\F_{t})} \oplus \underset{\mu_{1}\cdots\mu_{n} + 1 - \mu_{1}\cdots\mu_{n-1} }{\C} &\text{ otherwise }
 \end{cases}
	\]
As $\mu_{1}\cdots\mu_{n} + 1 - \mu_{1}\cdots\mu_{n-1} < \mu_{1}\cdots\mu_{n}$, we see that both possible algebras contains a unital inclusion of $\underset{\alpha}{\C} \oplus \underset{\beta}{\C}$ with $1 \leq \alpha \leq \beta \leq \mu_{1}\cdots\mu_{n}$.  Arguing as in cases 1.b.ii and 2.a, we see that
	\begin{align*}
		(r\cN_{2}r, \vphi^{r})\cong (r\cN_{1}r, \vphi^{r}) * \overset{p_0, r-p_0}{\underset{1,\mu_1\cdots \mu_n}{M_2(\bbC)}}
\cong (T_{\lambda}, \vphi_{\lambda})
	\end{align*}
As before, Lemma~\ref{lem:antman1} implies that $(p_{n-1}\cN_{2}p_{n-1}, \vphi^{p_{n-1}}) \cong (T_{\lambda}, \vphi_{\lambda})$, and Lemma~\ref{lem:Dykema} implies
	\begin{align*}
		(p_{n-1}\cM(\Gamma_0,\mu)p_{n-1}, \vphi^{p_{n-1}}) \cong  (p_{n-1}\cN_{2}p_{n-1}, \vphi^{p_{n-1}}) * (L(\Z),\tau) \cong (T_{\Lambda}, \vphi_{\lambda}).
	\end{align*}
Proceeding as in the previous cases, we obtain $(z\cM(\Gamma_0, \mu), \vphi^{z})  \cong (T_{\lambda}, \vphi_{\lambda})$ and the formula for $(\cM(\Gamma_0, \mu), \vphi)$ follows from Remark~\ref{rem:atomic}.
\end{proof}

\begin{rem}\label{rem:loop}
In the proof of Lemma~\ref{lem:base_case3}, we needed a very special choice of $\Gamma_{\Tr}$ to utilize our free complementation arguments.  But from Proposition~\ref{prop:changestate}, we know $(zM(\Gamma_0, \mu), \vphi) \cong (T_{\lambda}, \vphi_{\lambda})$ regardless of the choice of $\Gamma_{\Tr}$ inducing $\vphi$. This will be crucial in the proof of Lemma~\ref{lem:adding_edge_to_existing_vertices} below.
\end{rem}


\subsection{Under construction}\label{subsec:under_construction}

In the following three lemmas we have a fixed ambient graph $\Gamma=(V,E)$ with edge weighting $\mu$. We let $\vphi$ be the positive linear functional on $\cM(\Gamma,\mu)$ as defined in the beginning of this section.

In this first lemma, we consider constructing a new graph from an existing one by adding a edge whose source is part of the original graph but whose target is not. Provided the von Neumann algebra generated by the original graph is a free Araki--Woods factor (up to direct sum with a finite dimensional abelian von Neumann algebra), then so is the von Neumann algebra generated by the new graph.

\begin{lem}\label{lem:adding_edge_and_vertex}
Let $\Gamma_1 =(V_1,E_1)$ and $\Gamma_2 = (V_2,E_2)$ be subgraphs of $\Gamma$ such that
	
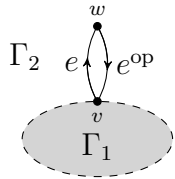
\[
		\begin{tikzpicture}
		\draw[dashed, fill={rgb:black,1;white,5}] (0,0) ellipse (1 and 0.5); 
		\node at (0,-.1) {$\Gamma_1$};
		
		\node[left] at (-0.6,1.1) {$\Gamma_2$}; 
		
		\node[below] at (0,0.5) {\scriptsize$v$}; 
		\draw[fill=black] (0,0.5) circle (0.05);
		
		\node[left] at (-0.1,1) {$e$}; 
		\node[right] at (0.1,1) {$e^{\op}$};
		\draw (0,0.5) arc (210:150:1);
		\draw[->] (0,0.5) arc (210:175:1);
		\draw (0,0.5) arc (-30:30:1);
		\draw[->] (0,1.5) arc (30:-5:1);
		
		\node[above] at (0,1.5) {\scriptsize$w$}; 
		\draw[fill=black] (0,1.5) circle (0.05);
		
		\node[right] at (2,1.5) {$\bullet$ $V_2=V_1\sqcup \{w\}$};
		
		\node[right] at (2,0.5) {$\bullet$ $E_2 = E_1\sqcup\{e,e^{\op}\}$ where $s(e)=v\in V_1$ and $t(e)=w$};
		
		\node[right] at (2,-0.5) {$\bullet$ $\vphi(p_w)=\mu(e)\vphi(p_v)$};
		
		\end{tikzpicture}
	\] 
Assume
	\[
		(\cM(\Gamma_1,\mu), \varphi) \cong ( T_H, \vphi_H)\oplus \bigoplus_{u\in V_1} \overset{ r_u}{\bbC},
	\]
where $H=H(\Gamma_1,\mu)$ and $r_u\leq p_u$ exists if and only if $\sum_{\substack{f\in E_1\\ s(f)=u}} \mu(f) < 1$, in which case $ \varphi(r_u) = \varphi(p_u) \left[ 1- \sum_{\substack{f\in E_1\\ s(f)=u}} \mu(f)\right]$. Then
	\[
		(\cM(\Gamma_2, \mu), \varphi) \cong (T_H, \varphi_H) \oplus \bigoplus_{u\in V_2} \overset{s_u}{\bbC},
	\]
where $s_u\leq p_u$ exists if and only if $\sum_{\substack{f\in E_2\\ s(f)=u}} \mu(f)< 1$, in which case $\varphi(s_u) = \varphi(p_u)\left[1 - \sum_{\substack{f\in E_2\\ s(f)=u}} \mu(f)\right]$. In particular, $s_u=r_u$ for all $u\in V_1\setminus \{v\}$.
\end{lem}
\begin{proof}
Denote by $\alpha:= \vphi(r_v)$. Observe that
	\begin{align*}
		(p_v \cM(\Gamma_2, \mu) p_v,\vphi^{p_v}) &= (p_v \cM(\Gamma_1, \mu) p_v,\vphi^{p_v}) * (p_v W^*(Y_eY_e^*) p_v,\vphi^{p_v})\\
			&\cong \begin{cases} \left[\underset{\vphi(p_v) - \alpha}{(T_H,\varphi_H)} \oplus\underset{\alpha}{\overset{r_v}{\bbC}}\right] * (L(\bbZ),\tau) &\text{if }\mu(e)\geq 1\\
			&\\
			\left[\underset{\vphi(p_v) - \alpha}{(T_H,\varphi_H)} \oplus\underset{\alpha}{\overset{r_v}{\bbC}}\right] * \left[(L(\bbZ),\tau)\oplus \underset{\varphi(p_v)[1-\mu(e)]}{\bbC}\right]& \text{if }\mu(e)<1\end{cases}.
	\end{align*} 
Consider the case $\mu(e)\geq 1$. Let us denote $q_v:=p_v-r_v$ and:
	\begin{align*}
		(\cN,\vphi)&:= \left[ \overset{q_v}{\underset{ \vphi(p_v)-\alpha }{\bbC}} \oplus \overset{r_v}{\underset{\alpha}{\bbC}}\right] * (L(\bbZ),\tau).
	\end{align*}
We have that $\cN\cong L(\bbF_t)$ for some $t>1$, which implies that $z(q_v\colon p_v\cM(\Gamma_2,\mu)^\varphi p_v)= p_v$. Moreover, by Lemma~\ref{lem:Dykema} and free absorption we have (for some $t'>1$)
	\[
		(q_v \cM(\Gamma_2,\mu) q_v, \vphi^{q_v}) \cong (q_v \cN q_v, \vphi^{q_v}) * (T_H, \varphi_H) \cong (L(\bbF_{t'}),\tau) * (T_H, \varphi_H) \cong (T_H, \varphi_H).
	\]
Thus $(p_v \cM(\Gamma_2,\mu) p_v, \vphi^{p_v}) \cong (T_H,\vphi_H)$  by Lemma~\ref{lem:antman2}. 

By Remark~\ref{rem:central_support_same_in_centralizer}, we have
	\[
		z:=z(p_v\colon \cM(\Gamma_2,\mu)^\vphi) = z(p_v\colon \cM(\Gamma_2,\mu))
	\]
and so Lemma~\ref{lem:antman2} implies
	\[
		(z\cM(\Gamma_2,\mu), \vphi^z) \cong (T_H,\vphi_H).
	\]
The formula for $(\cM(\Gamma_2,\mu),\vphi)$ then follows from Remark~\ref{rem:atomic}.

Next we consider the case $\mu(e)<1$. Let us denote:
	\[
		(\cN,\vphi):= \left[ \overset{q_v}{\underset{ \vphi(p_v)-\alpha }{\bbC}} \oplus \underset{\alpha}{\bbC}\right] * \left[\underset{\mu(e)\vphi(p_v)}{(L(\bbZ),\tau)}\oplus \overset{p_v - u_eu_e^*}{\underset{\varphi(p_v)[1-\mu(e)]}{\bbC}}\right].
	\]
Then for some $t>1$ we have
	\[
		(\cN,\vphi)\cong (L(\bbF_t),\tau)\oplus \overset{q_v\wedge(p_v - u_eu_e^*)}{\bbC} \oplus \overset{r_v\wedge(p_v-u_eu_e^*)}{\underset{\alpha - \mu(e)\vphi(p_v)}{\bbC}}.
	\]
By Lemma~\ref{lem:Dykema} and free absorption we have
	\[
		(q_v \cM(\Gamma_2,\mu) q_v,\vphi^{q_v}) \cong (q_v \cN q_v, \varphi^{q_v}) * (T_H,\vphi_H) \cong (T_H,\vphi_H).
	\]
Realize that, by Lemma~\ref{lem:centralizersupport}, we know
	\[
		q_v= p_{v} \cdot \bigvee_{f\in E_1} u_fu_f^*,
	\]
and so
	\[
		z:=z(q_v\colon \cM(\Gamma_2,\mu)) = z(q_v\colon \cM(\Gamma_2,\mu)^\vphi) = \bigvee_{f\in E_2} u_fu_f^*.
	\]
Consequently, $(z \cM(\Gamma,\mu),\vphi^z)\cong (T_H,\vphi_H)$ by Lemma~\ref{lem:antman2}. The formula for $(\cM(\Gamma_2,\mu),\vphi)$ then follows from Remark~\ref{rem:atomic}.
\end{proof}

Note that the hypothesis $\vphi(p_w)= \mu(e)\vphi(p_v)$ in the previous lemma can be guaranteed if $e\in \Gamma_{\Tr}$. Thus, when we begin to build $\Gamma$ from $\Gamma_0$, we will use the previous lemma to first add only edges in $\Gamma_{\Tr}$. Since $V(\Gamma_{\Tr})=V$, to finish building up to $\Gamma$ we will just need to add edges whose source and target are already in place. We consider loops and non-loops separately in the following two lemmas.

\begin{lem}\label{lem:adding_loop_to_existing_vertex}
Let $\Gamma_1 =(V_1,E_1)$ and $\Gamma_2 = (V_2,E_2)$ be subgraphs of $\Gamma$ such that
	
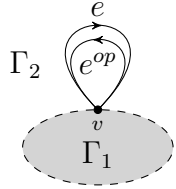
\[
		\begin{tikzpicture}
		\draw[dashed, fill={rgb:black,1;white,5}] (0,0) ellipse (1 and 0.5); 
		\node at (0,-.1) {$\Gamma_1$};
		
		\node[left] at (-0.6,1.1) {$\Gamma_2$}; 
		
		\node[below] at (0,0.5) {\scriptsize$v$}; 
		\draw[fill=black] (0,0.5) circle (0.05);
		
		\draw (0,0.5) .. controls (-1.5,2) and (1.5,2) .. (0,0.5); 
		\draw[->] (0.01,1.62) -- (0.05,1.62);
		\node[above] at (0,1.6) {$e$};

		\draw (0,0.5) .. controls (-1.25, 1.75) and (1.25,1.75) .. (0,0.5); 
		\draw[->] (-0.02, 1.43) -- (-0.06,1.43);
		\node[below] at (0,1.4) {$e^{op}$};
		

		\node[right] at (2,1.5) {$\bullet$ $V_2=V_1$};
		
		\node[right] at (2,0.5) {$\bullet$ $E_2 = E_1\sqcup\{e,e^{\op}\}$ where $s(e)=t(e)=v\in V_1$; or};
		
		\node[right] at (2,-0.5) {$\bullet$ $E_2=E_1\sqcup\{e\}$ where $e=e^{op}$ and $s(e)=t(e)=v\in V_1$};
		
		\end{tikzpicture}
	\] 
Assume
	\[
		(\cM(\Gamma_1,\mu), \varphi) \cong ( T_{H'}, \vphi_{H'})\oplus \bigoplus_{u\in V_1} \overset{ r_u}{\bbC},
	\]
where $H'=H(\Gamma_1,\mu)$ and $r_u\leq p_u$ exists if and only if $\sum_{\substack{f\in E_1\\ s(f)=u}} \mu(f) < 1$, in which case $ \varphi(r_u) = \varphi(p_u) \left[ 1- \sum_{\substack{f\in E_1\\ s(f)=u}} \mu(f)\right]$. Then
	\[
		(\cM(\Gamma_2, \mu), \varphi) \cong (T_H, \varphi_H) \oplus \bigoplus_{u\in V_2} \overset{s_u}{\bbC},
	\]
where $H=H(\Gamma_2,\mu)$, $s_u=r_u$ for $u\neq v$, and $s_v=0$.
\end{lem}
\begin{proof}
First observe that, by Lemma~\ref{lem:antman1}, we have
	\[
		(p_v \cM(\Gamma_1,\mu)p_v,\vphi^{p_v}) \cong (T_{H'},\vphi_{H'}) \oplus\overset{r_v}{\bbC}.
	\]
Furthermore, by Proposition~\ref{prop:edges_elements_are_eigenops}, we have
	\[
		(W^*(Y_e), \vphi)\cong (p_vW^*(Y_e) p_v, \vphi^{p_v}) \cong \begin{cases} (T_{\mu(e)},\vphi_{\mu(e)}) & \text{if }\mu(e)\neq 1\\ 
			(L(\mathbb{F}_2),\tau) & \text{if }\mu(e)=1\text{ but }e\neq e^{op}\\
			(L(\bbZ),\tau) & \text{if }\mu(e)=1\text{ and }e=e^{op}\end{cases}.
	\] 
In each of the above three cases we therefore have
	\begin{align*}
		(p_v\cM(\Gamma_2,\mu)p_v,\mu),\vphi^{p_v})\cong (p_v \cM(\Gamma_1,\mu)p_v,\vphi^{p_v}) * (p_vW^*(Y_e) p_v, \vphi^{p_v}) \cong (T_H,\vphi_H),
	\end{align*}
where $H=H(\Gamma_2,\mu)$. Then Remark~\ref{rem:atomic} and Lemma~\ref{lem:antman2}  yield the claimed isomorphism.
\end{proof}

\begin{lem}\label{lem:adding_edge_to_existing_vertices}
Let $\Gamma_1 = (V_1,E_1)$ and $\Gamma_2 = (V_2,E_2)$ be subgraphs of $\Gamma$ such that
	
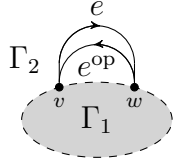
\[
		\begin{tikzpicture}
		\draw[dashed, fill={rgb:black,1;white,5}] (0,0) ellipse (1 and 0.5); 
		\node at (0,0) {$\Gamma_1$};
		
		\node[left] at (-0.6,0.8) {$\Gamma_2$}; 
		
		\node[below] at (-0.5,0.433) {\scriptsize$v$}; 
		\draw[fill=black] (-0.5,0.433) circle (0.05);
		
		\node[above] at (0,1.25) {$e$}; 
		\draw (-0.5,0.433) -- (-0.5,0.75);
		\draw (0.5,0.433) -- (0.5,0.75);
		\draw (-0.5,0.75) arc (180:0:0.5);
		\draw[->] (-0.5,0.75) arc (180:85:0.5);
		
		\node[below] at (0,1) {$e^{\op}$}; 
		\draw (0.5,0.5) arc (0:180:0.5);
		\draw[->] (0.5,0.5) arc (0:95:0.5);
		
		\node[below] at (0.5,0.433) {\scriptsize$w$}; 
		\draw[fill=black] (0.5,0.433) circle (0.05);
		
		\node[right] at (2,1) {$\bullet$ $V_2=V_1$};
		
		\node[right] at (2,0) {$\bullet$ $E_2 = E_1\sqcup \{e, e^{\op}\}$ where $s(e)=v$ and $t(e)=w$ and $\mu(e)\geq 1$};
		
		\end{tikzpicture}
	\] 
Assume
	\[
		(\cM(\Gamma_1,\mu), \varphi) \cong ( T_{H'}, \vphi_{H'})\oplus \bigoplus_{u\in V_1} \overset{ r_u}{\bbC},
	\]
where $H'=H(\Gamma_1,\mu)$,  $r_u\leq p_u$ exists if and only if $\sum_{\substack{f\in E_1\\ s(f)=u}} \mu(f) < 1$, in which case $ \varphi(r_u) = \varphi(p_u) \left[ 1- \sum_{\substack{f\in E_1\\ s(f)=u}} \mu(f)\right]$. Then
	\[
		(\cM(\Gamma_2, \mu), \varphi) \cong (T_H, \varphi_H) \oplus \bigoplus_{u\in V_2} \overset{s_u}{\bbC},
	\]
where $H=H(\Gamma_2,\mu)$, $s_u=r_u$ for $u\neq v,w$, $s_v=0$, and $s_w\leq r_w$ exists if and only if $\sum_{\substack{f\in E_2\\ s(f)=w}} \mu(f)< 1$, in which case $\varphi(s_w) = \varphi(p_w)\left[1 - \sum_{\substack{f\in E_2\\ s(f)=w}} \mu(f)\right]$.
\end{lem}

\begin{proof}
We will compute the compression of $\cM(\Gamma_2,\mu)$ by $p_v+p_w$. We will use a bit  of \emph{sleight of hand} to modify $\Gamma_2$ by swapping out the generic $\Gamma_1$ for a more controlled graph. Set $\alpha:=\vphi(r_v)$ and $\beta:=\vphi(r_w)$. Consider the graph $\Gamma_2'=(V_2', E_2',)$  such that
	\[
		\begin{tikzpicture}
		\node[left] at (-1.6,1.7) {$\Gamma_2'$};

		\draw (-2,0) arc (185:-5:2); 
		\draw[->] (-2,0) arc (185:87.5:2);
		\node[above] at (0,2.15) {$e$};

		\draw (-2,0) arc (170:10:2.031); 
		\draw[->] (2,0) arc (10:92.5:2.031);
		\node[above] at (0,1.6) {$e^{\op}$};

		\node[left] at (-2,0) {\scriptsize$v$}; 
		\draw[fill=black] (-2,0) circle (0.05);
		
		\node[above] at (-1.5,0.288675) {$f_v$}; 
		\draw (-2,0) arc (150:30:0.57735);
		\draw[->] (-2,0) arc (150:85:0.57735);
		\node[below] at (-1.5,-0.288675) {$f_v^{\op}$};
		\draw (-2,0) arc (-150:-30:0.57735);
		\draw[->] (-1,0) arc (-30:-95:0.57735);
		
		\node[left] at (-1,0) {\scriptsize$0$}; 
		\draw[fill=black] (-1,0) circle (0.05);
		
		\node[right] at (1,0) {\scriptsize$1$}; 
		\draw[fill=black] (1,0) circle (0.05);

		\node[above] at (1.5,0.288675) {$f_w$}; 
		\draw (2,0) arc (30:150:0.57735);
		\draw[->] (2,0) arc (30:95:0.57735);
		\node[below] at (1.5,-0.288675) {$f_w^{\op}$};
		\draw (2,0) arc (-30:-150:0.57735);
		\draw[->] (1,0) arc (-150:-85:0.57735);

		\node[right] at (2,0) {\scriptsize$w$}; 
		\draw[fill=black] (2,0) circle (0.05);
		
		\node[below] at (0, 0.57735) {$e_1$}; 
		\draw (-1,0) arc (150:30:1.1547);
		\draw[->] (-1,0) arc (150:85:1.1547);
		
		\draw[densely dotted,thick] (0,0.6) -- (0,1);
		
		\node[above] at (0,1) {$e_n$}; 
		\draw (-1,0) arc (180:0:1);
		\draw[->] (-1,0) arc (180:85:1);
		
		\node[above] at (0,-0.57735) {$e_1^{\op}$}; 
		\draw (1,0) arc (-30:-150:1.1547);
		\draw[->] (1,0) arc (-30:-95:1.1547);
		
		\draw[densely dotted, thick] (0,-0.6) -- (0,-1);
		
		\node[below] at (0,-1) {$e_n^{\op}$}; 
		\draw (-1,0) arc (-180:0:1);
		\draw[->] (1,0) arc (0:-95:1);

		\node[right] at (3,1.5) {$\bullet$ $V_2'=\{v,w,0,1\}$};
		\node[right] at (3,.5) {$\bullet$ $E_2'=\{e,e^{\op},f_v,f_v^{\op}, f_w,f_w^{\op}, e_1,e_1^{\op},\ldots, e_n,e_n^{\op}\}$};
		\node[right] at (3,0) {{\color{white}$\bullet$} where $e=(v,w)$, $f_v=(v,0)$, $f_w=(w,1)$,};
		\node[right] at (3,-.5) {{\color{white}$\bullet$} and $e_i=(0,1)$ for $i=1,\ldots,n$};

		\end{tikzpicture}
	\]
Define an edge weighting $\mu'$ on $\Gamma_2'$ as follows:
	\[
		\mu'(e):=\mu(e)\qquad  \mu'(f_v):= \frac{\vphi(p_v) - \alpha}{\vphi(p_v)} \qquad \mu'(f_w) := \frac{\vphi(p_w) - \beta}{\vphi(p_w)}\qquad \mu'(e_1) := \frac{\vphi(p_w) - \beta}{\vphi(p_v) -\alpha},
	\]
and choose $n\geq 2$ and $\mu'(e_2),\ldots, \mu'(e_n)$ distinct (and distinct from $\mu'(e_1)$) such that
	\[
		\< \frac{\mu'(e_i)}{\mu'(e_j)}\colon i,j=1,\ldots, n\> = H'. 
	\]
We then define $\Gamma_1'=(V_1',E_1')$ to be the subgraph with $V_1'=V_2'$ and $E_1'=E_2'\setminus\{e,e^{\op}\}$:
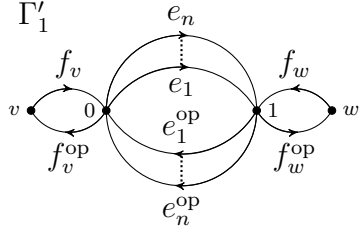
	\[
		\begin{tikzpicture}
		\node[left] at (-1.6,1.25) {$\Gamma_1'$};

		\node[left] at (-2,0) {\scriptsize$v$}; 
		\draw[fill=black] (-2,0) circle (0.05);
		
		\node[above] at (-1.5,0.288675) {$f_v$}; 
		\draw (-2,0) arc (150:30:0.57735);
		\draw[->] (-2,0) arc (150:85:0.57735);
		\node[below] at (-1.5,-0.288675) {$f_v^{\op}$};
		\draw (-2,0) arc (-150:-30:0.57735);
		\draw[->] (-1,0) arc (-30:-95:0.57735);
		
		\node[left] at (-1,0) {\scriptsize$0$}; 
		\draw[fill=black] (-1,0) circle (0.05);
		
		\node[right] at (1,0) {\scriptsize$1$}; 
		\draw[fill=black] (1,0) circle (0.05);

		\node[above] at (1.5,0.288675) {$f_w$}; 
		\draw (2,0) arc (30:150:0.57735);
		\draw[->] (2,0) arc (30:95:0.57735);
		\node[below] at (1.5,-0.288675) {$f_w^{\op}$};
		\draw (2,0) arc (-30:-150:0.57735);
		\draw[->] (1,0) arc (-150:-85:0.57735);

		\node[right] at (2,0) {\scriptsize$w$}; 
		\draw[fill=black] (2,0) circle (0.05);
		
		\node[below] at (0, 0.57735) {$e_1$}; 
		\draw (-1,0) arc (150:30:1.1547);
		\draw[->] (-1,0) arc (150:85:1.1547);
		
		\draw[densely dotted,thick] (0,0.6) -- (0,1);
		
		\node[above] at (0,1) {$e_n$}; 
		\draw (-1,0) arc (180:0:1);
		\draw[->] (-1,0) arc (180:85:1);
		
		\node[above] at (0,-0.57735) {$e_1^{\op}$}; 
		\draw (1,0) arc (-30:-150:1.1547);
		\draw[->] (1,0) arc (-30:-95:1.1547);
		
		\draw[densely dotted, thick] (0,-0.6) -- (0,-1);
		
		\node[below] at (0,-1) {$e_n^{\op}$}; 
		\draw (-1,0) arc (-180:0:1);
		\draw[->] (1,0) arc (0:-95:1);

		\end{tikzpicture}
	\]
We will define a positive linear functional $\vphi'$ on $\cM(\Gamma_2',\mu')$ as in Section~\ref{sec:adding_weighting_to_edges} after carefully choosing $(\Gamma_2')_{\Tr}$; namely, we want to ensure $e\not\in (\Gamma_2')_{\Tr}$. We must consider two cases.\\

\noindent\underline{\textbf{Case 1:}} Assume $\lambda:=\mu'(f_v)\mu'(e_1)\mu'(f_w^{\op})\mu'(e^{\op})\neq 1$. In this case, we choose $(\Gamma_2')_{\Tr}$ to be maximal among subgraphs $(\Xi,\mu')$ of $(\Gamma_2',\mu')$ satsifying $H(\Xi,\mu')=\{1\}$ and which contain $f_v$, $f_w$, and $e_1$. In fact, it is easy to see that this mean $(\Gamma_2')_{\Tr}$ is the graph formed by precisely these three edges and their opposites. We then let $\vphi'$ be the induced positive linear functional on $\cM(\Gamma_2',\mu')$.

Observe that we have
	\[
		\vphi'(p_w) = \mu'(f_v)\mu'(e_1)\mu'(f_w^{\op})\vphi'(p_v) = \frac{\vphi(p_w)}{\vphi(p_v)} \vphi'(p_v).
	\]
Thus, rescaling if necessary, we may assume $\vphi'(p_v)=\vphi(p_v)$ and $\vphi'(p_w)=\vphi(p_w)$. By Lemma~\ref{lem:switcheroo}, Lemma~\ref{lem:adding_edge_and_vertex} (applied twice), and Lemma~\ref{lem:antman1} we have by our choice of $\mu(e_2),\ldots,\mu(e_n)$ that
	\[
		(P \cM(\Gamma_1',\mu') P ,(\vphi')^{P})\cong (P\cM(\Gamma_1,\mu) P,\vphi^{P}),
	\]
where $P:=p_v+p_w$. Consequently
	\[
		(P \cM(\Gamma_2,\mu)P,\vphi^P) \cong \left([P \cM(\Gamma_1',\mu') P]\Asterisk_D W^*(Y_e, D),(\vphi')^P\right)\cong (P\cM(\Gamma_2',\mu')P,(\vphi')^P),
	\]
where $D:=\overset{p_v}{\bbC}\oplus \overset{p_w}{\bbC}$.

Now, consider the subgraph $\Gamma_0'$ of $\Gamma_2'$ consisting of the edges $f_v,f_w, e_1, e$ and their opposites:
	\[
		\begin{tikzpicture}
		\node[left] at (-1.6,1.7) {$\Gamma_0'$};

		\draw (-2,0) arc (185:-5:2); 
		\draw[->] (-2,0) arc (185:87.5:2);
		\node[above] at (0,2.15) {$e$};

		\draw (-2,0) arc (170:10:2.031); 
		\draw[->] (2,0) arc (10:92.5:2.031);
		\node[above] at (0,1.6) {$e^{\op}$};

		\node[left] at (-2,0) {\scriptsize$v$}; 
		\draw[fill=black] (-2,0) circle (0.05);
		
		\node[above] at (-1.5,0.288675) {$f_v$}; 
		\draw (-2,0) arc (150:30:0.57735);
		\draw[->] (-2,0) arc (150:85:0.57735);
		\node[below] at (-1.5,-0.288675) {$f_v^{\op}$};
		\draw (-2,0) arc (-150:-30:0.57735);
		\draw[->] (-1,0) arc (-30:-95:0.57735);
		
		\node[left] at (-1,0) {\scriptsize$0$}; 
		\draw[fill=black] (-1,0) circle (0.05);
		
		\node[right] at (1,0) {\scriptsize$1$}; 
		\draw[fill=black] (1,0) circle (0.05);

		\node[above] at (1.5,0.288675) {$f_w$}; 
		\draw (2,0) arc (30:150:0.57735);
		\draw[->] (2,0) arc (30:95:0.57735);
		\node[below] at (1.5,-0.288675) {$f_w^{\op}$};
		\draw (2,0) arc (-30:-150:0.57735);
		\draw[->] (1,0) arc (-150:-85:0.57735);

		\node[right] at (2,0) {\scriptsize$w$}; 
		\draw[fill=black] (2,0) circle (0.05);
		
		\node[below] at (0, 0.57735) {$e_1$}; 
		\draw (-1,0) arc (150:30:1.1547);
		\draw[->] (-1,0) arc (150:85:1.1547);

		\node[above] at (0,-0.57735) {$e_1^{\op}$}; 
		\draw (1,0) arc (-30:-150:1.1547);
		\draw[->] (1,0) arc (-30:-95:1.1547);

		\end{tikzpicture}
	\]
It follows from Lemma~\ref{lem:base_case3}, Remark~\ref{rem:loop}, and Lemma~\ref{lem:antman1} that
	\[
		\left(P\cM(\Gamma_0', \mu')P,(\vphi')^P\right) \cong (T_{\lambda}, \vphi_{\lambda}) \oplus \overset{r_{w}}{\underset{\vphi(p_{w})[1 - \mu(e^{\op}) - \mu'(f_w)]}{\C}}.
	\]
To finish, we add the edges $e_{2}, \dots, e_{n}$ (and their opposites).  Set $Q:= p_{0} + p_{1}$. Lemmas \ref{lem:antman1} and \ref{lem:antman2} (applied to the above compression), together with the proof of Lemma~\ref{lem:switcheroo} yield that
	\[
		(Q\cM(\Gamma_2', \mu')Q, (\vphi')^Q) \cong (T_H,\vphi_H)\oplus \overset{r_0}{\bbC}\oplus \overset{r_1}{\bbC},
	\]
where $r_0\leq p_{0}$ and $r_1\leq p_{1}$, and each exists if and only if the respective following numbers are positive:
	\begin{align*}
		\vphi'(r_0)&=\vphi'(p_0)\left[ 1- \mu'(f_v^{\op}) - \sum_{j=1}^{n}\mu'(e_{j})\right]\\
		\vphi'(r_1)&=\vphi'(p_1)\left[1- \mu'(f_w^{op})- \sum_{j=1}^{n}\mu'(e_{j}^{\op}) \right].
	\end{align*}
Then Lemmas \ref{lem:antman1} and \ref{lem:antman2} and Remark~\ref{rem:atomic} imply that 
	\[
		(P\cM(\Gamma_2,\mu)P,\vphi^P)\cong (P\cM(\Gamma_2', \mu')P, (\vphi')^P) \cong (T_H, \vphi_H) \oplus \overset{r_{w}}{\underset{\vphi(p_{w})[1 - \mu(e^{\op}) - \mu'(f_w)]}{\C}}
	\]
Using Lemma~\ref{lem:antman2} and Remark~\ref{rem:atomic} completes this case.\\

\noindent\underline{\textbf{Case 2:}} Assume $\mu'(f_{v})\mu'(e_{1})\mu'(f_{w}^{\op})\mu'(e^{\op}) = 1$. In this case we choose $(\Gamma_2')_{\Tr}$ to be the graph consisting of $f_v, f_w, e_2$ and their opposites (recall $\mu'(e_2)\neq \mu'(e_1)$), and let $\vphi''$ be the induced positive linear functional on $\cM(\Gamma_2',\mu')$. Proceeding as in Case 1 we obtain:
	\[
		(P\cM(\Gamma_2',\mu')P, (\vphi'')^P) \cong  (T_H, \vphi_H) \oplus \overset{r_{w}}{\underset{\vphi''(p_{w})[1 - \mu(e^{\op}) - \mu'(f_w)]}{\C}}
	\]
Letting $\vphi'$ be as in Case 1, Proposition~\ref{prop:changestate} implies
	\begin{align*}
		(P\cM(\Gamma_2,\mu)P,\vphi^P)&\cong (P\cM(\Gamma_2',\mu')P,(\vphi')^P)\\
			&\cong (P\cM(\Gamma_2',\mu')P,(\vphi'')^P) \cong (T_H, \vphi_H) \oplus \overset{r_{w}}{\underset{\vphi''(p_{w})[1 - \mu(e^{\op}) - \mu'(f_w)]}{\C}}
	\end{align*}
Finally, appealing to Lemma~\ref{lem:antman2} and Remark~\ref{rem:atomic} completes the proof.
\end{proof}

\subsection{The big reveal}

\begin{thm}
Suppose that $H:=\< \mu(e_1)\cdots \mu(e_n)\colon e_1\cdots e_n \in \Lambda_\Gamma\> <\bbR^+$ is non-trivial. Then exists a faithful, positive linear functional $\vphi$ on $\cM(\Gamma,\mu)$ such that
	\[
		(\cM(\Gamma, \mu), \vphi) \cong (T_H,\vphi_H)\oplus \bigoplus_{v\in V} \overset{r_v}{\bbC},
	\]
where $r_v\leq p_v$ is non-zero if and only if $\sum_{\substack{e\in E\\ s(e)=v}} \mu(e)<1$, in which case
	\[
		\vphi(r_v)=\vphi(p_v)\left[ 1- \sum_{\substack{e\in E\\ s(e)=v}} \mu(e)\right].
	\]
In particular, if
	\[
		\sum_{\substack{e\in E\\ s(e)=v}} \mu(e) \geq 1
	\]
for all $v\in V$, then $(\cM(\Gamma, \mu), \vphi)\cong (T_H,\vphi_H)$.
\end{thm}
\begin{proof}
Let $\sigma_0=e_1\cdots e_n$, $\Gamma_{\Tr}$, $\Gamma_0$, and $\vphi$ be as in the beginning of this Section. Observe that if $n=1$, then $(\Gamma_0,\mu,\vphi)$ are as in Lemma~\ref{lem:base_case0}; if $n=2$, then $(\Gamma_0,\mu,\vphi)$ are as in Lemma~\ref{lem:base_case1}; and if $n\geq 3$, then $(\Gamma_0,\mu,\vphi)$ are as in Lemma~\ref{lem:base_case3}. In each case we obtain
	\[
		(\cM(\Gamma_0,\mu), \vphi) \cong (T_\lambda, \vphi_\lambda) \oplus \bigoplus_{v\in V_0} \overset{r_v}{\bbC},
	\]
where $\lambda=\frac{1}{\mu(e_1)\cdots\mu(e_n)}$ and $r_v\leq p_v$ satisfy
	\[
		\vphi(r_v) = \vphi(p_v) \left[ 1- \sum_{\substack{e\in E\\ s(e) = v}} \mu(e)\right]
	\]
(and are zero if the above quantity is non-positive).

We first succesively add the rest of the edges in $\Gamma_{\Tr}$, and repeatedly apply Lemmas~\ref{lem:adding_edge_and_vertex}, \ref{lem:adding_loop_to_existing_vertex}, and \ref{lem:adding_edge_to_existing_vertices}. Note that the hypothesis $\vphi(p_w)=\mu(e)\vphi(p_v)$ in Lemma~\ref{lem:adding_edge_and_vertex} is always satisfied because we are expanding along $\Gamma_{\Tr}$ which determined $\vphi$. This resulting graph will contain all the vertices $V$ by the maximality of $\Gamma_{\Tr}$, and so we simply add the rest of the edges one-by-one by applying Lemmas~\ref{lem:adding_loop_to_existing_vertex} and \ref{lem:adding_edge_to_existing_vertices}.  At each step of this construction we obtain the predicted isomorphism, and so after finitely many steps we obtain the desired isomorphism for $(\cM(\Gamma,\mu),\vphi)$.
\end{proof}


\section*{Appendix}


\subsection*{Some infinite index subfactors}

Let $\TL_{\bullet}$ denote the Temperley Lieb planar algebra with parameter $\delta \in [2, \infty)$.  In \cite{MR2732052}, one considers $\Gr(\TL) = \oplus_{n=1}^{\infty} \TL_{n}$.  
Recall from \cite{MR2732052} that $\Gr(\TL)$ has the structure of a graded algebra with graded multiplication $\wedge$ and normalized Voiculescu trace $\tau$:
	\[
		\tau(x \wedge y)=\tau\left(
							\begin{tikzpicture}[baseline = .1cm]
							\draw (0, 0)--(0, .6);
							\draw (.8, 0)--(.8, .6);
							\nbox{unshaded}{(0,0)}{.3}{0}{0}{$x$}
							\nbox{unshaded}{(.8,0)}{.3}{0}{0}{$y$}
							\end{tikzpicture}
						\right)=
							\begin{tikzpicture}[baseline=.3cm]
							\draw (.4,0)--(.4,.6);
							\draw (-.4,0)--(-.4,.6);
							\nbox{unshaded}{(-.4,0)}{.3}{0}{0}{$x$}
							\nbox{unshaded}{(.4,0)}{.3}{0}{0}{$y$}
							\nbox{unshaded}{(0,.8)}{.3}{.4}{.4}{$\Sigma \, \TL$}	
							\end{tikzpicture}.
	\]
Moreover, this graded algebra acts on itself by left and right multiplication, which is bounded with respect to the $\|\cdot \|_2$-norm induced by this trace. Let $\cN = (\Gr(\TL), \tau)''$.  It was shown in \cite{MR3110503} that $\cN \cong L(\F_{\infty})$.

Jones and Penneys in \cite[Section 6.4]{JP17} studied a directed graph $\Gamma=(V, E)$ with edge-weighting $\mu$ as in Section \ref{sec:adding_weighting_to_edges} with the additional requirement that $(\Gamma,\mu)$ is \emph{balanced}:
	\[
		\displaystyle\sum_{e\colon s(e) = v} \mu(e) = \delta\qquad \forall v \in V. 
	\]
Fixing $v \in V$, they studied the loop algebra $\textbf{A}_{v}$ which is spanned by formal linear combinations of loops based at $v$.  It comes equipped with the following $*$-algebra structure and state $\phi$:
	\begin{itemize}

	\item $(e_{1}\cdots e_{n}) \cdot (f_{1}\cdots f_{m}) = e_{1}\cdots e_{n}f_{1}\cdots f_{m}$

	\item $(e_{1}\cdots e_{n})^{*} = \sqrt{\mu(e_{1})\cdots\mu(e_{n})}e_{n}\cdots e_{1}$

	\item With $NC_{2}([n])$ the set of non-crossing pair partitions on $\{1, \cdots, n\}$,
		\[
			\phi(e_{1}\cdots \e_{n}) =
				\begin{tikzpicture}[baseline=.3cm]
				\draw (0,0)--(0,.6);
				\nbox{unshaded}{(0,0)}{.3}{.5}{.5}{$e_{1}\cdots e_{n}$}
				\nbox{unshaded}{(0,.8)}{.3}{.4}{.4}{$\Sigma \, \TL$}		
				\end{tikzpicture}
 			:= \sum_{\pi \in NC_{2}([n])} \prod_{ \substack{ i \sim_{\pi} j \\ i < j}} \sqrt{\mu(e_{i})}\delta_{e_{i} = e_{j}^{op}}
		\]
\end{itemize}
The following proposition is a standard combinatorial argument involving the elements $Y_{e} \in (\cM(\Gamma, \mu), \vphi)$.  
	
\begin{prop*} The map $p_{v}\C\langle Y_{e} : e \in E\rangle p_{v} \rightarrow \textbf{A}_{v}$ given by $Y_{e_{1}}\cdots Y_{e_{n}} \mapsto e_{1}\cdots e_{n}$ for every loop based at $v$ is a $*$-algebra isomorphism statisfying $\phi(e_{1}\cdots e_{n}) = \vphi(Y_{e_{1}}\cdots Y_{e_{n}})$.
\end{prop*}

This proposition allows us to conclude that $(\cM, \phi) := (\textbf{A}_{v}, \phi)'' \cong (T_{H}, \vphi_{H})$ for $H = H(\Gamma, \mu)$. (Note that $\delta\geq 2$, and so requiring that $\mu$ is balanced ensures $\cM(\Gamma,\mu)$ is a factor.)  Section 6.4 of \cite{JP17} studied an inclusion of $i : \cN \hookrightarrow \cM$ as follows.  Identifying each diagram $x\in \Gr(\TL)$ with a partition $\pi \in NC_{2}([2n])$ in the natural way, one has 
	\[
		i(x) = \sum_{\substack{e_{1}\cdots e_{2n}: \\ i \sim_{\pi} j, \, i < j \Rightarrow e_{i} = e_{j}^{\op}}} \left(\prod_{ \substack{i \sim_{\pi} j \\ i < j}} \mu(e_{i}) \right)e_{1}\cdots e_{2n}
	\]
This inclusion preserves the tracial state $\tau$ on $\cN$, and so $\cN \hookrightarrow \cM^{\phi}$.  

Jones and Penneys showed that the subfactor $\cN \subset \cM$ is of infinite index, irreducible ($\cN' \cap \cM = \C$),  discrete ($L^{2}(\cM)$ decomposes as a direct sum of irreducible, finite index $\cN-\cN$  bimodules), and that the subfactor $\cN \subset \cM$ is determined by $(\Gamma, \mu)$.  Our work shows that the inclusion $\cN \subset \cM$ can be realized as an inclusion $L(\F_{\infty}) \subset (T_{H}, \vphi_{H})^{\vphi_{H}} \subset (T_{H}, \vphi_{H})$.


\subsection*{Summary of Notation}
{\footnotesize
\begin{center}	
	\begin{tabular}{|c| p{3.8in} |c|}
	\hline
	Notation & Description & Page\\
	\hline
	$\overset{p_1,\ldots, p_n}{\underset{t_1,\ldots,t_n}{M_n(\bbC)}}$, $\overset{p}{\underset{t}{(A,\phi)}}$ & implicit states or weighting on states & \pageref{notation:implicit_states}\\
	\hline
	$\psi_\lambda$ & a state on $\cB(\cH)$ or $M_n(\bbC)$ determined by a set of matrix units & \pageref{notation:psi}\\
	\hline
	$\phi^p$ & compression of a positive linear functional & \pageref{notation:state_compression}\\
	\hline
	$(T_\lambda,\vphi_\lambda)$ & the free Araki--Woods factor generated by two variables and its free quasi-free state & \pageref{notation:fAWf}\\
	\hline
	$y_\lambda$ & a generalized circular element & \pageref{notation:gce}\\
	\hline
	$(T_H,\vphi_H)$ & a free product of free Araki--Woods factors & \pageref{notation:fAWfH}\\
	\hline
	 $z(p\colon M)$ & the central support of a projection $p$ in a von Neumann algebra $M$ & \pageref{notation:central_support}\\
	\hline
	$\Gamma=(V,E)$ & a graph with vertices $V$ and edges $E$ &\pageref{notation:graph}\\
	\hline
	$s(e)$, $t(e)$, $e^{\op}$ & the source, target, and opposite of an edge $e$ & \pageref{notation:edge}\\
	\hline
	$\mu$ & an edge weighting & \pageref{notation:mu}\\
	\hline
	$Y_e$, $u_e$ & an edge operator and its polar part & \pageref{notation:Y_e}\\
	\hline
	$S(\Gamma,\mu)$ & the $C^*$-algebra associated to a graph $\Gamma$ and an edge weighting $\mu$ & \pageref{notation:S}\\
	\hline
	$\mc{M}(\Gamma,\mu)$ & the von Nemann algebra associated to a graph $\Gamma$ and an edge weighting $\mu$ & \pageref{notation:M}\\
	\hline
	$\Pi_\Gamma$ & the space of paths in a graph $\Gamma$ & \pageref{notation:paths}\\
	\hline
	$\Lambda_\Gamma$ & the space of loops in a graph $\Gamma$ & \pageref{notation:loops}\\
	\hline
	$H(\Gamma,\mu)$ & the subgroup of $\R^+$ generated by $\mu(e_1)\cdots \mu(e_n)$ for loops $e_1\cdots e_n$ in a graph $\Gamma$ with edge weighting $\mu$ & \pageref{notation:H}\\
	\hline
	$\Gamma_{\Tr}$ & a maximal subgraph of $(\Gamma,\mu)$ subject to the condition $H(\Gamma_{\Tr},\mu)$ is trivial & \pageref{notation:tracial_subgraph}\\
	\hline
	$\vphi$ & a faithful normal positive linear functional on $\cM(\Gamma,\mu)$ induced by $\Gamma_{\Tr}$ & \pageref{notation:phi}\\
	\hline
	\end{tabular}
\end{center}
}

\bibliographystyle{amsalpha}
\bibliography{bibliography}

\end{document}